\documentclass[a4paper,notitlepage,twoside,reqno,10pt]{amsart}

\usepackage{bbm,pifont,latexsym,dcolumn,indentfirst}
\usepackage[hypertexnames=false,bookmarksopen=true,linktocpage=true,pdfstartview={XYZ null null 1.25}]{hyperref}
\usepackage{amsthm,amsmath,amssymb,amscd,amsfonts,mathrsfs}
\usepackage{color,graphicx,xcolor,graphics,subfigure,caption2}
\usepackage{titletoc}
\usepackage{enumitem}
\usepackage{tikz,tikz-cd}

\newtheorem{thm}{Theorem}[section]
\newtheorem{lem}[thm]{Lemma}
\newtheorem{cor}[thm]{Corollary}
\newtheorem{prop}[thm]{Proposition}

\newtheorem{rmk}[thm]{Remark}

\newtheorem*{mainthm}{Main Theorem}

\theoremstyle{definition}
\newtheorem*{defi}{Definition}

\newcommand{\D}{\mathbb{D}}
\newcommand{\T}{\mathbb{T}}

\newcommand{\N}{\mathbb{N}}

\newcommand{\C}{\mathbb{C}}
\newcommand{\EC}{\widehat{\mathbb{C}}}

\newcommand{\Crit}{\textup{Crit}}

\newcommand{\diam}{\textup{diam}}

\newcommand{\ii}{\textup{i}}

\makeatletter\@addtoreset{equation}{section}\makeatother

\begin{document}
\author[Y. Fu]{Yuming Fu}
\address{School of Mathematical Sciences, Shenzhen University, Shenzhen 518060, P. R. China}
\email{yumingfuxy@szu.edu.cn}
\author[J. Hu]{Jun Hu}
\address{Department of Mathematics,
Brooklyn College of CUNY,
Brooklyn, NY 11210, USA and
Ph.D. Program in Mathematics,
Graduate Center of CUNY,
365 Fifth Avenue, New York, NY 10016, USA}
\email{junhu@brooklyn.cuny.edu or JHu1@gc.cuny.edu}
\author[O. Muzician]{Oleg Muzician}
\address{Department of Mathematics,
BMCC of CUNY, 199 Chambers Street, New York, NY 10007}
\email{OMuzician@bmcc.cuny.edu}

\title[Cubic polynomials with a $2$-cycle of Siegel disks]{Cubic polynomials with a $2$-cycle of Siegel disks}

\begin{abstract}
Under conjugation by affine transformations, the dynamical moduli
space of cubic polynomials $f$ with a $2$-cycle of Siegel disks is parameterized by a three-punctured complex plane as a degree-$2$ cover. Assuming
the rotation number of $f^2$ on the Siegel disk is of bounded
type, we show that on the three-punctured complex plane, the locus of the cubic polynomials with both finite critical
points on the boundaries of the Siegel disks on the $2$-cycle is comprised of two arcs, corresponding to the
cases with two critical points on the boundary of the same Siegel
disk, and a Jordan curve, corresponding to the cases with two
critical points on the boundaries of different Siegel disks.
\end{abstract}



\date{\today}



\maketitle

\tableofcontents

\section{Introduction}\label{introduction}

\medskip
Julia sets of quadratic rational maps are either connected or
Cantor sets. Although cubic rational maps are just one degree
higher than quadratic ones, their dynamics are much richer. They
may have Herman rings in their Fatou sets. Even under no presence
of Herman rings, the Julia set of such a cubic rational map may be
neither connected nor a Cantor set. It seems quite challenging to
classify the Julia sets of all cubic rational maps at the
same time. So one may first explore the dynamics of cubic rational
maps under some constrains. For example, the dynamics of cubic
rational maps with all critical points escaping to an attracting
fixed point are explored in \cite{HuEtkin20} and the Julia set
dichotomy (either connected or a Cantor set) is obtained for such
a cubic rational map in \cite{HuEtkin22}. Lately, the Julia sets
of cubic rational maps with all critical points escaping to a
parabolic fixed point have been studied and classified in
\cite{HuEtkin23a}, and the Julia sets of cubic rational maps with
two attracting fixed points are investigated and classified in
\cite{HuEtkin23b}.

Under conjugacy by affine maps, the dynamical moduli space of the cubic polynomials with Siegel disks of period $1$ and a constant (irrational) rotation number is represented by a one-complex-parameter family of cubic polynomials. Zakeri proved in \cite{Zak99} that on this parameter space the locus of the maps with both critical points on the boundaries of the corresponding Siegel disks is a Jordan curve if the rotation number is of bounded type. Recently, it is proved in \cite{FYZ20} that in the dynamical moduli space of the quadratic rational maps with a Siegel disk of period $2$ and of bounded type,  the locus of the maps with two critical points on the boundaries of the Siegel disks on the $2$-cycle is also a Jordan curve. Motivated by these works, we study in this paper the dynamical moduli space of the cubic polynomials with a cycle of Siegel disks of period $2$ and of bounded type. Through conjugation by a M\"obius transformation, such cubic polynomials are expressed by a one-parameter family of cubic polynomials with the centers of the Siegel disks arranged at $0$ and $1$. More precisely, it is of the following form (see Proposition \ref{the desired family}):
\begin{equation}\label{f-alpha-family}
f_{\alpha}(z)=(\frac{\lambda}{\alpha}+\alpha+2)z^3-(\frac{\lambda}{\alpha}+2\alpha+3)z^2+\alpha
z+1,
\end{equation} where $\lambda=e^{2\pi\ii\theta}$, $\theta $
is an irrational number of bounded type, $\alpha\in \C$, $\alpha\neq 0$, and
$\frac{\lambda}{\alpha}+\alpha+2\neq 0$. An irrational number $\theta $ is said to
be of {\em bounded type} if its continued fraction expansion
$\theta=[a_0;a_1,a_2,\cdots,a_n,\cdots]$ satisfies $\sup_n
\{a_n\}<+\infty$. We call
\begin{equation}\label{equ:family}
\Sigma_\theta=\left\{\alpha\in\C: \alpha\neq 0 \text{ and
}\frac{\lambda}{\alpha}+\alpha+2\neq 0, \text{ where
}\lambda=e^{2\pi\ii\theta}\right\}
\end{equation}
the parameter space for $f_{\alpha }$.

Let $\alpha\in\Sigma_\theta$. Denote the two Siegel disks on the
$2$-cycle under $f_{\alpha}$ by
$\Delta_{\alpha}^0$ and $\Delta_{\alpha}^1$, containing $0$ and $1$
respectively. According to \cite{Zha11}, $\partial\Delta_{\alpha}^0$ and
$\partial\Delta_{\alpha}^1$ are quasi circles, and
$\partial\Delta_{\alpha}^0\cup\partial\Delta_{\alpha}^1$ contains
at least one critical point. The Julia set $J(f_\alpha)$ of $f_{\alpha}$ is connected if and only if the other (finite) critical point doesn't escape to $\infty $. We call
\begin{equation}\label{connectedness locus}
    \mathcal{M}_\theta=\{\alpha\in\Sigma_\theta: \text{ the orbits of the finite critical points are bounded}\}
\end{equation}
the connectedness locus of $f_\alpha$. Then $J(f_\alpha)$ is disconnected if and only if $\alpha \in \Sigma_\theta\setminus \mathcal{M}_\theta$. 

On Figure \ref{Fig_parameter}, four white
regions represent the components of $\Sigma_\theta\setminus \mathcal{M}_\theta$, which are punctured at $0$, $\infty $ and the roots of
$\frac{\lambda}{\alpha}+\alpha+2=0$ respectively. We view these four punctures as the centers of the four components. See Figure \ref{Fig_Jul-(-2-0p2i)} as an example of $J(f_\alpha)$ when $\alpha \in \Sigma_\theta\setminus \mathcal{M}_\theta$.

When $\alpha $ belongs to the interior of
$\mathcal{M}_{\theta}$, there is a (finite) critical point not on $\partial\Delta_{\alpha}^0\cup\partial\Delta_{\alpha}^1$, whose orbit may have three possible patterns: converging to an
attracting cycle, landing on the interior of a Siegel disk after finitely many iterates, or being neither of these two situations, which are respectively called a {\em hyperbolic-like, capture, or queer} case in \cite{Zak99}. Referring to Figure \ref{Fig_parameter}, the interiors of the black regions represent hyperbolic-like components of $\Sigma_{\theta}$, which show us infinitely many copies (homeomorphic images) of the hyperbolic components in the quadratic family $Q_c(z)=z^2+c$; the interiors of the yellow regions represent capture cases, which show us infinitely many Jordan domains. It is conjectured in \cite{Zak99} that queer components do not exist; on the other hand,
if such a queer component exists, then the Julia set of any cubic polynomial on this
component has positive Lebesgue measure and admits an invariant line field.
See Figure \ref{hyperbolic-like case} for an example of $J(f_\alpha)$
in a hyperbolic-like case and Figure \ref{Fig_Jul-(1p6066+p41583i)} in a capture case. 

\begin{figure}[!htpb]
 \setlength{\unitlength}{1mm}
  \centering
  \includegraphics[width=0.72\textwidth]{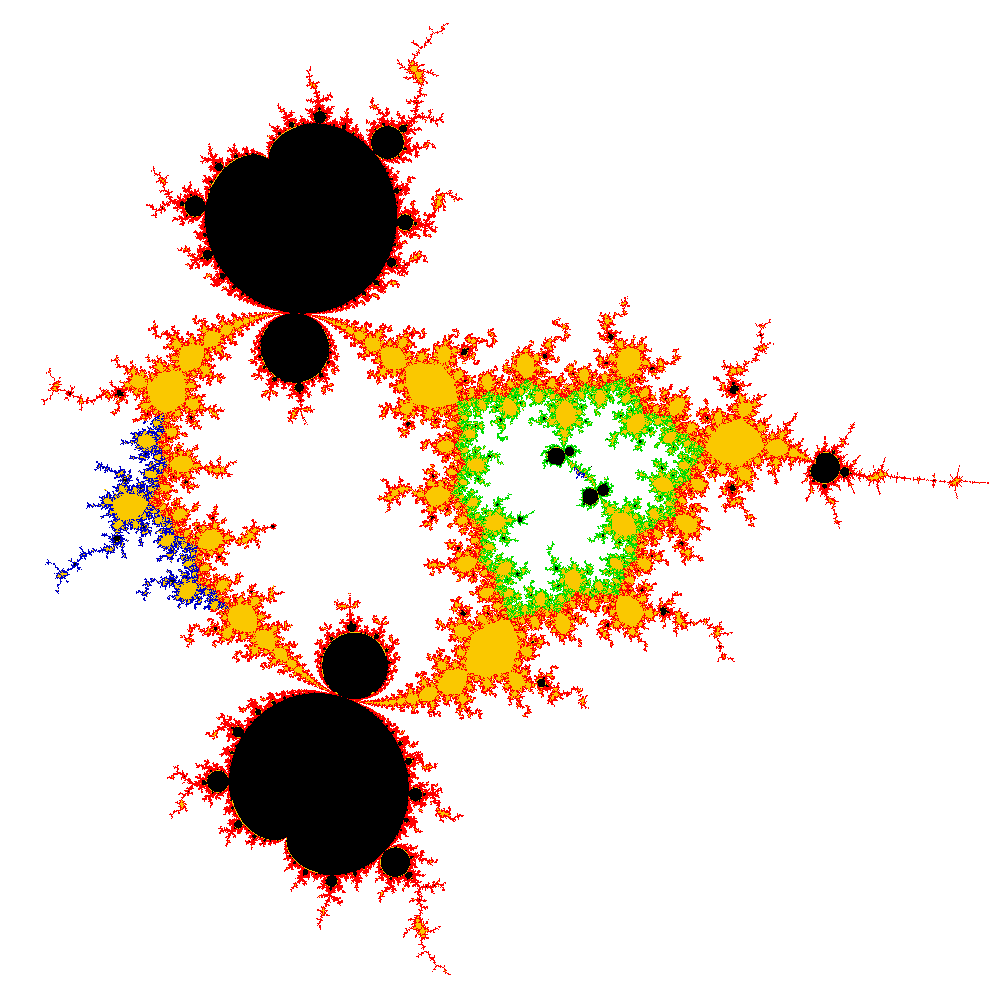}
  \caption{The parameter plane of $\Sigma_\theta$ with $\theta=(\sqrt{5}-1)/2$.}
  \label{Fig_parameter}
\end{figure}

In this paper, we study the following three
subsets of the connectedness locus $\mathcal{M}_{\theta}$.
\begin{defi}\label{def of Gamma curves} Let $\theta $ be an
irrational number of bounded type. We define
$$\Gamma_\theta=\left\{\alpha \in \Sigma_\theta: \text{ each
of }\partial\Delta_{\alpha}^0 \text{  and
}\partial\Delta_{\alpha}^1 \text{ contains a critical point
of } f_\alpha \right\},$$
$$\Gamma_\theta^0=\left\{\alpha \in \Sigma_\theta:
\partial\Delta_{\alpha}^0 \text{ contains two critical
points counted by multiplicity}\right\}, \text{ and}$$
$$\Gamma_\theta^1=\left\{\alpha \in \Sigma_\theta: \partial\Delta_{\alpha}^1
\text{ contains two critical points counted by multiplicity}\right\}.$$
\end{defi}

By a {\em Jordan arc} on a plane we mean an image of a closed and
bounded interval $[a, b]$ under an injective continuous map into
the plane.

\begin{mainthm}\label{maintheorem} Let $\theta $ be an irrational
number of bounded type. Then
\begin{enumerate}
\item $\partial\Delta_{\alpha}^0$ and $\partial\Delta_{\alpha}^1$
depend continuously on $\alpha\in\Sigma_\theta$; \item
$\Gamma_\theta$ is a Jordan curve on $\Sigma_\theta$ separating
$0$ from $\infty $; \item $\Gamma_\theta^0$ is a Jordan arc in the
component of $\Sigma_\theta\setminus\Gamma_\theta$ containing $0$;
\item $\Gamma_\theta^1$ is a Jordan arc in the component of
$\Sigma_\theta\setminus\Gamma_\theta$ containing $1$.
\end{enumerate}
\end{mainthm}
\begin{rmk} Note that if $\Gamma_\theta$, $\Gamma_\theta^0$ or $\Gamma_\theta^1$ has an interior component, then such a component is a queer component. Therefore, 
our Main Theorem implies that there are no queer components in  $\Gamma_\theta$, $\Gamma_\theta^0$ or $\Gamma_\theta^1$, which provides a piece of information to support the conjecture on the nonexistence of queer components in the interior of the connectedness locus $\mathcal{M}_\theta$.
\end{rmk}

\section{The parameter space $\Sigma_\theta$ and Dynamics of $f_{\alpha}$ when $\alpha\in\Sigma_{\theta}$}
In this section, we do a preliminary investigation of the parameter space $\Sigma_\theta$, and show how $\partial\Delta_\alpha^0\cup\partial\Delta_\alpha^1$ changes as $\alpha$ varies on $\Sigma_{\theta}$.
\subsection{The form of $f_\alpha$}
In this first subsection, we show how to derive the expression of $f_{\alpha }$. Then we prove a symmetry existing in the parameter space $\Sigma_\theta$. 

\begin{prop}\label{the desired family}
Let $f$ be a cubic polynomial having a $2$-cycle of Siegel disks with rotation number $\theta$. Then $f$ is conformally conjugate to $f_\alpha$ in the form of (\ref{f-alpha-family}) for some $\alpha\in\C\setminus\{0,-1\pm \sqrt{1-\lambda}\}$. In particular, $f_\alpha$ has a $2$-cycle of Siegel disks, centered at $0$ and $1$ respectively, with multiplier equal to $\lambda =e^{2\pi\ii\theta}$, 
which we denote by $\Delta_\alpha^0$ and $\Delta_\alpha^1$.
\end{prop}
\begin{proof}
Through conjugation by a M\"obius transformation, we may assume that the centers of the Siegel disks on 
a $2$-cycle are arranged at $0$ and $1$. Then $f$ can be written as
\begin{equation*}\label{polynomial expression}
f(z)=az^3+bz^2+cz+1
\end{equation*}
with
\begin{equation*}\label{cycle condition}
a+b+c+1=0
\end{equation*}
and
\begin{equation}\label{multiplier condition}
f'(0)f'(1)=c(3a+2b+c)=\lambda =e^{2\pi\ii\theta}.
\end{equation}
From \eqref{multiplier condition}, we know $c\neq 0$. So we let
$c=\alpha\in\C\setminus\{0\}$. Then
$a=\frac{\lambda}{\alpha}+\alpha+2\not=0$ and
$b=\frac{\lambda}{\alpha}+2\alpha+3$. Thus, $f$ is of the form of $f_\alpha$ given by
(\ref{f-alpha-family}) with
$\alpha\in\Sigma_{\theta}=\C\setminus\{0,-1\pm\sqrt{1-\lambda}\}$.
It follows that $f_\alpha$ has a $2$-cycle of Siegel disks centered at $0$ and $1$ respectively and with multiplier equal to $\lambda =e^{2\pi\ii\theta}$.
\end{proof}
From direct calculations, we can show that $f_\alpha$ has the
following properties.
\begin{lem}\label{lem:sym}
Let $\tau(z)=-z+1$. Then
\begin{enumerate}
\item $\tau\circ f_\alpha\circ \tau^{-1}=f_{\tilde\alpha}$, where
$\alpha\in\Sigma_\theta$ and $\tilde\alpha=\lambda/\alpha$; \item
$\tau$ maps any critical point of $f_\alpha $ on $\partial\Delta_{\alpha}^0$ (resp. $\partial\Delta_{\alpha}^1$), if exist, to a critical point of $f_{\tilde{\alpha}}$ on $\partial\Delta_{\tilde\alpha}^1$ (resp. $\partial\Delta_{\tilde\alpha}^0$).
\end{enumerate}
\end{lem}
\begin{prop}\label{Symmetry of Gamma} The connectedness locus
$\mathcal{M}_\theta$ and its subsets $\Gamma_\theta $, $\Gamma_\theta^0$ and $\Gamma_\theta^1$ present a symmetric property in the following sense:

(1) $\alpha\in\mathcal{M}_\theta $ if and only if
$\lambda/\alpha\in\mathcal{M}_\theta$.

(2) $\alpha\in\Gamma_\theta $ if and only if
$\lambda/\alpha\in\Gamma_\theta$.

(3) $\alpha\in\Gamma_\theta^0$ if and only if
$\lambda/\alpha\in\Gamma_\theta^1$.
\end{prop}

\subsection{Critical points of $f_\alpha$}\label{critical points}
The two finite critical points of $f_\alpha $ are locally expressed by
\begin{equation}\label{critical pts}
\frac{(\frac{\lambda}{\alpha}+2\alpha+3)\pm \sqrt{(\frac{\lambda}{\alpha}+\alpha+3)^2-\lambda}}{3(\frac{\lambda}{\alpha}+\alpha+2)}.
\end{equation}
Let us denote the roots of $\frac{\lambda}{\alpha}+\alpha+2=0$ by $\alpha_0$ and $\tilde{\alpha}_0$, where $\tilde{\alpha}_0=\lambda/\alpha_0$. Furthermore, we denote the roots of $(\frac{\lambda}{\alpha}+\alpha+3)^2=\lambda$ by $\alpha_4$, $\alpha_5$, $\tilde{\alpha}_4$ and $\tilde{\alpha}_5$, where $|\alpha_4|>1$, $|\alpha_5|>1$,
$\tilde{\alpha}_4=\lambda/\alpha_4$ and $\tilde{\alpha}_5=\lambda/\alpha_5$. These four roots are essential singularities of the expressions given by (\ref{critical pts}). When we view each finite critical point $c(\alpha)$ of $f_\alpha$ as a complex analytic function of $\alpha$, its maximal domain of analyticity is obtained from $\EC$ by removing $0$, $\infty $, $\alpha_0$, $\tilde{\alpha}_0$, one arc connecting $\alpha_4$ to $\alpha_5$ and one arc connecting $\tilde{\alpha}_4$ to $\tilde{\alpha}_5$. Let us point out here that the maximal domain of analyticity of $c(\alpha)$ is not unique. We may use the maximal domain obtained from $\EC$ by removing $\alpha_0$, $\tilde{\alpha}_0$, two disjoint arcs respectively connecting $\alpha_4$ and $\alpha_5$ to $\infty $, and two disjoint arcs repectively connecting $\tilde{\alpha}_4$ and $\tilde{\alpha}_5$ to $0$. In fact, we will use these two maximal domains in the proof of the main theorem. 

Clearly, when $\alpha$ is equal to one of  $\alpha_4$, $\alpha_5$, $\tilde{\alpha}_4$ and $\tilde{\alpha}_5$, $f_\alpha$ has a double (finite) critical point. Thus, all four parameters belong to $\Gamma_\theta^0\cup \Gamma_\theta^1$. By the symmetry between $\Gamma_\theta^0$ and $\Gamma_\theta^1$ given in Proposition \ref{Symmetry of Gamma}, we know two of them belong to $\Gamma_\theta^0$ and two of them belong to  $\Gamma_\theta^1$. In fact, $\alpha_4, \alpha_5\in \Gamma_\theta^1$ and $\tilde{\alpha}_4, \tilde{\alpha}_5\in \Gamma_\theta^0$ (see Figure \ref{arc_0^1}).

Let $\alpha\in \Gamma_\theta$ and let $c_0$ and $c_1$ be the critical points on $\partial\Delta_\alpha^0$ and $\partial\Delta_\alpha^1$ respectively. 
We first consider how many parameters there are on $\Gamma_\theta$ such that $f_\alpha(c_1)=c_0$. Although there is an ambiguity on which expression in (\ref{critical pts}) is used for $c_0$ or $c_1$, we may obtain an equation on $\alpha$ without ambiguity by squaring both sides of an equivalent equation with one side equal to the involved square root only. This resulting equation can be changed to a relative simple expression by the substitutions
\begin{equation}\label{xy}
x=\frac{\lambda}{\alpha}+\alpha+2 \quad \text{ and } \quad y=\alpha+1,
\end{equation}
which is of the form
\begin{equation*}\label{x}
    4x^4+(3y^2-14y-12)x^3+(-6y^3+6y^2+39y+3)x^2+(3y^4+16y^3-39y^2)x-8y^4=0.
\end{equation*}
After plugging back the expressions of $x$ and $y$ in (\ref{xy}), we obtain
\begin{equation*}\label{alpha}
4\alpha^8+9\alpha^7+a_5\alpha^5+a_4\alpha^4+a_3\alpha^3+8\lambda^3\alpha^2+9\lambda^3\alpha+4\lambda^4=0,
\end{equation*}
where $a_5=3\lambda^2-12\lambda+522$, $a_4=6\lambda^2+18\lambda$
and $a_3=3\lambda^3+12\lambda^2$.

This degree-$8$ polynomial has eight roots. Computer verification indicates that six of them lie on the curve $\Gamma_\theta$.
By the symmetry of $\Gamma_\theta $ given in Proposition \ref{Symmetry of Gamma}, three of them correspond to the case when $f(c_0)=c_1$ and the other three correspond to the other case when $f(c_1)=c_0$. We denote them by $\alpha_k$ and $\tilde{\alpha}_k$, where $\tilde{\alpha}_k=\lambda /\alpha_k$, and $k=1, 2, 3$. In the last section, we prove that these six parameters are the only parameters $\alpha$ on $\Gamma_\theta$ such that one critical point is mapped to another by $f_\alpha$, and furthermore $\Gamma_\theta$ is a Jordan curve. It follows that the other two roots of the above degree-$8$ polynomial lie on the boundaries of two capture domains respectively by using the symmetry of $\mathcal{M}_\theta$ given in Proposition \ref{Symmetry of Gamma}.

\subsection{Julia set $J(f_\alpha)$ when $\alpha \notin \partial \mathcal{M}$} 
In this subsection, we show three different types of Julia sets when $\alpha \notin \partial \mathcal{M}$.

By definition, $\alpha\in\Sigma_\theta\setminus\mathcal{M}_\theta$ if and
only if one of the finite critical points escapes to $\infty$ under iteration of $f_{\alpha}$. Applying the Riemann-Hurwitz formula to the pullbacks of a sufficiently small neighborhood
of $\infty$ (in the immediate attracting basin $B(\infty)$ of $\infty$), one can show that $B(\infty)$ is infinitely connected. Therefore, the Julia set of $f_{\alpha}$ is disconnected. Such an example is shown in Figure \ref{Fig_Jul-(-2-0p2i)}. 
\begin{figure}[!htpb]
 \setlength{\unitlength}{1mm}
  \centering
  \includegraphics[width=0.62\textwidth]{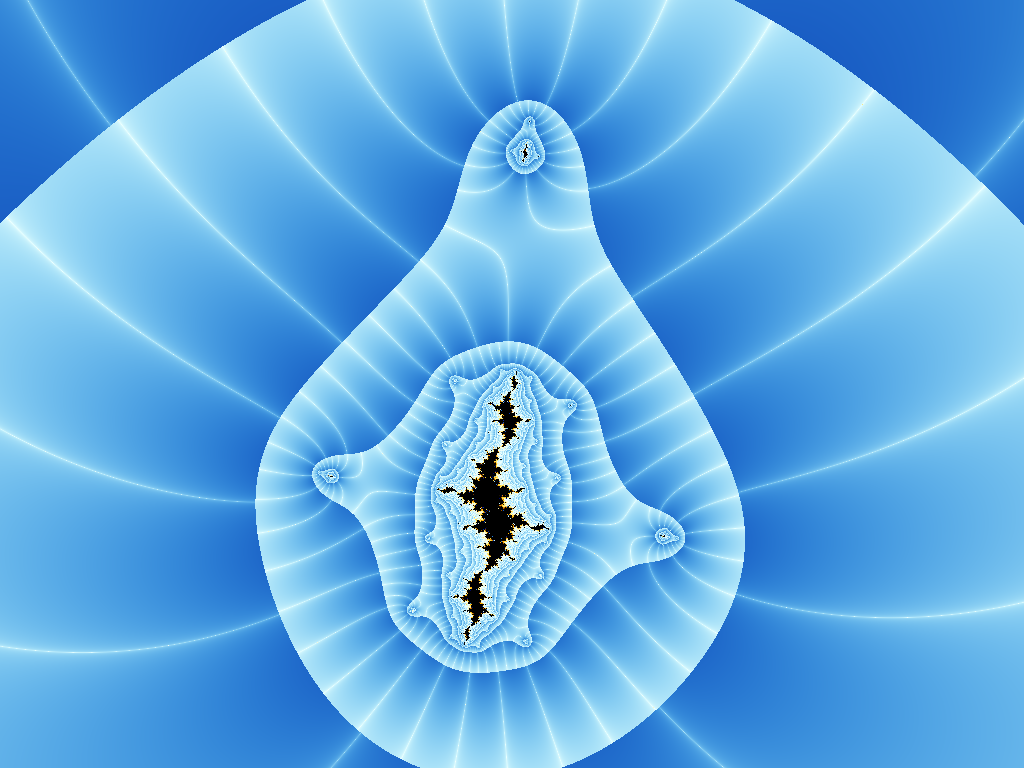}
  \caption{The Julia set of $f_\alpha$ for $\alpha=-2-0.2i$, a disconnected case.}
  \label{Fig_Jul-(-2-0p2i)}
\end{figure}

Figure \ref{hyperbolic-like case} shows the
Julia set of $f_\alpha$ in a hyperbolic-like case for which there is an attracting fixed point, and Figure \ref{Fig_Jul-(1p6066+p41583i)} shows the Julia set for a capture case in which one finite critical point lies on the boundary of $\Delta_\alpha^1$ while the other finite critical point is mapped into $\Delta_\alpha^1$.
\begin{figure}[!htpb]
 \setlength{\unitlength}{1mm}
  \centering
  \includegraphics[width=0.62\textwidth]{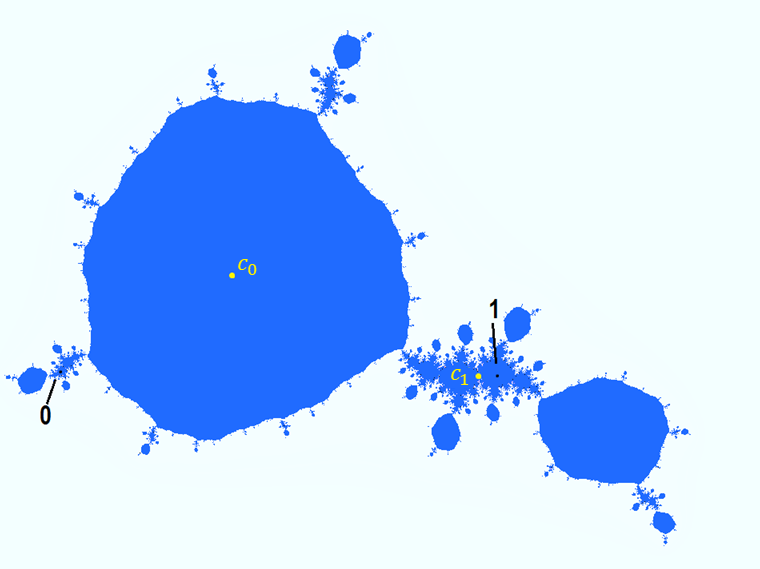}
  \caption{The Julia set of $f_\alpha$ for $\alpha=-2.5+2.5i$, a hyperbolic-like case.}
  \label{hyperbolic-like case}
\end{figure}

\begin{figure}[!htpb]
 \setlength{\unitlength}{1mm}
  \centering
  \includegraphics[width=0.72\textwidth]{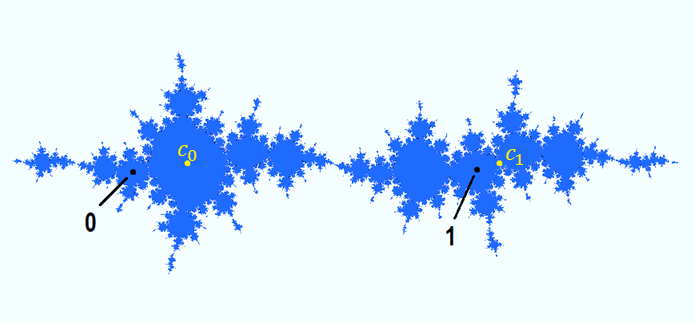}
  \caption{The Julia set of $f_\alpha$ for $\alpha=1.6066+.41583i$, a capture case.}
  \label{Fig_Jul-(1p6066+p41583i)}
\end{figure}

\subsection{Boundary of the connectedness locus $\mathcal{M}_\theta$}
Given $\alpha\in\Sigma_\theta$, $f_{\alpha}$ is said to be {\em J-stable} if there exists a neighborhood $O$ of $\alpha$ in $\Sigma_\theta$ such that for any $\beta\in O$, $f_{\beta}$ restricted on a neighborhood of  $J(f_{\beta})$ is topologically
conjugate to $f_{\alpha}$ restricted on a neighborhood of $J(f_{\alpha})$. We continue to denote the two (finite) critical points of $f_\alpha$ by $c_0(\alpha)$ and $c_1(\alpha)$. Keep in mind that in this subsection one of $c_0(\alpha)$ and $c_1(\alpha)$ may not lie on $\partial\Delta_{\alpha}^0\cup\partial\Delta_{\alpha}^1$.
According to \cite[Theorem 4.2]{McM94b}, $f_{\alpha}$ is J-stable if and only if both sequences $\{f^{\circ k}_{\alpha}(c_0(\alpha))\}_{k\ge 0}$ and $\{f^{\circ k}_{\alpha}(c_1(\alpha))\}_{k\ge 0}$ are normal on a neighborhood of $\alpha$. We prove the following characterization of the boundary of $\mathcal{M}_\theta$. 

\begin{prop}\label{unstable}
The map $f_{\alpha}$ is not J-stable if and only if 
$\alpha\in\partial\mathcal{M}_\theta$. 
\end{prop}

\begin{proof}
Clearly, $\mathcal{M}_\theta$ is a closed subset since $\Sigma_\theta\setminus \mathcal{M}_\theta$ is open. 

We first show that if $\alpha\in\partial\mathcal{M}_\theta$, then $f_{\alpha}$ is not J-stable.
When $\alpha_m\in\partial\mathcal{M}_\theta$, by the definition of $\mathcal{M}_\theta$, we know that any
neighborhood of $\alpha_m$ contains a parameter $\alpha$ such that the iterates of either $c_0$ or $c_1$ under $f_{\alpha}$ escape to $\infty$. But the iterates of both $c_0$ and $c_1$ under $f_{\alpha_m}$ are bounded. Therefore at least one of the sequences $\{f^{\circ k}_{\alpha}(c_0(\alpha))\}_{k\ge 0}$ and $\{f^{\circ k}_{\alpha}(c_1(\alpha))\}_{k\ge 0}$ fails to be normal in any neighborhood of $\alpha_m$. By \cite[Theorem 4.2]{McM94b}, $f_{\alpha_m}$ is not J-stable for
any $\alpha_m\in\partial\mathcal{M}_\theta$. 

It remains to show that if 
$\alpha_m$ is not on $\partial\mathcal{M}_\theta$, then $f_{\alpha_m}$ is J-stable. There are two cases to consider: $\alpha_m \in \Sigma_\theta \setminus \mathcal{M}_\theta$ or $\alpha_m \in int(\mathcal{M}_\theta)$.
Note that for any $\alpha_m\in \Sigma_\theta$, there is a neighborhood $U$ of $\alpha_m$ in $\Sigma_\theta$ and $\delta>0$ such that for any $\alpha\in U$, $f_\alpha(D_\delta(\infty))\subset D_\delta(\infty)$, where $D_\delta(\infty)$ is a disk, centered at $\infty $ and of radius $\delta$ in the spherical metric, in the attracting basin $B(f_\alpha)$ of $f_\alpha$ for all $\alpha\in U$. Using this fact,
one can see that if a sequence $\{f^{\circ k}_{\alpha}(c_j(\alpha_m))\}_{k\ge 0}$, $j=0$ or $1$, is unbounded, then there exists a neighborhood $O\subset U$ such that 
the sequence $\{f^{\circ k}_{\alpha}(c_j(\alpha))\}_{k\ge 0}$ converges to $\infty$ uniformly on $O$; if a sequence $\{f^{\circ k}_{\alpha}(c_j(\alpha_m))\}_{k\ge 0}$, $j=0$ or $1$, is bounded, then there exists a neighborhood $O\subset U$ such that 
the sequence $\{f^{\circ k}_{\alpha}(c_j(\alpha))\}_{k\ge 0}$ is uniformly bounded on $O$. This means that both sequences $\{f^{\circ k}_{\alpha}(c_0(\alpha))\}_{k\ge 0}$ and $\{f^{\circ k}_{\alpha}(c_1(\alpha))\}_{k\ge 0}$ are normal on $O$ in either case. Thus, $f_\alpha$ is $J$-stable in each of the two cases.
\end{proof}

\subsection{Boundaries of Siegel disks}
In this subsection, we prove that
$\partial\Delta_\alpha^0\cup\partial\Delta_\alpha^1$ changes
continuously as $\alpha$ varies in the parameter space
$\Sigma_\theta$. Let us first state
a result by Zhang in \cite{Zha11}.
\begin{thm}[\cite{Zha11}]\label{zhang11}
Let $d\ge 2$ be an integer and $0<\theta<1$ be an irrational
number of bounded type. Then there exists a constant
$1<K(d,\theta)<\infty$ depending only on $d$ and $\theta$ such
that for any rational map $f$ of degree $d$, if $f$ has a fixed
Siegel disk with rotation number $\theta$, then the boundary of
the Siegel disk is a $K(d,\theta)$-quasicircle which passes
through at least one critical point of $f$.
\end{thm}
\begin{cor}\label{atleast1}
For any $\alpha\in\Sigma_\theta$, $\partial{\Delta_{\alpha}^0}
\cup \partial{\Delta_{\alpha}^{1}}$ contains at least one of two
finite critical points.
\end{cor}
\begin{proof}
Note that $\Delta_\alpha^0$ and $\Delta_\alpha^1$ are fixed Siegel
disks of $f_\alpha^{\circ 2}$
and the rotation numbers of $f_\alpha^{\circ 2}$ on both of them are equal to 
$\theta$, which is of bounded type. Denote the critical points of $f_\alpha$ by $c_0$ and $c_1$. The set of critical points of
$f_\alpha^{\circ 2}$ is
\begin{equation*}\label{equ:crit-f2}
\Crit(f_\alpha^{\circ 2})=\{c_0,c_1\}\cup f_\alpha^{-1}(\{c_0\})\cup f_\alpha^{-1}(\{c_1\}).
\end{equation*}
According to \cite{Zha11}, each of $\partial\Delta_{\alpha}^0$ and
$\partial{\Delta_{\alpha}^{1}}$
must contain a point in
$\Crit(f_\alpha^{\circ 2})$, which implies that
$(\partial\Delta_{\alpha}^0\cup\partial\Delta_{\alpha}^1) \cap
\{c_0,c_1\}\neq\emptyset$.
\end{proof}
\begin{cor}\label{disjoint}
There exists $K=K(\theta)>1$ such that for any
$\alpha\in\Sigma_\theta$, $\partial\Delta_{\alpha}^0$ and
$\partial\Delta_{\alpha}^1$ are $K$-quasicircles, and
$\overline{\Delta_\alpha^{0}}\cap
\overline{\Delta_\alpha^{1}}=\emptyset$.
\end{cor}
\begin{proof}
By Theorem \ref{zhang11}, $\partial\Delta_{\alpha}^0$ and
$\partial\Delta_{\alpha}^1$ are $K$-quasicircles (where $K>1$ is
independent of $\alpha$) since they are the boundaries of two
fixed bounded type Siegel disks of $f_\alpha^{\circ 2}$. The map
$f^{\circ 2}_{\alpha}$ is topologically conjugate to the
irrational rotation $R_\theta(\zeta)=e^{2\pi\ii\theta}\zeta$ on
$\overline{\Delta_\alpha^{0}}$ (resp. $\overline{\Delta_\alpha^{1}}$).
Assume that $\partial\Delta_\alpha^{0}\cap
\partial\Delta_\alpha^{1}$ contains a point $z_0$. Then the
closure of the orbit $\{f^{\circ 2n}_{\alpha}(z_0):n\in\N\}$ is
the common boundary of $\overline{\Delta_\alpha^{0}}$ and
$\overline{\Delta_\alpha^{1}}$. This implies that every point of $\EC$ belongs to the Fatou set. Thus, the Julia set of
$f_\alpha^{\circ 2}$ is empty, which is impossible.
\end{proof}

Now we show the continuous dependence of
$\partial\Delta_{\alpha}^0\cup\partial\Delta_{\alpha}^1$ on the
parameter $\alpha$.
\begin{prop}\label{continuity}
Let $\alpha_0\in\Sigma_\theta$ and $\{\alpha_n\}_{n\geq 1}$ be a
sequence of parameters on $\Sigma_\theta$ converging to $\alpha_0$ as
$n\to\infty$. Then $\partial{\Delta_{\alpha_{n}}^0} \to
\partial{\Delta_{\alpha_0}^0}$, $
\overline{\Delta_{\alpha_{n}}^{0}} \to
\overline{\Delta_{\alpha_0}^{0}}$, 
$\partial{\Delta_{\alpha_{n}}^{1}} \to
\partial{\Delta_{\alpha_0}^{1}}$ and $
\overline{\Delta_{\alpha_n}^{1}} \to
\overline{\Delta_{\alpha_0}^{1}}$ as $n\to\infty$, with respect to the
Hausdorff metric.
\end{prop}
\begin{proof} Without loss of generality, we prove the results for the Siegel disk $\Delta_{\alpha}^0$. Let $\D$ be the unit disk. For any $\alpha\in\Sigma_\theta$, there exist a unique conformal map
\begin{equation*}\label{equ:h-0-infty}
h_{\alpha}^0: \Delta_{\alpha}^0\to\D 
\end{equation*}
such that 
  $h_{\alpha}^0(0)=0$, $(h_{\alpha}^0)'(0)>0$, and 
\begin{equation*}\label{equ:h-condition}
  h_{\alpha}^0 \circ f^{\circ 2}_{\alpha} \circ (h_{\alpha}^0)^{-1}(\zeta) =
  R_\theta(\zeta), ~\forall\, \zeta\in \D,
\end{equation*}
where $R_\theta(\zeta)=e^{2\pi\ii\theta}\zeta$. Let $\eta_\alpha:\D\to \Delta_{\alpha}^0$ be the inverse map of $h_\alpha^0$.

According to \cite[Theorem 2.4, p.\,17]{Leh87}, there exists a constant $c(K_0)>1$ such that for all $\alpha\in\Sigma_\theta$,
\begin{equation}\label{equ:Leh}
\frac{\max_{z\in\partial\D}|\eta_\alpha(z)|}{\min_{z\in\partial\D}|\eta_\alpha(z)|}=\frac{\max\{|z|:z\in\partial\Delta_\alpha^0\}}{\min\{|z|:z\in\partial\Delta_\alpha^0\}}\leq c(K_0).
\end{equation}

Let $\{\alpha_n\}_{n\in \mathbb{N}}$ be a sequence of parameters in $\Sigma_\theta$ converging to $\alpha_0\in \Sigma_\theta$. 
The size of $\Delta_{\alpha_n}^0$ cannot be arbitrarily large in the sense that there exists a constant $M=M(K_0)>1$ such that
\begin{equation}\label{equ:notbig}
    \overline{\Delta_{\alpha_n}^0}\subset\{z:|z|<M\}.
\end{equation}
since the immediately attracting basin of $\infty$ contains a disk of $\infty $ of a definite size (under the spherical metric on $\EC$) for all $\alpha $ in a neighborhood of $\alpha _0$ in $\Sigma_\theta$. Meanwhile, $\partial\Delta_{\alpha_n}^0$ contains either $c_0(\alpha_n)$ or $f_{\alpha_n}(c_0(\alpha_n))$ and $c_0(\alpha_n)\to c_0(\alpha_0)$ and $f_{\alpha_n}(c_0(\alpha_n))\to f_{\alpha_0}(c_0(\alpha_0))$ as $\alpha_n\to \alpha_0$. Then $\max\{|c_0(\alpha_n)|,|f_{\alpha_n}(c_0(\alpha_n))|\}\ge \delta$ for some constant $\delta$ and large $n$. By \eqref{equ:Leh}, 
\begin{equation}\label{equ:notsmall}
    \min\{|z|:z\in\partial\Delta_{\alpha_n}^0\}\ge \frac{\delta}{c(K_0)}.
\end{equation}

By Theorem \ref{zhang11}, there exists
$\widetilde{K}_0>1$ such that each $\eta_{\alpha_n}:\D\to \Delta_{\alpha}^0$ has a $\widetilde{K}_0$-quasiconformal extension to $\EC$ with $\tilde{\eta}_{\alpha_n}(0)=0$ (since $\eta_{\alpha_n}(0)=0$) and $\tilde{\eta}_{\alpha_n}(\infty)=\infty$. 
Using (\ref{equ:notbig}) and (\ref{equ:notsmall}) and 
applying \cite[Theorem 2.1, p.\,14]{Leh87} (by selecting three points at $0$, $1$ and $\infty $), we know that $\{\tilde{\eta}_{\alpha_n}:\EC\to\EC\}_{n\in\mathbb{N}}$ is a normal family. Then by passing to a subsequence, we may assume that $\tilde{\eta}_{\alpha_n}$ converges uniformly to a quasiconformal mapping $\tilde{\eta}:\EC\to\EC$ as $n\to\infty$.
Note that each
$\tilde{\eta}_{\alpha_n}$ is conformal on $\D$ and
$(\tilde{\eta}_{\alpha_n})'(0)>0$, and $\tilde{\eta}|_{\D}$ is also  conformal. Then $\tilde{\eta}'(0)>0$. Taking the limits from the both sides of the
equation $$f_{\alpha_n}^{\circ 2}\circ
\tilde{\eta}_{\alpha_n}(\zeta)=\tilde{\eta}_{\alpha_n}\circ
R_\theta(\zeta),$$ we obtain $$f_{\alpha_0}^{\circ 2}\circ
\tilde{\eta}(\zeta)=\tilde{\eta}\circ R_\theta(\zeta) \text{ for each }\zeta\in\D.$$ 
Recall that
$(\tilde{\eta}_{\alpha_0})^{-1}:\Delta_{\alpha_0}^0\to\D$ is the
unique conformal map that conjugates $f_{\alpha_0}^{\circ 2}$ to
$R_\theta$ with $\tilde{\eta}_{\alpha_0}(0)=0$ and
$(\tilde{\eta}_{\alpha_0})'(0)>0$. It implies that
$\tilde{\eta}|_{\D}=\tilde{\eta}_{\alpha_0}|_{\D}.$ Furthermore $\tilde{\eta}|_{\overline{\D}}=\tilde{\eta}_{\alpha_0}|_{\overline{\D}}$
 since $\tilde{\eta}$ and $\tilde{\eta}_{\alpha_0}$ are continuous.
Therefore, $\tilde{\eta}_{\alpha_n}(\partial{\D}) \to
\tilde{\eta}_{\alpha_0}(\partial{\D})$ and
$\tilde{\eta}_{\alpha_n}(\overline{\D}) \to
\tilde{\eta}_{\alpha_0}(\overline{\D})$ with respect to the
Hausdorff metric as $n\to\infty$.
\end{proof}

\section{Combinatorics of the Fatou components of $f_\alpha$ when $\alpha\in\Gamma_\theta\cup\Gamma_\theta^0\cup\Gamma_\theta^1$}\label{Combinatorial pattern of eventually Siegel disk comp}

In this section, we describe the combinatorial patterns of the Fatou components of $f_\alpha$ when $\alpha$ belongs to $\Gamma_\theta$, $\Gamma_\theta^0$ and $\Gamma_\theta^1$ respectively. 
\subsection{External rays landing a separating repelling fixed point}\label{rays landing a separating repelling fixed point}
Let us first prove two lemmas. Denote by $c_0$ and $c_1$ the two finite critical points of $f_\alpha$. Using the definitions of $\Gamma_\theta$, $\Gamma_\theta^0$ and $\Gamma_\theta^1$ given in the introduction, we know that (i) if $\alpha\in \Gamma_\theta$, then $c_0\neq c_1$ because $\Delta_\alpha^0$ and $\Delta_\alpha^1$ don't share any boundary point; (ii) for any parameter $\alpha$ such that $c_0=c_1$, $\alpha$ belongs to $\Gamma_\theta^0$ or $\Gamma_\theta^1$.

\begin{lem}
For any $\alpha\in\Gamma_\theta\cup\Gamma_\theta^0\cup\Gamma_\theta^1$, the following properties hold. 
\begin{enumerate}
    \item The Fatou set $F(f_\alpha)=\cup_{n=0}^{\infty}(f^{-n}_\alpha(\Delta_{\alpha}^0\cup\Delta_{\alpha}^1))\cup B(\infty)$, where $B(\infty)$ is the immediate attracting basin of $\infty$.
    \item For any two different nonnegative integers $m$ and $n$, $$f^{-2n}_\alpha(\overline{\Delta_{\alpha}^0})\cap f^{-2m}_\alpha(\overline{\Delta_{\alpha}^1})=\emptyset.$$
\end{enumerate}
\end{lem}

\begin{proof}
\medskip
(a) Since both critical points are contained in the boundaries of the
cycle $\{\Delta_\alpha^0,\Delta_\alpha^1\}$, it follows that all
Fatou components of $f_\alpha$ are eventually mapped onto this
cycle of Siegel disks except $B(\infty)$ since $B(\infty)$ is a completely
invariant Fatou component. The statement follows from
$f_\alpha^{\circ 2}(\Delta_\alpha^0)=\Delta_\alpha^0$ and
$f_\alpha^{\circ 2}(\Delta_\alpha^1)=\Delta_\alpha^1$.

(b) If $f_\alpha^{-2m}(\overline{\Delta_\alpha^0})\cap
f_\alpha^{-2n}(\overline{\Delta_\alpha^1})\neq\emptyset$ for some
$m$, $n\geq 0$, then $\overline{\Delta_\alpha^0}\cap
\overline{\Delta_\alpha^1}\neq\emptyset$. This contradicts Lemma
\ref{disjoint}.
\end{proof}

\medskip
Let $\Psi_\alpha : B(\infty) \rightarrow \mathbb{D}$ be a conformal map that conjugates $f_{\alpha}(z): B(\infty) \rightarrow B(\infty)$ to the map $z\mapsto z^3$ on $\mathbb{D}$. Note that $\Psi_\alpha$ is unique up to multiplication by $-1$. Throughout this paper, keep in mind that we have a consent that $\Psi_\alpha$ is chosen to be continuously depending on $\alpha$ as $\alpha$ varies. 
The image of the radius with angle $2\pi t$ radians, $0\le t\le 1$, under $\Psi_\alpha^{-1}$ is called the external ray of (normalized) angle $t$ or the $t$-ray, which we denote by $R_t$. Following this notation, $R_{0}$ and $R_{\frac{1}{2}}$ are the two fixed external rays under $f_\alpha$; $R_{\frac{1}{4}}$ and $R_{\frac{3}{4}}$, $R_{\frac{1}{8}}$ and $R_{\frac{3}{8}}$, and $R_{\frac{5}{8}}$ and $R_{\frac{7}{8}}$ form three $2$-cycles. 
\begin{lem}\label{existence of separating repelling fixed point}
    For each $\alpha\in\Gamma_\theta\cup\Gamma_\theta^0\cup\Gamma_\theta^1$, all three fixed points of $f_\alpha$ are repelling, and only one of them separates the Julia set $J(f_\alpha)$, which we denote by $P_3$.
\end{lem}
\begin{proof}
For any $\alpha \in \Sigma_\theta$, $f_{\alpha}$ has at most one attracting fixed point. Denote by $\mathcal{A}_\theta$ the collection of $\alpha\in \Sigma_\theta$ such that $f_\alpha $ has an attracting finite fixed point, which we call the locus with a finite attracting fixed point. On Figure \ref{Fig_parameter}, $\mathcal{A}_\theta$ is the union of the interiors of the main bodies (heart-shaped black regions) of the two biggest Mandelbrot-like sets and their images under the map $\alpha \rightarrow \frac{\lambda}{\alpha}$,  When $\alpha\notin \overline{\mathcal{A}_\theta}$, we obtain the followings: 
(i) all three fixed points of $f_\alpha$ are repelling; (ii) two of them are the landing points of the two fixed rays in $B(\infty)$ and they don't separate the Julia set $J(f_\alpha )$, which we denote by $P_1$ and $P_2$; 
(iii) the other fixed point is the landing point of the external rays on a cycle with period bigger than $1$ and it separates $J(f_\alpha )$, which we denote by $P_3$ and call a separating repelling fixed point.

Since every Cremer fixed point or Cremer periodic point of a rational map of degree $\ge 2$ is contained in the closure of its postcritical set (Theorem 11.17 in \cite{M1}), it follows that the closure of $\mathcal{A}_\theta$ is disjoint from $\Gamma_\theta\cup\Gamma_\theta^0\cup\Gamma_\theta^1$. Thus, for each $\alpha\in\Gamma_\theta\cup\Gamma_\theta^0\cup\Gamma_\theta^1$, $f_\alpha $ has three distinct fixed points, all of them are repelling, and only one of them separates the Julia set $J(f_\alpha)$.
\end{proof}

\begin{lem}\label{pattens for landing rays}
    For each $\alpha\in\Gamma_\theta\cup\Gamma_\theta^0\cup\Gamma_\theta^1$, there are exactly two external rays $R^{(1)}$ and $R^{(2)}$ on a $2$-cycle landing at the separating repelling fixed point $P_3$ and  $R^{(1)}\cup \{P_3\}\cup R^{(2)}$ separates $\Delta_\alpha^0$ from $\Delta_\theta^1$. Furthermore, under a proper normalization on the attracting basin at $\infty$, the (normalized) angles of $R^{(1)}$ and $R^{(2)}$ are equal to $\frac{1}{4}$ and $\frac{3}{4}$ when $\alpha\in \Gamma_\theta$; the angles of $R^{(1)}$ and $R^{(2)}$ are equal to $\frac{1}{8}$ and $\frac{3}{8}$ when $\alpha\in \Gamma_\theta^1$; the angles of $R^{(1)}$ and $R^{(2)}$ are equal to $\frac{5}{8}$ and $\frac{7}{8}$ when $\alpha\in \Gamma_\theta^0$.
\end{lem}

\begin{proof} Clearly, $f_\alpha $ has six periodic external rays of period $2$ in $B(\infty)$. They form three $2$-cycles, which we denote by 
 $\{R_{\frac{1}{4}}$, $R_{\frac{3}{4}}\}$, $\{R_{\frac{1}{8}}$, $R_{\frac{3}{8}}\}$, and $\{R_{\frac{5}{8}}$, $R_{\frac{7}{8}}\}$. Suppose these six external rays land at six distinct points. Then these six landing points form three distinct $2$-cycles of periodic points, which are on the boundary of $B(\infty)$. On the other hand, we know $0$ and $1$ are on an existing $2$-cycle in the interiors of the two Siegel disks $\Delta_{\alpha}^0$ and
$\Delta_{\alpha}^1$. Thus, $f_\alpha$ has at least four distinct $2$-cycles. This is a contradiction. Therefore, some of 
 $R_{\frac{1}{8}}$, $R_{\frac{1}{4}}$, $R_{\frac{3}{8}}$, $R_{\frac{5}{8}}$, $R_{\frac{3}{4}}$ and $R_{\frac{7}{8}}$ have to land at the same point. 
 We know that $f_{\alpha}(z)$ has three distinct fixed points. Two of them are the landing points of the two fixed rays $R_{0}$ and $R_{\frac{1}{2}}$, which are denoted by $P_1$ and $P_2$, and the other is a landing point of a cycle of periodic rays, which is denoted by $P_3$. 
 If two external rays on a $2$-cycle land at the same point, then the landing point has to be $P_3$. All possible landing patterns of the six rays $R_{\frac{1}{8}}$, $R_{\frac{1}{4}}$, $R_{\frac{3}{8}}$, $R_{\frac{5}{8}}$, $R_{\frac{3}{4}}$ and $R_{\frac{7}{8}}$ are presented in Figure \ref{all possible landing patterns}. 
\begin{figure}[!htpb]
 \setlength{\unitlength}{1mm}
  \centering
  \includegraphics[width=0.8\textwidth]{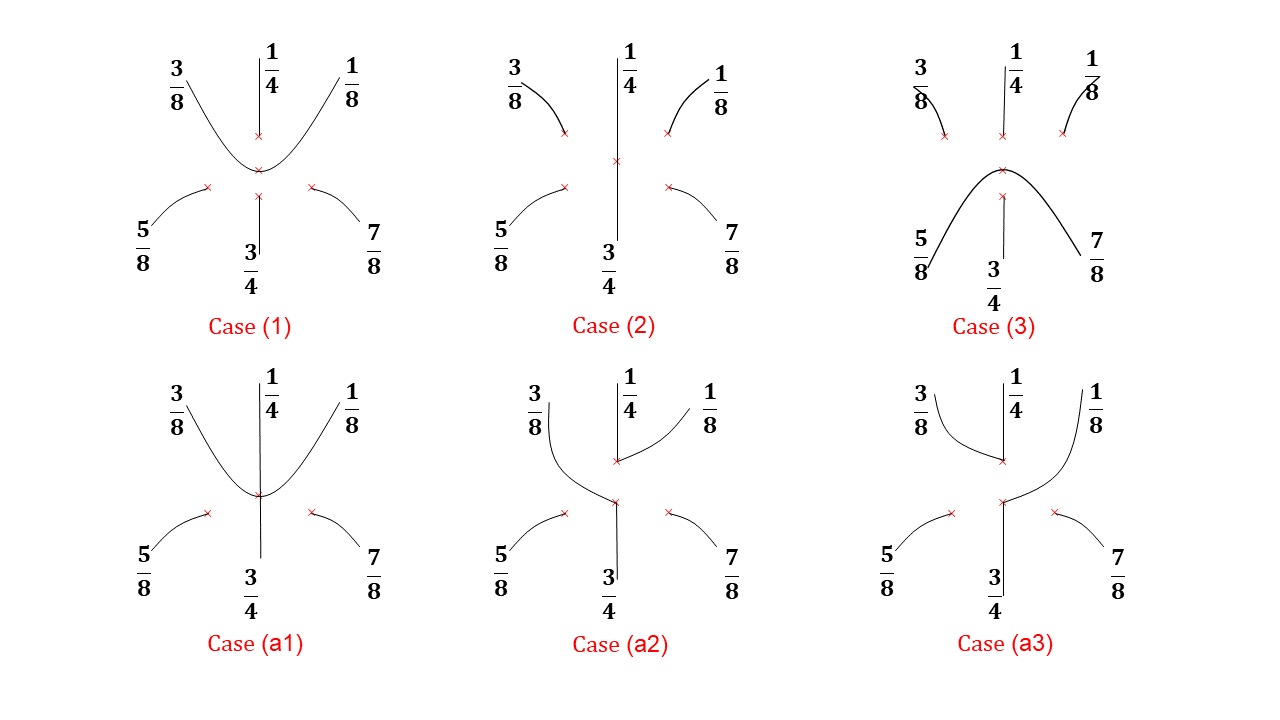}
  \includegraphics[width=0.8\textwidth]{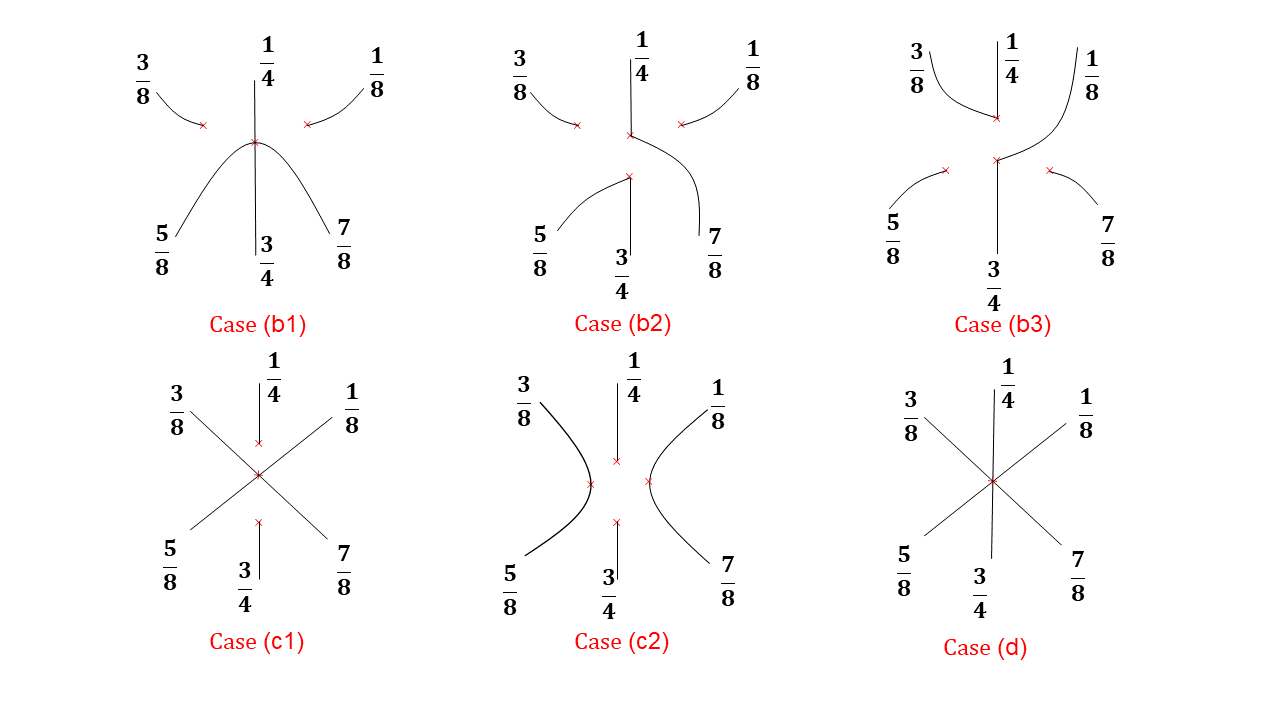}
  \caption{All possible landing patterns for periodic external rays of period $2$.}
  \label{all possible landing patterns}
\end{figure}

Under the assumption that $\alpha\in\Gamma_\theta\cup\Gamma_\theta^0\cup\Gamma_\theta^1$, we now show that except Case (1), Case (2) and Case (3), any other case can not happen. 

Suppose that Case (a1) happens. Then the four external rays $R_{\frac{1}{8}}$, $R_{\frac{1}{4}}$, $R_{\frac{3}{8}}$ and $R_{\frac{3}{4}}$ divide the complex plane into four domains, which we denote by $\Omega_0$, $\Omega_1$, $\Omega_2$ and $\Omega_3$ as shown in Case (a1) in Figure \ref{impossible landing patterns}. 

Under $f_\alpha$, $\Omega_0$ covers $\Omega_2$; $\Omega_1$ covers $\Omega_3$; $\Omega_2$ covers $\Omega_1$, $\Omega_2$ and $\Omega_3$ and covers $\Omega_0$ twice; $\Omega_3$ covers $\Omega_0$, $\Omega_2$ and $\Omega_3$ and covers $\Omega_1$ twice.
Then $\Omega_2$ is covered by three disjoint domains $\Omega_0$, $\Omega_2$ and $\Omega_3$. Therefore, $\Omega_2$ contains no finite critical value. Similarly, we obtain $\Omega_3$ contains no finite critical value. Thus, the two finite critical values, denoted by $v_0$ and $v_1$, are contained in $\Omega_0\cup\Omega_1$. We claim that neither $\Omega_0$ nor $\Omega_1$ can contain both $v_0$ and $v_1$. Otherwise, we may assume that both $v_0$ and $v_1$ are contained in $\Omega_0$. Then a path $\beta$ between $v_0$ and $v_1$ in $\Omega_0$ has three preimages in $\Omega_2$. On the other hand, $\Omega_3$ covers $\Omega_0$ under $f_\alpha$. Then $\beta$ has one more preimage in $\Omega_3$. This is a contradiction since $f_\alpha$ is a degree-3 rational map.
Thus, each of $\Omega_0$ and $\Omega_1$ contains a finite critical value. Let us assume that $v_0=f_\alpha(c_0)\in \Omega_0$ and $v_1=f_\alpha(c_1)\in \Omega_1$.
Since $\Omega_3$ covers $\Omega_0$ and $\Omega_2$ covers $\Omega_0$ twice, it follows that $c_0\in \Omega_2$. Similarly, we know that $c_1\in \Omega_3$.

Both $\overline{\Delta_\alpha^0}$ and $\overline{\Delta_\alpha^1}$ are contained in the interiors of one or two $\Omega$ domains. Since $\alpha \in\Gamma_\theta\cup\Gamma_\theta^0\cup\Gamma_\theta^1$, both critical points stay on $\partial\Delta_\alpha^0\cup \partial\Delta_\alpha^1$. We first see that $\alpha\notin \Gamma_\theta^0\cup\Gamma_\theta^1$. Secondly, if $\alpha\in \Gamma_\theta$, then $f_\alpha(c_0)\in \partial\Delta_\alpha^1\subset \Omega_3$, which is a contradiction to $f_\alpha(c_0)\in \Omega_0$. Thus, $\alpha\notin \Gamma_\theta$.
Therefore, $\alpha\notin\Gamma_\theta\cup\Gamma_\theta^0\cup\Gamma_\theta^1$, which contradicts the assumption. So we conclude that Case (a1) can not happen. 

\begin{figure}[!htpb]
 \setlength{\unitlength}{1mm}
  \centering
  \includegraphics[width=0.8\textwidth]{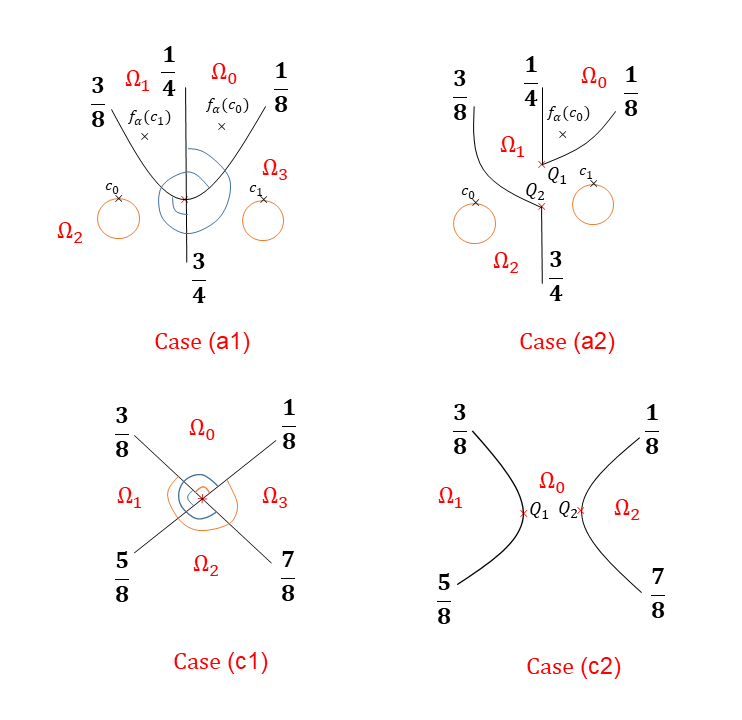}
  \caption{Illustrations of the domains used to in the process to dismiss certain landing patterns of periodic external rays of period $2$.}
  \label{impossible landing patterns}
\end{figure}

Suppose that Case (a2) happens. Denote by $Q_1$ the landing point of $R_{\frac{1}{8}}$ and $R_{\frac{1}{4}}$ and by $Q_2$ the landing point of $R_{\frac{3}{8}}$ and $R_{\frac{3}{4}}$. These four rays divide the complex plane into three regions, which we denote by $\Omega_0$, $\Omega_1$ and $\Omega_2$ as shown for Case (a2) in Figure \ref{impossible landing patterns}.
Under the map $f_\alpha$, $\Omega_0$ is covered by $\Omega_2$ twice and by $\Omega_1$ once; $\Omega_1$ is covered by $\Omega_1$ twice and by $\Omega_2$ once; $\Omega_2$ is covered by each of $\Omega_0$, $\Omega_2$ and $\Omega_1$ once. By applying arguments similar to the ones used to dismiss Case (a1), we first see that there is one critical point in $\Omega_2$, denoted by $c_0$, with its critical value $f_\alpha(c_0)\in \Omega_0$ and there is one critical point in $\Omega_1$, denoted by $c_1$, with its critical value $f_\alpha(c_1)\in \Omega_1$; then we obtain $\alpha\notin\Gamma_\theta^0\cup\Gamma_\theta^1$ and $\alpha\notin\Gamma_\theta$. Thus, we derive a contradiction to the assumption $\alpha\in\Gamma_\theta\cup\Gamma_\theta^0\cup\Gamma_\theta^1$ as in Case (a1). Therefore, Case (a2) can not appear neither. 

The method to dismiss Case (a3) is very similar to the one to dismiss Case (a2). 

Note that the landing patterns given in Cases (b1), (b2) and (b3) are the respective images of Cases (a1), (a2) and (a3) under adding $\frac{1}{2}$ modulus $1$. Therefore, the proofs for the nonexistence of Cases (a1), (a2) and (a3) can be modified to prove the nonexistence of Cases (b1), (b2) and (b3) respectively. 

Now we show the nonexistence of Case (c1). Suppose it happens. Then the four external rays $R_{\frac{1}{8}}$, $R_{\frac{3}{8}}$, $R_{\frac{5}{8}}$ and $R_{\frac{7}{8}}$ divide the complex plane into four domains $\Omega_0$, $\Omega_1$, $\Omega_2$ and $\Omega_3$ as shown in Case (c1) in Figure \ref{impossible landing patterns}. Then $f_\alpha(\Omega_0)=\Omega_1\cup \Omega_2\cup \Omega_3$,
$f_\alpha(\Omega_1)=\Omega_0\cup \Omega_1\cup \Omega_2$, $f_\alpha(\Omega_2)=\Omega_3\cup \Omega_0\cup \Omega_1$, and $f_\alpha(\Omega_3)=\Omega_2\cup \Omega_3\cup \Omega_0$. It follows that each of $\Omega_0$, $\Omega_1$, $\Omega_2$ and $\Omega_3$ is covered by three of them under $f_\alpha $. It implies that for each point $w\in J(f_\alpha )\setminus \{P_3\}$, it has three distinct preimages under $f_\alpha $. Thus, $f_\alpha $ has no critical point on $\mathbb{C}\setminus \{P_3\}$. This is impossible. Thus, Case (c1) can not appear. 

Next we prove the nonexistence of Case (c2). Denote by $Q_1$ the landing point of $R_{\frac{3}{8}}$ and $R_{\frac{5}{8}}$ and by $Q_2$ the landing point of $R_{\frac{1}{8}}$ and $R_{\frac{7}{8}}$. These four rays divide the complex plane into three regions, which we denote by $\Omega_0$, $\Omega_1$ and $\Omega_2$ as shown for Case (c2) in Figure \ref{impossible landing patterns}.
Under the map $f_\alpha$, $\Omega_1$ covers $\Omega_0$ and $\Omega_1$ and it doesn't cover $\Omega_2$; $\Omega_2$ covers $\Omega_0$ and $\Omega_2$ and it doesn't cover $\Omega_1$; $\Omega_0$ covers $\Omega_0$ and it covers each of  $\Omega_1$ and $\Omega_2$ twice. 
Since each of $\Omega_0$, $\Omega_1$ and $\Omega_2$ covers $\Omega_0$, there is no critical value in $\Omega_0$. Because $\Omega_1$ and $\Omega_2$ don't cover each other, either both of the two Siegel disks on the $2$-cycle are contained in $\Omega_1$ or $\Omega_2$ or one of them is contained in $\Omega_0$. 
If one of the two Seigel disks is contained in $\Omega_0$, then there is no critical values on its boundary. This implies that there is at least one Sigel disk contained in $\Omega_1$ or $\Omega_2$. Without loss of generality, we assume $\Delta_\alpha^0\subset \Omega_1$.
If there is only one Siegel disk on the $2$-cycle contained in $\Omega_1$, then 
there are two different critical values on the boundary of $\Delta_\alpha^0$. It follows that $\Delta_\alpha^0$ had three distinct preimages in $\Omega_0$, but it also has a fourth preimage in $\Omega_1$. That is a contradiction. Thus, both of the Siegel disks are contained in $\Omega_1$. Since there are critical values on the boundaries of the Siegel disks, one of them, namely $\Delta_\alpha^0$, has at least two preimages in $\Omega_1$. But because $\Omega_0$ covers $\Omega_1$ twice under $f_\alpha $, $\Delta_\alpha^0$ has two other preimages in $\Omega_0$. This is a contradiction. Thus, the possibility for the two Siegel disks in $\Omega_1$ can not happen neither. In summary, we conclude that Case (c2) can not appear. 

Note that the method to dismiss Case (a1), (b1), (c1) or (c2)  can be applied to show the nonexistence of Case (d).

By now, we have shown that there are two external rays $R^{(1)}$ and $R^{(2)}$ landing at the separating repelling fixed point $P_3$. In the remaining part, under the assumption that $\alpha\in \Gamma_\theta\cup \Gamma_\theta^1\cup \Gamma_\theta^0$ we show that Case (1) happens if and only if $\alpha \in \Gamma_\theta^1$; Case (2) happens if and only if $\alpha \in \Gamma_\theta$; Case (3) happens if and only if $\alpha \in \Gamma_\theta^0$.

Let us first consider Case (2); that is, the two rays landing at $P_3$ are $R_{\frac{1}{4}}$ and $R_{\frac{3}{4}}$. Their union divides the complex plane into two regions, which we denote by $\Omega_1$ and $\Omega_2$, containing the $R_{\frac{1}{8}}$ and $R_{\frac{3}{8}}$ respectively. Under $f_\alpha$, $\Omega_1$ covers $\Omega_2$ twice and vice versa.
Thus, each of $\Omega_1$ and $\Omega_2$ contains a critical value. This implies that the two Siegel disks $\Delta_\alpha^0$ and $\Delta_\alpha^1$ lie on the different sides of the union of the rays landing at $P_3$. Therefore, $\alpha\in \Gamma_\theta$. See an example in Figure \ref{Julia set with a parameter on the Gamma curve}.

When Case (1) happens, $R_{\frac{1}{8}}\cup \{P_3\}\cup R_{\frac{3}{8}}$ divides the complex plane into two regions, which we denote by $\Omega_1$ and $\Omega_2$, containing $R_{\frac{1}{4}}$ and $R_{\frac{3}{4}}$ respectively. Under $f_\alpha$, $\Omega_1$ covers $\Omega_2$ once, and $\Omega_2$ covers itself twice and $\Omega_1$ three times. Thus,
$\Omega_2$ contains two critical points (counted by multiplicity) and the critical values are contained in $\Omega_1$. It follows that one of $\Delta_\alpha^0$ and $\Delta_\alpha^1$ is contained in $\Omega_1$ and the other is contained in $\Omega_2$.
Thus, $\alpha \in \Gamma_\theta^0\cup \Gamma_\theta^1$. Note also that $\Omega_1$ is the domain containing the Siegel disk without critical points on its boundary and the difference of the angles of the two rays bounding $\Omega_1$ is $\frac{1}{4}$.
Similarly, we can see that when Case (3) happens, $\alpha \in \Gamma_\theta^0\cup \Gamma_\theta^1$. Therefore, if $\alpha\in \Gamma_\theta$, then Case (2) happens.

Now assume $\alpha\in \Gamma_\theta^1$ (resp. $\Gamma_\theta^0$) and let $R^{(1)}$ and $R^{(2)}$ be the two rays on a $2$-cycle landing at $P_3$. From the arguments in the previous paragraph, we know that $R^{(1)}\cup \{P_3\}\cup R^{(2)}$ separates $\Delta_\alpha^0$ from $\Delta_\alpha^1$ and the domain on the left (right) side of $R^{(1)}\cup \{P_3\}\cup R^{(2)}$ is spanned by an external angle of measure $\frac{1}{4}$. 
By choosing a proper normalization on the attracting basin at $\infty $, we may have 
$R^{(1)}$ and $R^{(2)}$ equal to $R_{\frac{1}{8}}$ and $R_{\frac{1}{8}}$ respectively when $\alpha \in \Gamma_\theta^1$ and have $R^{(1)}$ and $R^{(2)}$ equal to $R_{\frac{5}{8}}$ and $R_{\frac{7}{8}}$ respectively when $\alpha \in \Gamma_\theta^0$.
\end{proof}
\subsection{Labels of eventually Siegel-disk components when $\alpha\in \Gamma_\theta$}\label{Labels of eventually Siegel comps - curve}
The Fatou components on the grand orbit of $\Delta_\alpha^0$ under $f_\alpha$ are called {\em the eventually Siegel-disk Fatou components of $f_\alpha$}.  
From the previous two lemmas, we know that for each $\alpha\in\Gamma_\theta\cup\Gamma_\theta^0\cup\Gamma_\theta^1$, all three fixed points of $f_\alpha$ are repelling and there are two external rays $\beta_1$ and $\beta_2$ in $B_\infty$ on a $2$-cycle under $f_\alpha$ landing on the separating repelling fixed point $P_3$. We first use the grand orbit of these two external rays under $f_\alpha$ to divide the eventually Siegel-disk Fatou components into different groups. We call the union of the components in the same group {\em a cluster of the eventually Siegel-disk Fatou components}. Furthermore, a cluster is called a {\em crucial} cluster if it contains at least one of $\Delta_\alpha^0$ and $\Delta_\alpha^1$. In fact, when $\alpha\in\Gamma_\theta$, there are two crucial clusters containing $\Delta_\alpha^0$ and $\Delta_\alpha^1$ respectively, which we denote by $\mathcal{C}_0$ and $\mathcal{C}_1$; when $\alpha\in\Gamma_\theta^0\cup\Gamma_\theta^1$, there is one crucial cluster containing both $\Delta_\alpha^0$ and $\Delta_\alpha^1$. 

Let us first introduce the main ideas on how to describe the combinatorial pattern of the eventually Siegel-disk components of $f_\alpha$ when $\alpha\in\Gamma_\theta$.
The relative positions of all clusters are determined by the partition of $\C$ by the grand orbit of $\beta_1$ (or $\beta_2$) under $f_\alpha$. So the next step is to distinguish the relatively homotopic class of $\mathcal{C}_0$ (resp. $\mathcal{C}_1$) with respect to $\Delta_\alpha^0$ and $c_0$ (resp. $\Delta_\alpha^1$ and $c_1$), which is uniquely determined by the relative positions of the Fatou components in $\mathcal{C}_0$ (resp. $\mathcal{C}_1$). When $\alpha$ varies locally,
the conformal angle between $c_0$ and $f_{\alpha}(c_1)$ on $\Delta_\alpha^0$ determines the relative positions for those Fatou components attached to the boundary of $\Delta_\alpha^0$ in the course of pulling back $\Delta_\alpha^0$ and $\Delta_\alpha^1$, which in turn determine 
the relative positions of all Fatou components on each of two crucial clusters. Note that this idea is first introduced in \cite[\S 0]{Pet96} and used in \cite{FYZ20} to describe the relative positions of all Fatou components of a quadratic rational map with two critical points respectively on the boundaries of two Siegel disks on a 2-cycle (since all Fatou components are eventually Siegel-disk Fatou components in this case). In fact, in the one-parameter family of quadratic rational maps considered in \cite{FYZ20}, there is a unique parameter such that the critical point on the boundary of one Siegel disk is mapped to the critical point on the boundary of the other Siegel disk. But for the locus 
$\Gamma_\theta$ studied in this paper, there are different parameter values such that $c_0$ is mapped to $c_1$. This means that in order to distinguish the relative homotopic class of $\mathcal{C}_0$, besides using the conformal angle between $c_0$ and $f_\alpha(c_0)$, we need to apply extra information of the combinatorial pattern of the eventually Siegel-disk components on $\mathcal{C}_0\cup\mathcal{C}_1$, for which we 
will construct special sequences of components on $\mathcal{C}_0\cup\mathcal{C}_1$ converging to the separating fixed point $P_3$ and show why they determine different relative homotopic classes of $\mathcal{C}_0\cup \mathcal{C}_1$ with respect to $\Delta_\alpha^0$, $\Delta_\alpha^1$ and $P_3$.
\begin{figure}[!htpb]
 \setlength{\unitlength}{1mm}
  \centering
  \includegraphics[width=\textwidth]{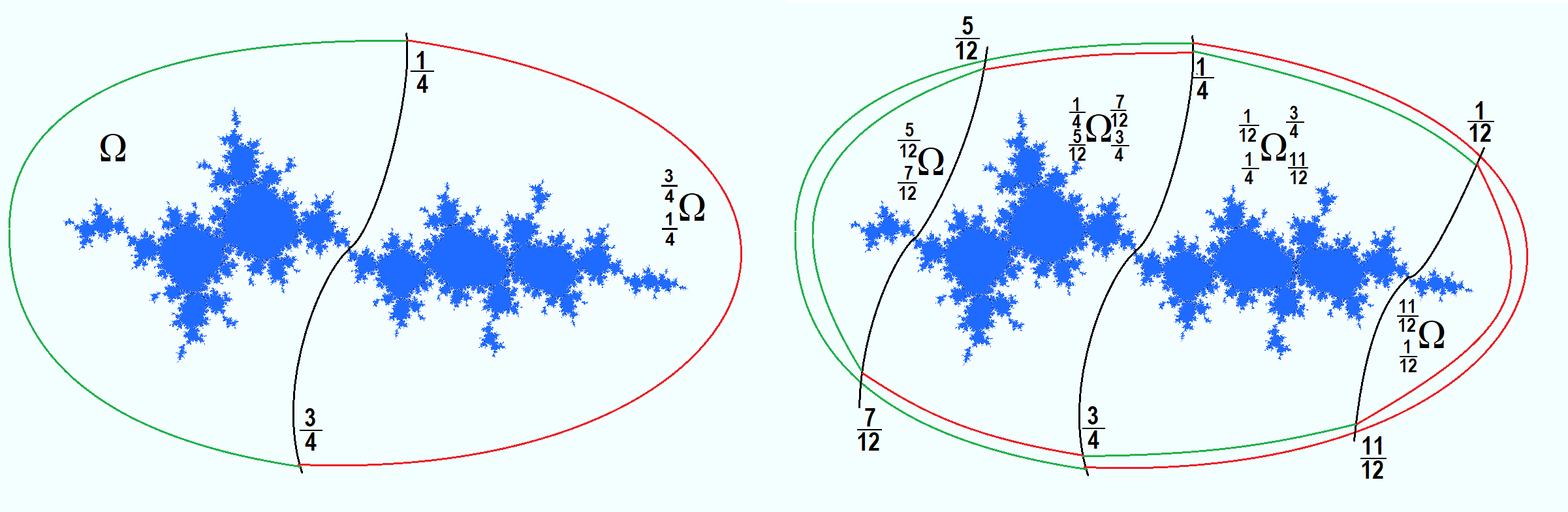}
  \caption{Beginning partitions of the filled-in Julia set of $f_{\alpha}$ by $R_{\frac{1}{4}}\cup \{P_3\}\cup R_{\frac{3}{4}}$ and its first pullback, where $\alpha\approx 0.786776772+0.69912474i$ and the two critical points $c_0$ and $c_1$ lie on the boundaries of $\triangle_{\alpha}^0$ and $\triangle_{\alpha}^1$ respectively. The union of the green and red curves on the left drawing represents an equipotential curve in the attractive basin of $f_\alpha$ at $\infty$, and the extra green and red curves on the right drawing are their preimages.}
  \label{Beginning partition of a Julia set with a parameter on the Gamma curve}
\end{figure}

Now we precisely describe the combinatorial pattern of the components on $\mathcal{C}_0$ by introducing an inductive process to assign labels to all components on this cluster. 
From Lemma \ref{pattens for landing rays}, we know that the two external rays $R_{\frac{1}{4}}$ and $R_{\frac{3}{4}}$ land at the separating fixed point $P_3$ and their union separates $\Delta_\alpha^0$ from $\Delta_\alpha^1$. Denote by $\Omega$ (resp. $_{\frac{1}{4}}^{\frac{3}{4}}\Omega$) the component of $\mathbb{C}\setminus (R_{\frac{1}{4}}\cup \{P_3\}\cup R_{\frac{3}{4}})$ containing $\Delta_{\alpha}^0$ (resp. $\Delta_{\alpha}^1$) (see the left drawing on Figure \ref{Beginning partition of a Julia set with a parameter on the Gamma curve}), and denote by $c_0$ (resp. $c_1$) the critical point on the boundary of $\Delta_{\alpha}^0$ (resp. $\Delta_{\alpha}^1$). 
Clearly, $\Omega$ has two preimage components under $f_\alpha$, which we denote by $_{\frac{7}{12}}^{\frac{5}{12}}\Omega$ and $_{\frac{1}{4}}^{\frac{1}{12}}\Omega_{\frac{11}{12}}^{\frac{3}{4}}$, where $_{\frac{7}{12}}^{\frac{5}{12}}\Omega$ is the region from $R_{\frac{5}{12}}$ to $R_{\frac{7}{12}}$ in the counterclockwise direction and $_{\frac{1}{4}}^{\frac{1}{12}}\Omega_{\frac{11}{12}}^{\frac{3}{4}}$ is the region bounded by the four external rays $R_{\frac{1}{12}}$, $R_{\frac{1}{4}}$, $R_{\frac{3}{4}}$ and $R_{\frac{11}{12}}$. Meanwhile, $_{\frac{1}{4}}^{\frac{3}{4}}\Omega$ has two preimage components under $f_\alpha$ which we denote by $_{\frac{1}{12}}^{\frac{11}{12}}\Omega$ and 
$_{\frac{5}{12}}^{\frac{1}{4}}\Omega_{\frac{3}{4}}^{\frac{7}{12}}$, where $_{\frac{1} {12}}^{\frac{11}{12}}\Omega$ is the region from $R_{\frac{11}{12}}$ and $R_{\frac{1}{12}}$ in the counterclockwise direction and $_{\frac{5}{12}}^{\frac{1}{4}}\Omega_{\frac{3}{4}}^{\frac{7}{12}}$ is bounded by the four external rays $R_{\frac{1}{4}}$, $R_{\frac{5}{12}}$, $R_{\frac{7}{12}}$ and $R_{\frac{3}{4}}$. These four domains are illustrated on the right drawing on Figure \ref{Beginning partition of a Julia set with a parameter on the Gamma curve}. Inductively, further pullbacks of these rays divide the complex plane into more and more pieces of regions. Figure \ref{Julia set with a parameter on the Gamma curve} shows the first and second pullbacks of $R_{\frac{1}{4}}$ and $R_{\frac{3}{4}}$ and the regions separated by all the rays up to the second pullbacks of $R_{\frac{1}{4}}$ and $R_{\frac{3}{4}}$. 
\begin{figure}[!htpb]
 \setlength{\unitlength}{1mm}
  \centering
  \includegraphics[width=0.8\textwidth]{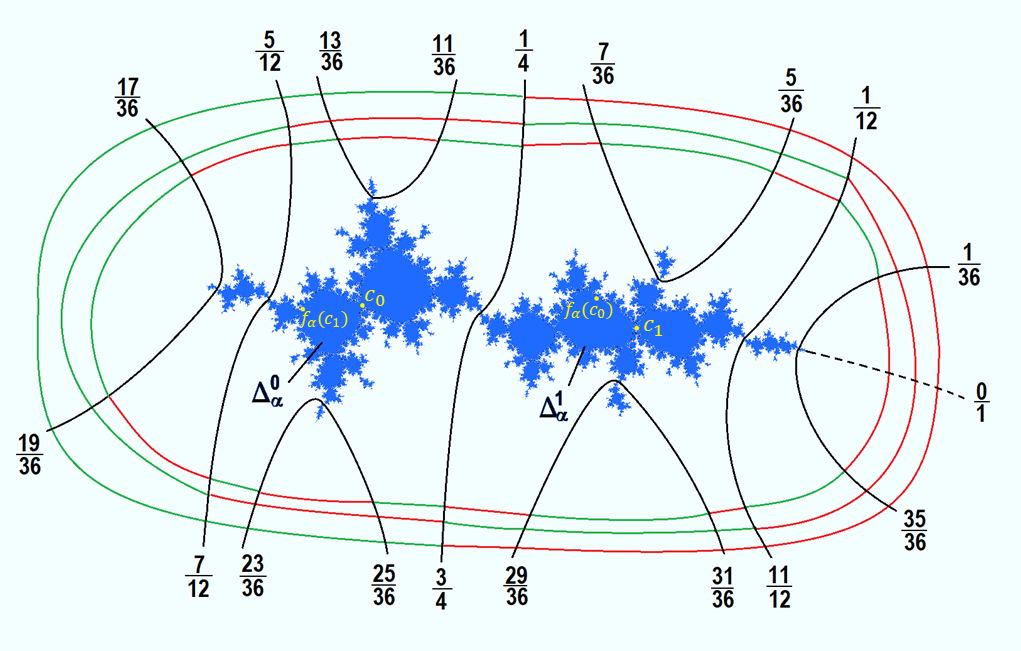}
  \caption{Partition of the filled-in Julia set of $f_{\alpha}$ with $\alpha\approx 0.786776772+0.69912474i$ by the rays up to the second pullbacks of $R_{\frac{1}{4}}\cup \{P_3\}\cup R_{\frac{3}{4}}$.}
  \label{Julia set with a parameter on the Gamma curve}
\end{figure}

Figure \ref{fatouset3} shows how the first and second pullbacks of $R_{\frac{1}{4}}$ and $R_{\frac{3}{4}}$ divide the first and second pullbacks of $\Delta_{\alpha}^0$ and $\Delta_{\alpha}^1$. More precisely, $U=\Delta_{\alpha}^0$ and ${^\frac{3}{4}_\frac{1}{4}}U=\Delta_{\alpha}^1$, and the four components of $f_\alpha^{-1}(U\cup {^\frac{3}{4}_\frac{1}{4}}U)$ other than $U$ and $^\frac{3}{4}_\frac{1}{4}U$ are denoted by $U^0_0$, ${^\frac{3}{4}_\frac{1}{4}}U_0^0$, ${^\frac{11}{12}_\frac{1}{12}}U$ and ${^\frac{5}{12}_\frac{7}{12}}U$. Among those components of
$f_\alpha^{-2}(U\cup {^\frac{3}{4}_\frac{1}{4}}U)$ other than already being labelled, 
if they are obtained by applying the compositions among the two univalent pullbacks of $f_\alpha$ defined on $\Omega$ or $_{\frac{1}{4}}^{\frac{3}{4}}\Omega$ to $U$ or $^\frac{3}{4}_\frac{1}{4}U$, then they are distinguished by the upper and lower indices on the left of $U$ which are the labels of the boundaries of the images of $\Omega$ or $_{\frac{1}{4}}^{\frac{3}{4}}\Omega$ under the compositions; the image of $U^0_0$ under the univalent pullback defined on $\Omega$ is labelled by adding the labels of the boundary of the image of $\Omega$ under the pullback as the upper and lower indices on the left of $U^0_0$; the image of ${^\frac{3}{4}_\frac{1}{4}}U_0^0$ under the univalent pullback defined on $_{\frac{1}{4}}^{\frac{3}{4}}\Omega$ is labelled by replacing the upper and lower indices on the left of ${^\frac{3}{4}_\frac{1}{4}}U_0^0$ by the labels of the boundary of the image of $_{\frac{1}{4}}^{\frac{3}{4}}\Omega$ under the pullback; the remaining components are labelled by $U_0^1$ and $U_0^2$ if they are on $\Omega$ and by  $^\frac{3}{4}_\frac{1}{4}U_0^1$ and $^\frac{3}{4}_\frac{1}{4}U_0^2$ accordingly if they are on $_{\frac{1}{4}}^{\frac{3}{4}}\Omega$. From these discussions, one can see that the eventually Siegel-disk Fatou components of $f_{\alpha}$ are divided by the grand orbit of $R_{\frac{1}{4}}$ (or $R_{\frac{3}{4}}$) into infinitely many groups. 
All eventually Siegel-disk Fatou components in the same group or their union is called a cluster of eventually Siegel-disk Fatou components. 
The two clusters containing $\Delta_\alpha^0$ and $\Delta_\alpha^1$ respectively are crucial, which we may call two {\em crucial} clusters and denote by $\mathcal{C}_0$ and $\mathcal{C}_1$ respectively. All other clusters are univalent pullbacks of the two crucial clusters under $f_\alpha$. In order to obtain a complete description of the combinatorial pattern of all eventually Siegel-disk Fatou components, it suffices to give a clear description of the relative positions of the Fatou components on one of the two crucial clusters.

\begin{figure}[!htpb]
 \setlength{\unitlength}{1mm}
  \centering
  \includegraphics[width=.9\textwidth]{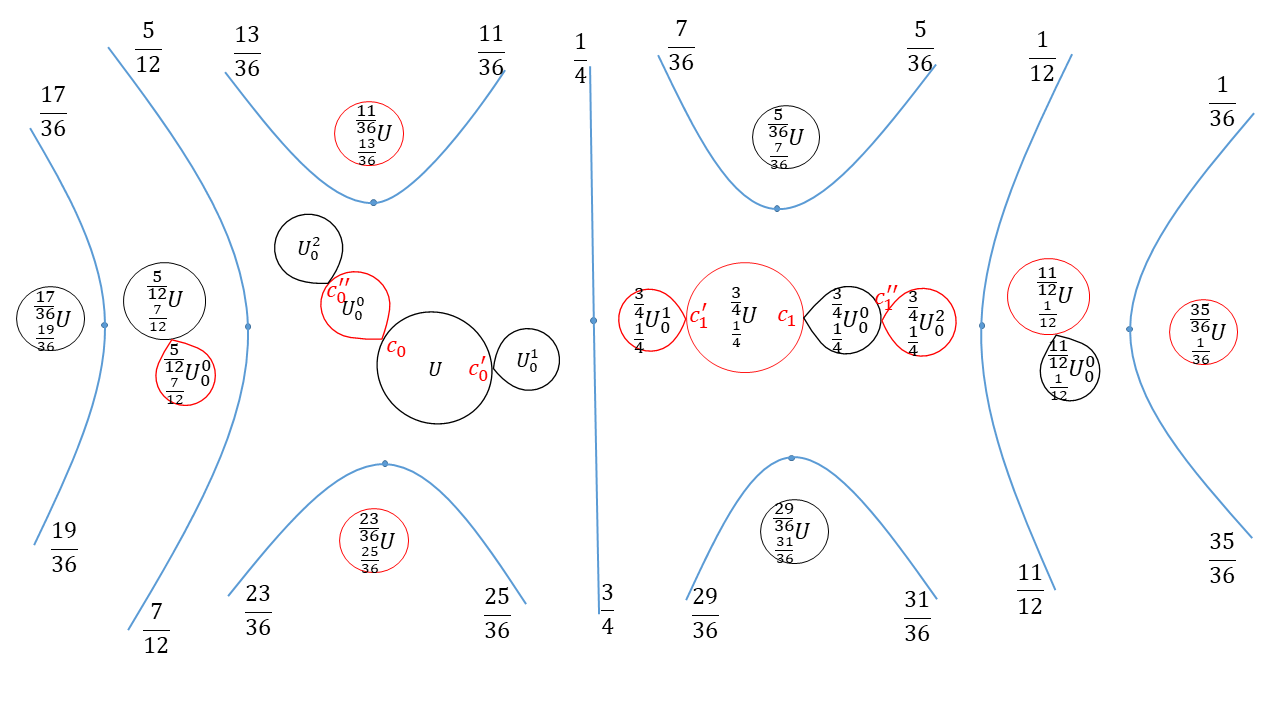}
  \caption{The model for the Fatou components on $\cup_{k=0}^{3}f^{-k}_{\alpha}(\triangle_{\alpha}^0\cup \triangle_{\alpha}^1)$ and the angles of the external rays on the first and second pullbacks of $R_{\frac{1}{4}}$ and $R_{\frac{3}{4}}$, where $c_0\in \partial \triangle_{\alpha}^0 $ and $c_1\in \partial \triangle_{\alpha}^1$.}
  \label{fatouset3}
\end{figure}

\medskip
Without loss of generality, we consider the Fatou components on $\mathcal{C}_0$. Note first that all Fatou components on $\mathcal{C}_0$ are eventually mapped to $\Delta_\alpha^0$ under $f_\alpha^2$. There are two combinatorial patterns for the components of $\mathcal{C}_0$ on the first pullbacks of $\Delta_\alpha^0$ under $f_\alpha^2$ according to $f_\alpha (c_0)\neq c_1$ or $f_\alpha (c_0)=c_1$, which are shown respectively on Figure \ref{Patterns of the first pullbacks} and Figure \ref{Patterns of the first pullbacks when c-1 is mapped to c_2}. Under further pullbacks, we need to consider three cases: (i) $f_\alpha^{\circ (2m+1)}(c_0)\neq c_1$ for any $m\geq 0$, which is the general case; (ii) $f_\alpha (c_0)=c_1$; (iii) $f_\alpha^{\circ (2m+1)}(c_0)=c_1$ for some $m\geq 1$. In fact, these three cases are exactly the ones appearing in the study of the combinatorial patterns of the Siegel disks on a $2$-cycle for a quadratic rational map with the critical points respectively on the boundaries of the two Siegel disks, which have been described in \cite{FYZ20}. 

In the following, we first recapitulate the details of a process in \cite{FYZ20} to label the Fatou components on $\mathcal{C}_0$ in Case (i). 

\begin{figure}[htbp]
  \setlength{\unitlength}{1mm}
  \centering
  \includegraphics[width=0.6\textwidth]{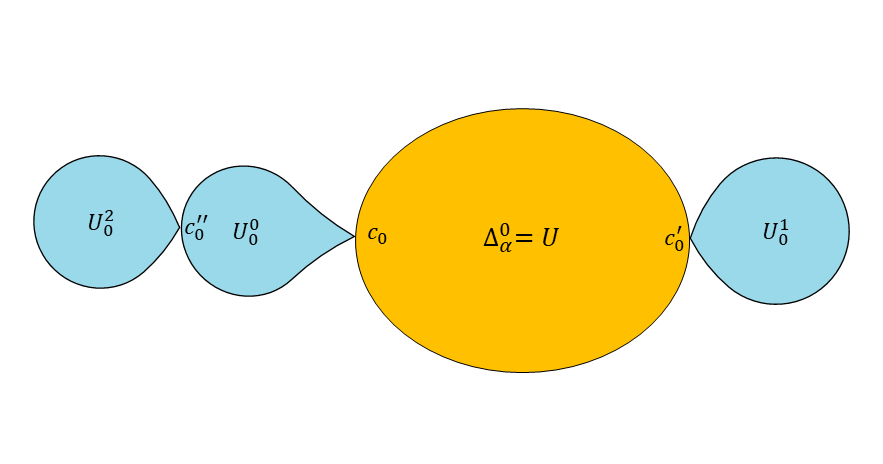}
  \caption{Patterns of the four components of $f_\alpha^{-2}(\Delta_\alpha^0)$ on $\mathcal{C}_0$ when $f_\alpha(c_0)\not=c_1$.}
  \label{Patterns of the first pullbacks}
\end{figure}

\begin{figure}[!htpb]
 \setlength{\unitlength}{1mm}
  \centering
  \includegraphics[width=.8\textwidth]{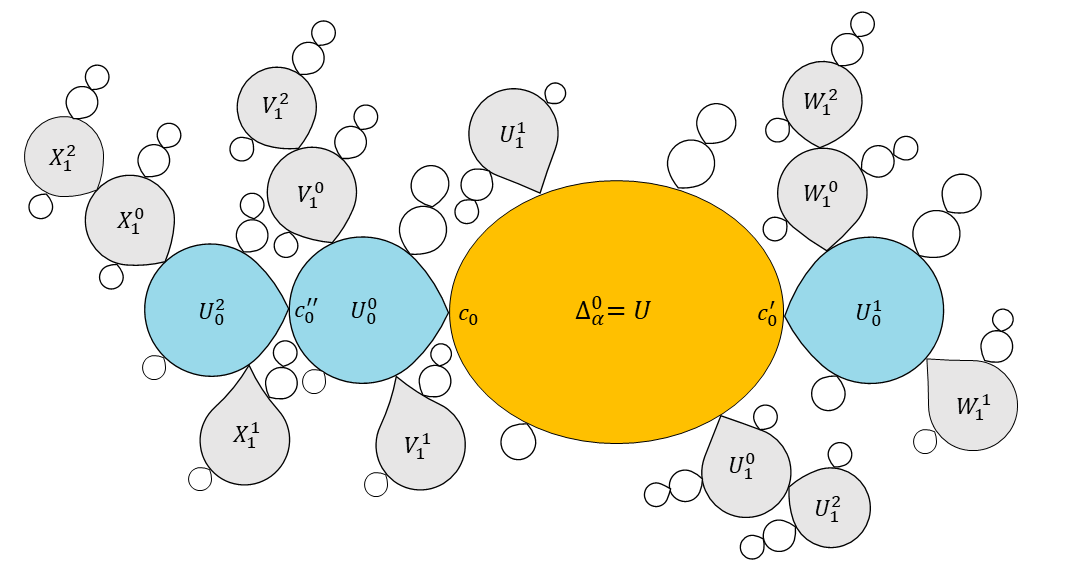}
  \caption{Labels of the components of $f_\alpha^{-4}(\Delta_\alpha^0)$ on $\mathcal{C}_0$ in Case (i).}
  \label{labels of the components up to the second pullbacks}
\end{figure}

Let $U=\Delta_\alpha^0$, $U_0^0$, $U_0^1$ and $U_0^2$ be the labels of the components of $f^{-2}_\alpha(\Delta_\alpha^0)$ on $\mathcal{C}_0$ as shown on Figure \ref{Patterns of the first pullbacks}. These four components are the univalent preimages of $U$ under $f^2_\alpha$ on $\mathcal{C}_0$ and they are connected at three critical points $c_0$, $c_0'$ and $c_0''$ of 
$f^2_\alpha$, where $f_\alpha(c_0')=f_\alpha(c_0'')=c_1$. We will see that the combinatorial pattern of these four domains, the conformal angle between $c_0'$ and $c_0$ on $U$ with respect to the conjugacy of $f^2_\alpha$ on $U$ and the rotation number of $f^2_\alpha$ on $U$ determine uniquely the combitorial pattern of all Fatou components on $\mathcal{C}_0$, which can be described by dividing them into four groups as follows. 

Given integers $s\geq 1$ and $t\in\{0,1,2\}$, let $U_s^t$, $V_s^t$, $W_s^t$ and $X_s^t$ stand for 
the components of $f_\alpha^{-2s}(U_0^t)$ such that
\begin{itemize}
\item $\{U_s^0,U_s^1\}$, $\{V_s^0,V_s^1\}$, $\{W_s^0,W_s^1\}$, $\{X_s^0,X_s^1\}$ are pairs of the two components attached to $U$, $U_0^0$, $U_0^1$ and $U_0^2$ respectively; and
\item $U_s^2$, $V_s^2$, $W_s^2$, $X_s^2$ are attached to $U_s^0$, $V_s^0$, $W_s^0$, $X_s^0$ respectively.
\end{itemize}
For any $m\geq 2$ and any sequences $(s_1,\cdots, s_m)$ and $(t_1,\cdots,t_m)$, where $s_i\geq 1$ and $t_i\in\{0,1,2\}$ with $1\leq i\leq m$, let $U_{s_1,\cdots,s_m}^{t_1,\cdots,t_m}$, $V_{s_1,\cdots,s_m}^{t_1,\cdots,t_m}$, $W_{s_1,\cdots,s_m}^{t_1,\cdots,t_m}$ and $X_{s_1,\cdots,s_m}^{t_1,\cdots,t_m}$ be the unique components of $f_\alpha^{-2(s_1+\cdots +s_m+1)}(\Delta_\alpha^0)$ satisfying:
\begin{itemize}
\item $\{U_{s_1,\cdots,s_{m-1},s_m}^{t_1,\cdots,t_{m-1},0},U_{s_1,\cdots,s_{m-1},s_m}^{t_1,\cdots,t_{m-1},1}\}$ is the pair of the two components attached to $U_{s_1,\cdots,s_{m-1}}^{t_1,\cdots,t_{m-1}}$.
\item $U_{s_1,\cdots,s_{m-1},s_m}^{t_1,\cdots,t_{m-1},2}$ is attached to $U_{s_1,\cdots,s_{m-1},s_m}^{t_1,\cdots,t_{m-1},0}$.
\item The above relationships hold similarly for $V_{s_1,\cdots,s_m}^{t_1,\cdots,t_m}$, $W_{s_1,\cdots,s_m}^{t_1,\cdots,t_m}$ and $X_{s_1,\cdots,s_m}^{t_1,\cdots,t_m}$.
\item If we let $Y$ stand for each of $U$, $V$, $W$ and $X$, then
\begin{equation*}
f_\alpha^{2}(Y_{s_1,s_2,\cdots,s_m}^{t_1,t_2,\cdots,t_m})=
\left\{
\begin{array}{ll}
U_{s_1-1,s_2\cdots,s_m}^{t_1,t_2,\cdots,t_m},  &~~~~~~~\text{when}~s_1\geq 2, \\
V_{s_2,\cdots,s_m}^{t_2,\cdots,t_m},  &~~~~~~~\text{when}~s_1=1 \text{ and } t_1=0, \\
W_{s_2,\cdots,s_m}^{t_2,\cdots,t_m},  &~~~~~~~\text{when}~s_1=1 \text{ and } t_1=1, \\
X_{s_2,\cdots,s_m}^{t_2,\cdots,t_m},  &~~~~~~~\text{when}~s_1=1 \text{ and } t_1=2.
\end{array}
\right.
\end{equation*}
\end{itemize}
Following this procedure, the labels of the Fatou components on $\mathcal{C}_0$ up to the second pullbacks of $U$ under $f^2_\alpha$ are shown on Figure \ref{labels of the components up to the second pullbacks}; the labels of the Fatou components of $f^{-6}_\alpha(\Delta^0_\alpha)$ on $\mathcal{C}_0$ are shown on Figure \ref{Label-of-f^-6}.

\begin{figure}
    \centering
    \includegraphics[width=0.9\linewidth]{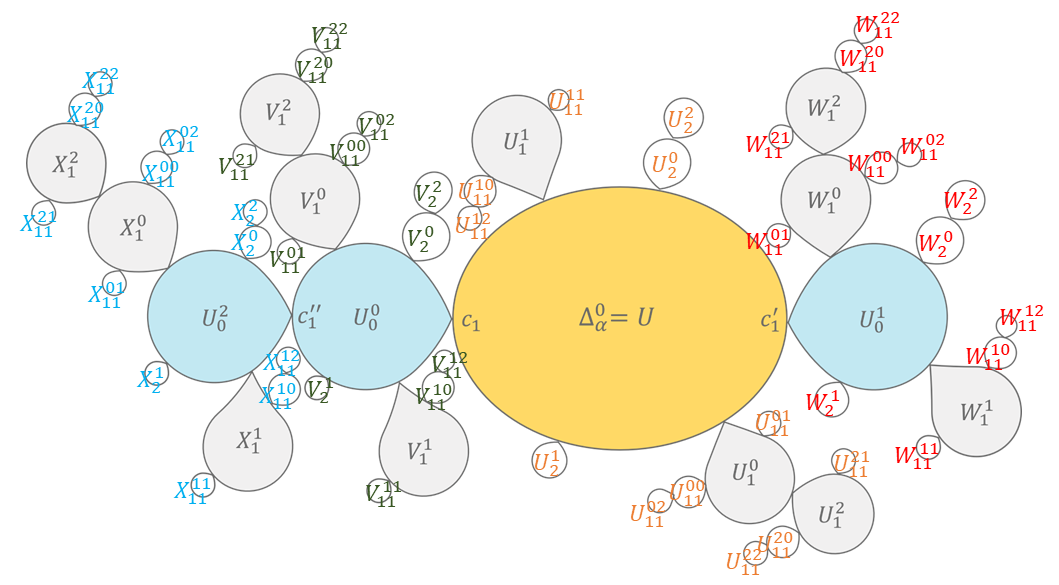}
    \caption{Labels of the components of $f_\alpha^{-6}(\Delta_\alpha^0)$ on $\mathcal{C}_0$ in Case (i).}
    \label{Label-of-f^-6}
\end{figure}

Now we introduce labels to the Fatou components on $\mathcal{C}_0$ in Case (ii). 
\begin{figure}[htbp]
  \setlength{\unitlength}{1mm}
  \centering
  \includegraphics[width=0.8\textwidth]{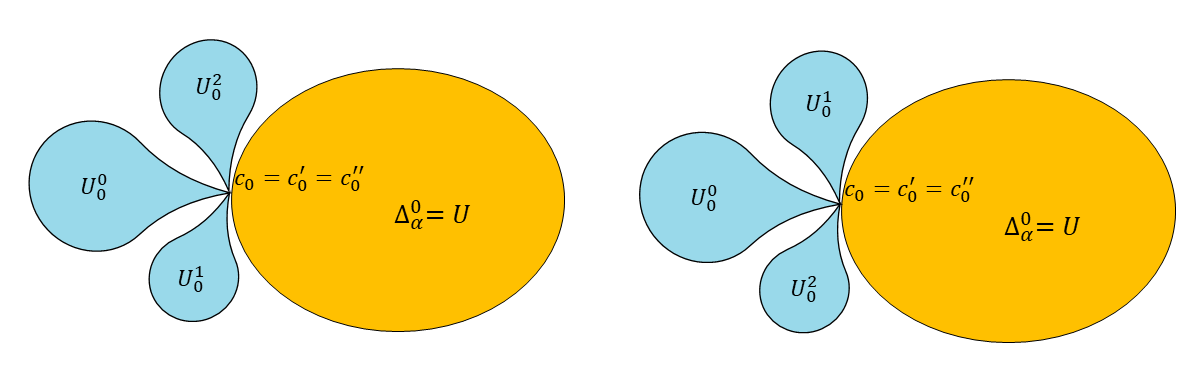}
  \caption{Patterns of the four components of $f_\alpha^{-2}(\Delta_\alpha^0)$ on $\mathcal{C}_0$ when $f_\alpha(c_0)=c_1$.}
  \label{Patterns of the first pullbacks when c-1 is mapped to c_2}
\end{figure}
When $f_\alpha (c_0)=c_1$, there are three Fatou components attached to $\Delta_\alpha^0=U$ at $c_0$. One of them is a preimage of $\Delta_{\alpha}^1$ under $f_\alpha$, which is labelled by $U_0^0$ in Figure \ref{Patterns of the first pullbacks when c-1 is mapped to c_2}, and the other two are the second pullbacks of $\Delta_\alpha^0$ under $f_\alpha$ attached to $\Delta_\alpha^0$ at $c_0'=c_0''=c_0$, which are labelled by $U_0^1$ and $U_0^2$ in two different situations in Figure \ref{Patterns of the first pullbacks when c-1 is mapped to c_2}.
One can see that the labels of the four components of $f_\alpha^{-2}(\Delta_\alpha^0)$ on $\mathcal{C}_0$ shown on the left one in Figure \ref{Patterns of the first pullbacks when c-1 is mapped to c_2} are viewed as the limiting pattern of Figure \ref{Patterns of the first pullbacks} when $c_0''$ (resp. $C_0'$) moves towards $c_0$ along the boundary of $U_0^0$ (resp. $U$) clockwise, and the labels shown on the right one in Figure \ref{Patterns of the first pullbacks when c-1 is mapped to c_2} are viewed as the limiting pattern of Figure \ref{Patterns of the first pullbacks} when $c_0''$ (resp. $C_0'$) moves towards $c_0$ along the boundary of $U_0^0$ (resp. $U$) counter-clockwise. See Figure \ref{labels of the components up to the second pullbacks in Case (ii)} for the labeling of the Fatou components on $\mathcal{C}_0$ up to the second pullbacks of $U$ under $f^2_\alpha$.

\begin{figure}[!htpb]
 \setlength{\unitlength}{1mm}
\includegraphics[width=0.9\textwidth]{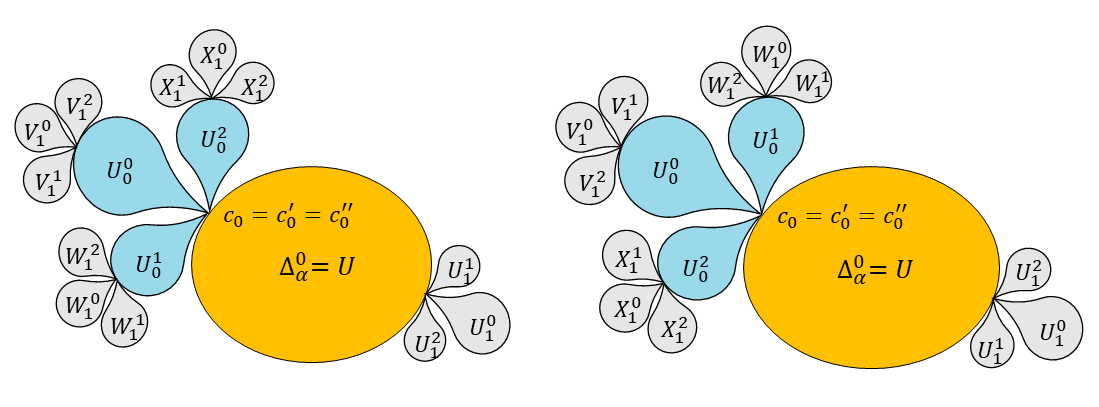}
  \caption{Labels of the components of $f_\alpha^{-4}(\Delta_\alpha^0)$ on $\mathcal{C}_0$ in Case (ii).}
  \label{labels of the components up to the second pullbacks in Case (ii)}
\end{figure}
\begin{rmk}
Note that if $f_\alpha(c_0)=c_1$, then $f^{2j+1}(c_0)\neq c_1$ for any positive integer $j$. Thus, for each pattern of the relative positions among $U$, $U_0^0$, $U_0^1$ and $U_0^2$ given in Figure \ref{Patterns of the first pullbacks when c-1 is mapped to c_2}, the labels of other Fatou components on $\mathcal{C}_0$ are worked out by the same procedure introduced for Case (i). We will also identify the two resulting patterns of the labels of the Fatou components on $\mathcal{C}_0$ from the following viewpoint. Let $\alpha_0$ be a parameter on $\Gamma_{\theta}$ such that $f_{\alpha_0}(c_0)=c_1$ and let $\alpha$ be a parameter on $\Gamma_{\theta}$ such that $f_\alpha^{2j+1}(c_0)\neq c_1$ for any $j\ge 0$. Then when $\alpha$ approaches $\alpha_0$ along $\Gamma_{\theta}$ from different sides of $\alpha_0$, the pattern of the labels of the Fatou components of $f_\alpha$ on $\mathcal{C}_0$ approaches different limiting patterns. Thus, if we hope that the pattern of the labels of the Fatou components of $f_\alpha$ on $\mathcal{C}_0$ changes continuously as the parameter $\alpha$ varies, we need to identify the above two different resulting patterns or we may identify them by simply making a switch between the labels $U_0^1$ and $U_0^2$ and corresponding switches of the labels in further pullbacks.
\end{rmk}

Now we consider Case (iii). If $f_\alpha^{\circ (2m+1)}(c_0)=c_1$ for some $m\geq 1$, then $m$ is the unique moment when $c_0$ lands on $c_1$ under the iteration of $f_\alpha$. Following the inductive process to introduce labels in Case (i), we first see that the first $m$ pullbacks of $U$, under $f^2_\alpha $, on $\mathcal{C}_0$ have the same labels as those components in Case (i). Then there are two more components among the $(m+1)^{th}$ pullbacks of $U$ under $f^2_\alpha$ attached to $\Delta_\alpha^0$ at $c_0$. Similar to Case (ii), there are two different ways to label them as $U_m^1$ and $V_m^1$, which can be viewed as limiting patterns of Case (i) when $f_\alpha^{\circ (2m+1)}(c_0)$ approaches $c_1$ from different sides. After their labels are given, the labeling for the further pullbacks of $U$ (under $f^2_\alpha$) on $\mathcal{C}_0$ continues to follow the same procedure as in Case (i). Furthermore, similar to Case (ii), we identify these two resulting patterns of the labels of the Fatou components on $\mathcal{C}_0$. See Figure \ref{Labels of the components up to the second pullbacks in Case (iii) when m=1} for the labeling of the Fatou components on $\mathcal{C}_0$ up to the second pullbacks of $U$ under $f^2_\alpha $ in Case (iii) when $m=1$.
\begin{figure}[!htpb]
 \setlength{\unitlength}{1mm}
    \includegraphics[width=0.9\textwidth]{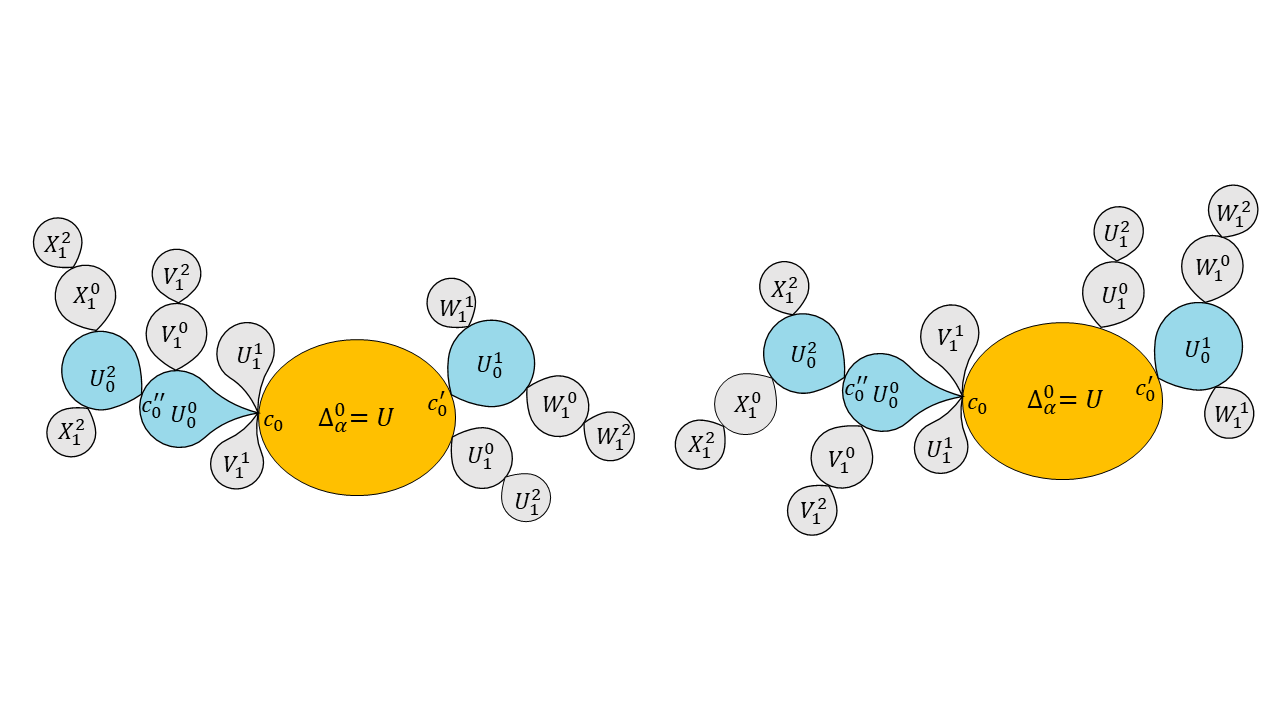}
  \caption{Labels of the components of $f_\alpha^{-4}(\Delta_\alpha^0)$ on $\mathcal{C}_0$ in Case (iii) when $m=1$.}
  \label{Labels of the components up to the second pullbacks in Case (iii) when m=1}
\end{figure}

We have denoted $\Delta_\alpha^0$ by $U$. In the following, we denote $\Delta_\alpha^1$ by $\tilde{U}$. Let us pay attentions to three sequences of Fatou components on $\mathcal{C}_0$ and three corresponding sequences of Fatou components on $\mathcal{C}_1$. By a \emph{Fatou chain} $\Pi$ (resp. $\tilde{\Pi}$) on $\mathcal{C}_0$ (resp. $\mathcal{C}_1$) we mean a sequence of Fatou components on $\mathcal{C}_0$ (resp. $\mathcal{C}_1$) such that (i) it starts with $U$ (resp. $\tilde{U}$) in $\mathcal{C}_0$ (resp. $\mathcal{C}_1$), (ii) the union of their closure is connected but has no loop, and (iii) this union is forward invariant under $f_\alpha^2$. We show in the following lemma that for each $\alpha \in \Gamma_\theta$, there are three Fatou chains on $\mathcal{C}_0$ and three such chains on $\mathcal{C}_1$ converging to repelling fixed points $f_\alpha^2$, and each convergent Fatou chain on $\mathcal{C}_0$ is mapped into a convergent Fatou chain on $\mathcal{C}_1$ by $f_\alpha$.

\begin{lem}\label{convergent Fatou chain} Let $\alpha\in \Gamma_\theta$.
    There are three different Fatou chains on the crucial cluster $\mathcal{C}_0$, which are defined by (\ref{three chains of C_0}) and denoted by $\Pi_r$, $\Pi_g$ and $\Pi_p$ respectively, \footnote{The subscripts r, g and p represent red, green and purple respectively, corresponding to the colors of the convergent Fatou chains shown on Figures \ref{three_chains}, \ref{three_paths}, \ref{degenerate_Fatou_chain} and \ref{three_special_parameters_in_jordan_curve_case}.} and three on the crucial cluster $\mathcal{C}_1$, which are defined by replacing $U$ by $\tilde{U}$ in (\ref{three chains of C_0}) and denoted by $\tilde{\Pi}_r$, $\tilde{\Pi}_g$ and $\tilde{\Pi}_p$ respectively, converging to repelling fixed points of $f_\alpha^2$. Furthermore, exactly one of $\Pi_r$, $\Pi_g$ and $\Pi_p$, depending on $\alpha$, converges to the separate repelling fixed point $P_3$ and exactly one of $\tilde{\Pi}_r$, $\tilde{\Pi}_g$ and $\tilde{\Pi}_p$ converges to $P_3$, where one chain is mapped into the other by $f_\alpha$. Overall, there are exactly three different combinatorial patterns of the two chains converging to $P_3$, which are given as follows:
 \begin{enumerate}
        \item[(i)] $\Pi_r$ and $\tilde{\Pi}_r$ converge to $P_3$, which we call a symmetric connecting.
        \item[(ii)] $\Pi_g$ and $\tilde{\Pi}_p$ converge to $P_3$, which we call an asymmetric connecting.  
        \item[(iii)] $\Pi_p$ and $\tilde{\Pi}_g$ converge to $P_3$, which we call the conjugate of the asymmetric connecting. 
    \end{enumerate}
\end{lem}
\begin{figure}[htbp]
  \setlength{\unitlength}{1mm}
  \includegraphics[width=\textwidth]{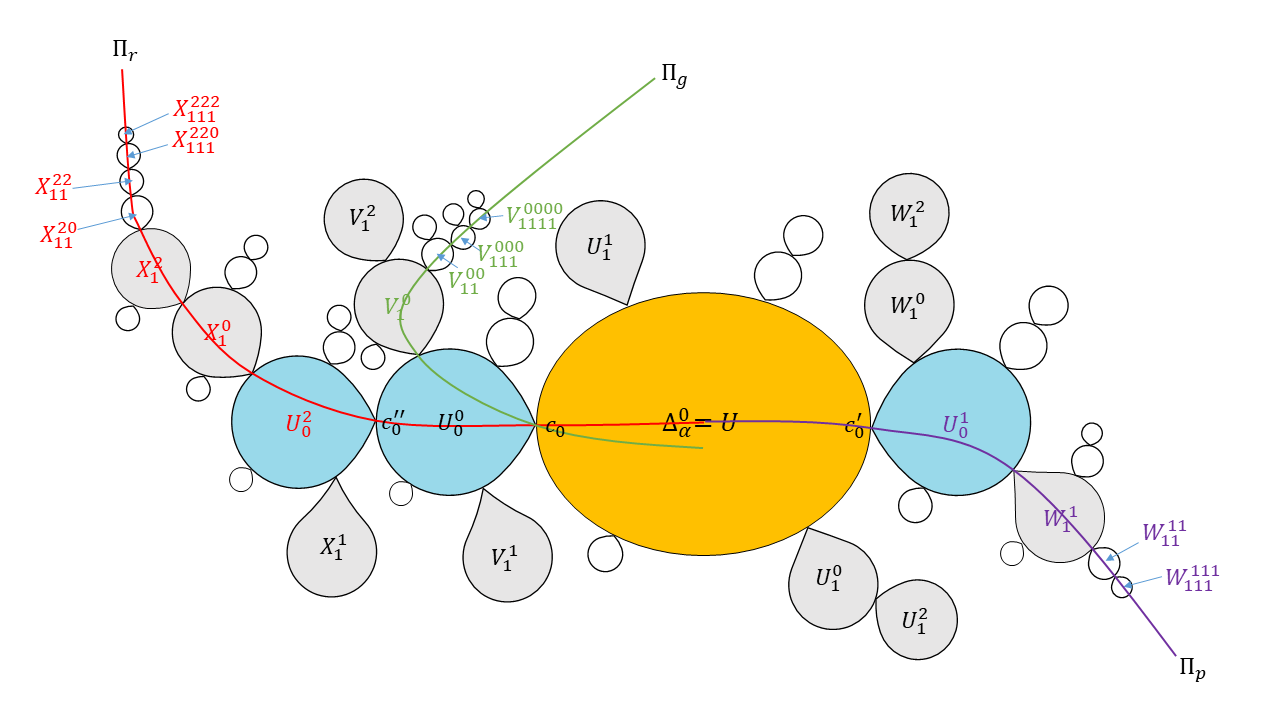}
  \caption{Three convergent Fatou chains on $\mathcal{C}_0$}
  \label{three_chains}
\end{figure}
Let us first state a result on the Julia set $J(f_\alpha)$ before we prove this lemma. 
\begin{thm}\label{thm for local-connectivity}\cite{WYZZ22}
Suppose that $f$ is a rational map with Siegel disks of bounded types and the Julia set $J(f)$ is connected. Assume that the forward orbit of every critical point of $f$ satisfies one of the following:
\begin{enumerate}
\item It is finite; 
\item It lies in an attracting cycle; 
\item It intersects the closure of a bounded type Siegel disk.
\end{enumerate}
Then $J(f)$ is locally connected.
\end{thm}
\begin{cor}\label{local connectivity}
For any $\alpha\in \Gamma_\theta\cup\Gamma_\theta^0\cup\Gamma_\theta^1$, $J(f_\alpha)$ is connected and locally connected. 
\end{cor}
\begin{proof}[Proof of Lemma \ref{convergent Fatou chain}]
Let $\alpha\in \Gamma_\theta$. 

We first prove this result for the general case; that is, $f_\alpha^{\circ (2m+1)}(c_0)\neq c_1$ for any $m\geq 0$. By considering how $f_\alpha^2$ maps the Fatou components on $\mathcal{C}_0$, we can see that 
\begin{equation}\label{three chains of C_0}
\begin{split}
        &\Pi_r=\{U,U_0^0,U_0^2,X_1^0,X_1^2,X_{11}^{20},X_{11}^{22},X_{111}^{220},X_{111}^{222},X_{1111}^{2220},X_{1111}^{2222},\cdots\},\\
        &\Pi_g=\{U,U_0^0,V_1^0,V_{11}^{00},V_{111}^{000},V_{1111}^{0000},\cdots\}, \text{ and}\\
        &\Pi_p=\{U,U_0^1,W_1^1,W_{11}^{11},W_{111}^{111},W_{1111}^{1111},\cdots\}  
        \end{split}.
\end{equation}
are three Fatou chains on $\mathcal{C}_0$. See Figure \ref{three_chains}. By Corollary \ref{local connectivity}, the Julia set of $f_\alpha$ is locally connected. Thus, $\Pi_r, \Pi_g$ and $\Pi_p$ converge to some fixed points of $f_\alpha^2$. We also know they converge to three distinct fixed points of $f_\alpha^2$; otherwise, their union contain loops of Fatou components in the filled-in Julia set of $f_\alpha$. Suppose none of $\Pi_r, \Pi_g$ and $\Pi_p$ converges to $P_3$. Then including $\{0,1\}$, there will be four $2$-cycles of periodic points under $f_\alpha$, which contradicts that $f_\alpha$ has at most three $2$-cycles of periodic points. Thus, one of them converges to $P_3$. 

Note that the combinatorial pattern among $U, U_0^0, U_0^1$ and $U_0^2$ and the elements of $\{U, U_0^0, U_0^1, U_0^2\}$ lying on each of the Fatou chains $\Pi_r, \Pi_g$ and $\Pi_p$ determine each chain and distinguish it from the other two. See Figure \ref{three_chains}.

\begin{figure}[htbp]
  \setlength{\unitlength}{1mm}
  \subfigure[{ $\Pi_r\to P_3$ and $\tilde{\Pi}_r\to P_3$.}]
  {\includegraphics[width=.9\textwidth]{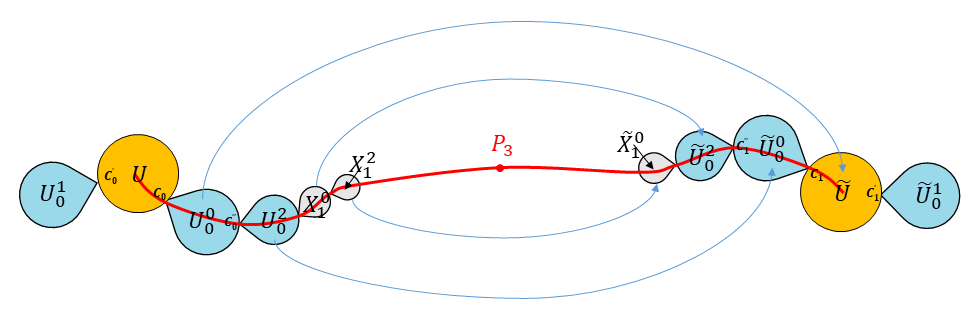}}
  \subfigure[{$\Pi_g\to P_3$ and $\tilde{\Pi}_p\to P_3$.}]
  {\includegraphics[width=.9\textwidth]{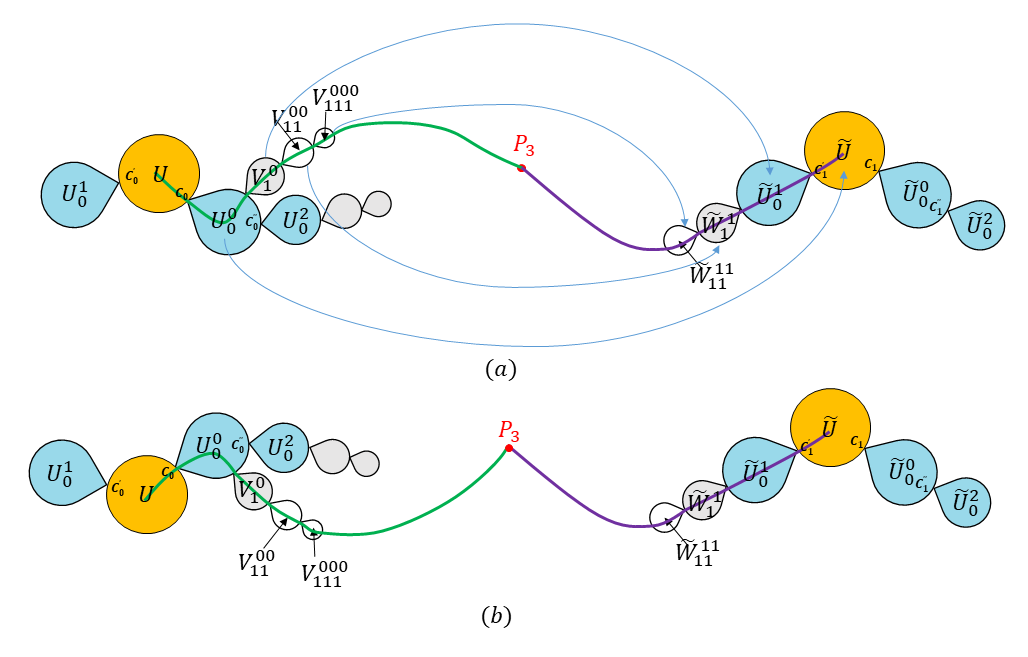}}
  \subfigure[{$\Pi_p\to P_3$ and $\tilde{\Pi}_g\to P_3$.}]
  {\includegraphics[width=.9\textwidth]{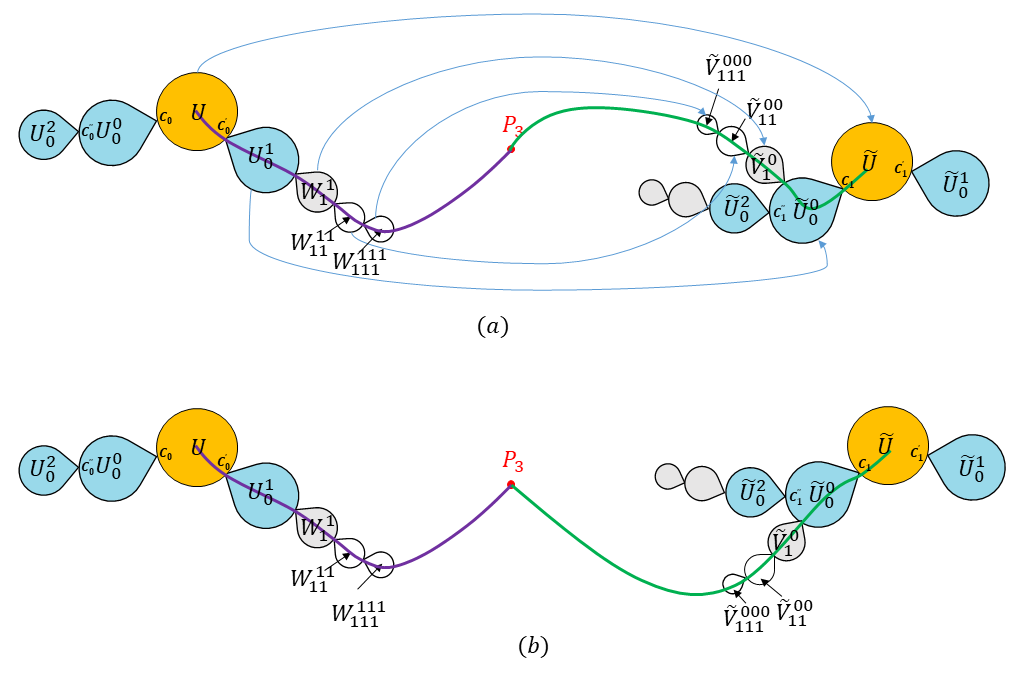}}
  \caption{Three different patterns for two Fatou chains converging to $P_3$.}
  \label{three_paths}
\end{figure}
By applying $f_\alpha$ to a Fatou chain on $\mathcal{C}_0$ converging to $P_3$, we obtain a Fatou chain on $\mathcal{C}_1$ converging to $P_3$ as well. Furthermore, we
can see that there are three combinatorial patterns for the two Fatou chains converging to $P_3$.

Case (i): $\Pi_r\to P_3$. 

Since $f_\alpha :$ $X_1^2\to \tilde{X}_1^0$,  $X_1^0\to \tilde{U}_0^2$, $U_0^2\to \tilde{U}_U^0$, and $U_0^0\to \tilde{U}$, it follows that $\tilde{\Pi}_r\to P_3$. See (i) on Figure \ref{three_paths}.

Case (ii): $\Pi_g\to P_3$. 

Since $f_\alpha :$ $V_{111}^{000}\to \tilde{W}_{11}^{11}$,  $V_{11}^{00}\to \tilde{W}_{1}^{1}$, $V_{1}^{0}\to \tilde{U}_{0}^{1}$, and $U_0^0\to \tilde{U}$, it follows that $\tilde{\Pi}_p\to P_3$. See (ii) on Figure \ref{three_paths}.

Case (iii): $\Pi_p\to P_3$. 

Since $f_\alpha :$ $W_{111}^{000}\to \tilde{V}_{111}^{000}$,  $W_{11}^{11}\to \tilde{V}_{11}^{00}$, $W_{1}^{1}\to \tilde{V}_{1}^{0}$, $U_0^1\to \tilde{U_0^0}$, and $U\to \tilde{U}$, it follows that $\tilde{\Pi}_g\to P_3$. See (iii) on Figure \ref{three_paths}.

Now we show the existence of three different cases of the Fatou chains on $\mathcal{C}_0$ converging to $P_3$ and the corresponding three cases of the Fatou chains on $\mathcal{C}_1$ converging to $P_3$ for the case when $f_\alpha^{\circ (2m+1)}(c_0)=c_1$ or $f_\alpha^{\circ (2m+1)}(c_1)=c_0$ for some $m\geq 0$.   
Let us focus on showing the existence for the case when $f_\alpha^{\circ (2m+1)}(c_0)=c_1$ for some $m\geq 0$. The existence for the other case comes out of a symmetric argument. 
\begin{figure}[htbp]
  \setlength{\unitlength}{1mm}
  \subfigure[$\Pi_r^{d_0}$ is obtained by moving $c_0'$ and $c_0''$ to $c_0$ clockwise and $\Pi_p^{d_0}$ is obtained by moving $c_0'$ and $c_0''$ to $c_0$ counterclockwise.]{\includegraphics[width=.8\textwidth]{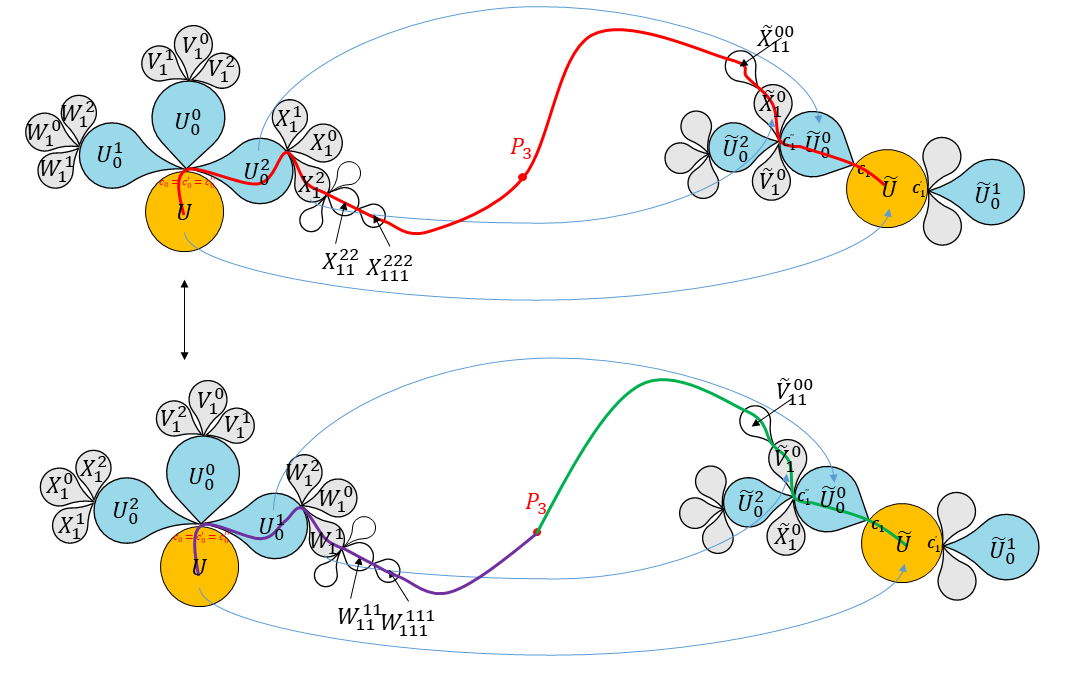}}
  \subfigure[$\Pi_r^{d_0}$ is obtained by moving $c_0'$ and $c_0''$ to $c_0$ counterclockwise and $\Pi_p^{d_0}$ is obtained by moving $c_0'$ and $c_0''$ to $c_0$ clockwise.]{\includegraphics[width=.8\textwidth]{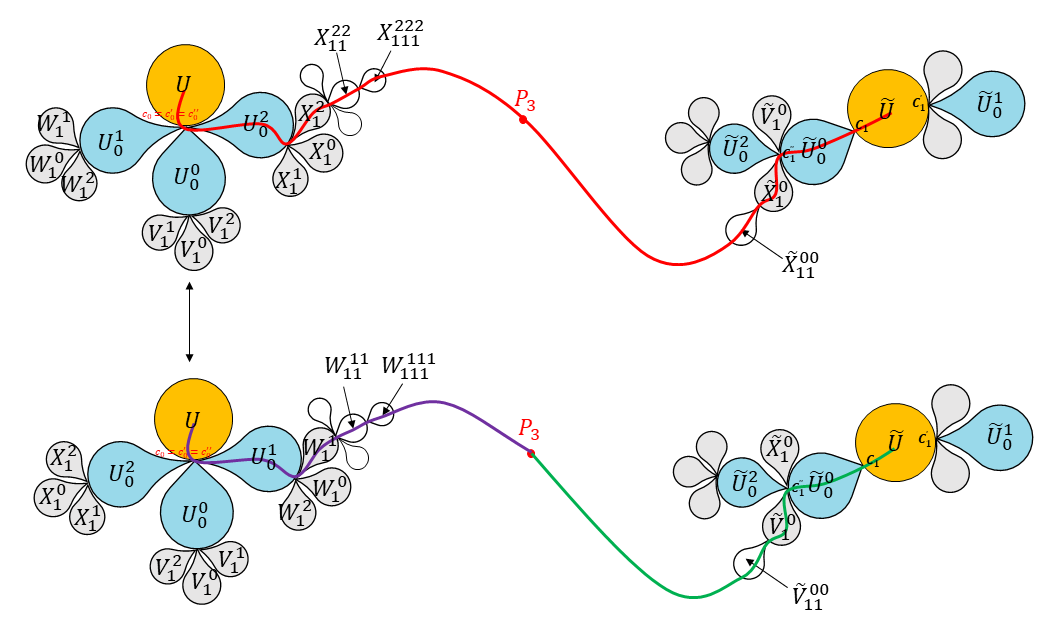}}
   \subfigure[$\Pi_g^{d_0}$ is obtained by moving $c_0'$ and $c_0''$ to $c_0$ clockwise(resp. counterclockwise).]{\includegraphics[width=.8\textwidth]{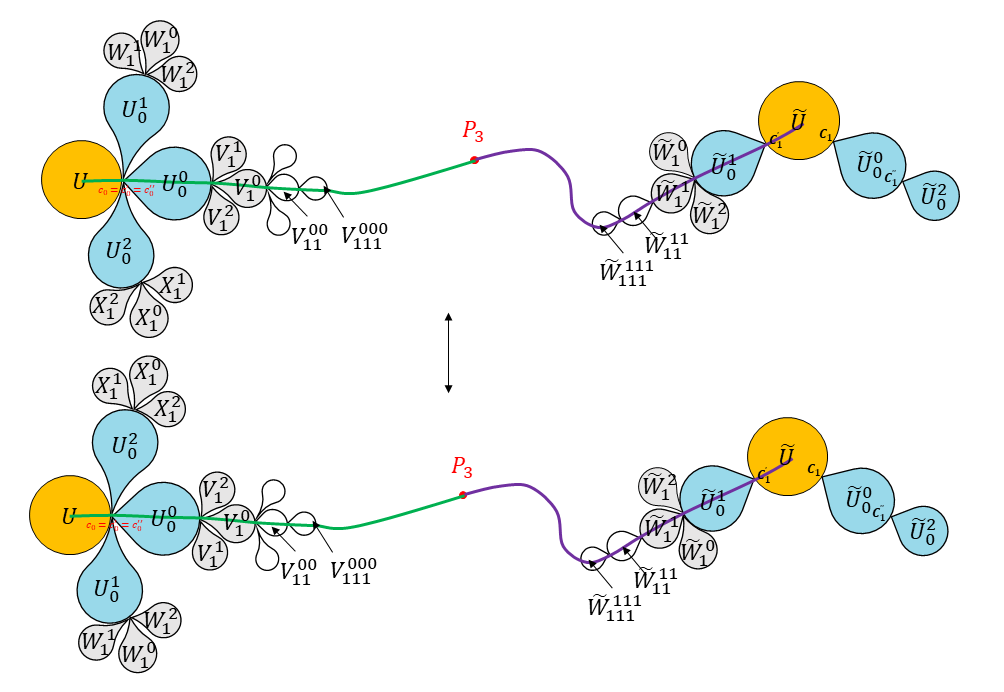}}
  \caption{Three different combinatorial patterns of the Fatou chains connected at  $P_3$ when $f_{\alpha}(c_0)=c_1$.}
  \label{degenerate_Fatou_chain}
\end{figure}

Let us first consider $\Pi_r$ for $f_\alpha^{\circ (2m+1)}(c_0)=c_1$ as a limit of $\Pi_r$ given by (\ref{three chains of C_0}) when  $f_\alpha^{\circ (2m+1)}(c_0)$ approaches $c_1$ along $\partial\Delta_\alpha^1$, which we call a degenerate Fatou chain and denote by $\Pi_r^{d_m}$. 
More precisely, when $m=0$, $c_0''$ moves toward $c_0$ along $\partial U_0^0$ as $f(c_0)$ approaches $c_1$ along $\partial\Delta_\alpha^1$. So we may view the red arc on $U^0_0$, which connects $c_0$ with $c_0''$,  shrinking to a point as $c_0''$ moves toward $c_0$ along $\partial U_0^0$. Thus, $U^0_0$ ends up not on $\Pi_r^{d_0}$.
This means that after removing $U^0_0$ and each of its pullback components on (\ref{three chains of C_0}) under $f_\alpha^2$, we obtain the components with their labels on $\Pi_r^{d_0}$. Note that under the condition $f_\alpha^{\circ (2m+1)}(c_0)=c_1$, the labels of the components of $\Pi_g$ and $\Pi_p$ are the same as the ones given in (\ref{three chains of C_0}), but for sake of using a uniform notation for all possible Fatou chains converging to $P_3$ under this condition as the one used for $\Pi_r^{d_m}$, 
we will denote them by $\Pi_g^{d_m}$ and $\Pi_p^{d_m}$. In particular, we obtain 
\begin{equation*}\label{three chains of C_0 when m=0}
\begin{split}
        &\Pi_r^{d_0}=\{U,U_0^2,X_1^2,X_{11}^{22},X_{111}^{222},X_{1111}^{2222},\cdots\},\\
        &\Pi_g^{d_0}=\{U,U_0^0,V_1^0,V_{11}^{00},V_{111}^{000},V_{1111}^{0000},\cdots\}, \text{ and}\\
        &\Pi_p^{d_0}=\{U,U_0^1,W_1^1,W_{11}^{11},W_{111}^{111},W_{1111}^{1111},\cdots\}.  
        \end{split}
\end{equation*}
If $m>0$, then the sequence of the labels of the components on $\Pi_r^{d_m}$ can be obtained from $\Pi_r$ by removing $X_{1,1,\cdots,1} ^{2,2,\cdots,0}$, where all $m$ lower indices equal to $1$, and removing each of its pullback components on $\Pi_r$ under $f_\alpha^2$. 

Note that when going through the above process to find the components (with their labels) on $\Pi_r^{d_m}$, there are two different ways to have the involved red arc shrinking to a point since $f_\alpha^{\circ (2m+1)}(c_0)$ may approach $c_1$ (along the boundary of $\Delta_\alpha^1$) clockwise or counterclockwise. Although we obtain the same labels of the components on $\Pi_r^{d_m}$, the resulting $\Pi_r^{d_m}$ has different relative patterns among $\mathcal{C}_0$. For example, considering the chains converging to $P_3$ in three different cases shown on Figure \ref{three_paths}, the resulting $\Pi_r^{d_0}$ has two different relative patterns among $\mathcal{C}_0$ shown as the Fatou components on the left side of $P_3$ and connected by the red curve in (i) and (ii) on Figure \ref{degenerate_Fatou_chain}, where they are different in the sense that the two marked combinatorial patterns of $\overline{\Pi_r^{d_0}\cup U\cup U^0_0\cup U^1_0\cup U_0^2}$ are not homotopic to each other by a continuous family of orientation-preserving homeomorphisms rel $U$ and $P_3$. 
On Figure \ref{degenerate_Fatou_chain}, we also present two relative patterns of $\Pi_p^{d_0}$ and two relative patterns of $\Pi_g^{d_0}$ among $\mathcal{C}_0$, and the corresponding relative patterns of $f_\alpha(\Pi_r^{d_0})$, $f_\alpha(\Pi_p^{d_0})$ and $f_\alpha(\Pi_g^{d_0})$ among $\mathcal{C}_1$.
Note also that one relative pattern of $\Pi_g^{d_0}$ among $\mathcal{C}_0$ is obtained by moving $c_0'$ and $c_0''$ in (a) of (ii) on Figure \ref{three_paths} to $c_0$ clockwise and the other is obtained by moving $c_0'$ and $c_0''$ in (b) of (ii) on Figure \ref{three_paths} to $c_0$ counterclockwise.

The Fatou chains and their images under $f_\alpha $ on Figure \ref{degenerate_Fatou_chain} show several important properties, which we will use in the last section to prove that $\Gamma_\theta$ is a Jordan curve. 

{\bf Property 1.} One relative pattern of $\Pi_r^{d_0}$ among $\mathcal{C}_0$ is equivalent to one relative pattern of $\Pi_p^{d_0}$ among $\mathcal{C}_0$ in the sense that their relative patterns are the same by switching the labels of  $U_0^1$ with $U_0^2$ and correspondingly switching the labels of the pullback components through the labeling algorithm. Under this sense, the relative patterns of $\Pi_r^{d_0}$ and $\Pi_p^{d_0}$ in (i) (resp. (ii)) among $\mathcal{C}_0$ on Figure \ref{degenerate_Fatou_chain} are equivalent. Similarly, two relative patterns of $\Pi_g^{d_0}$ among $\mathcal{C}_0$ presented in (iii) on Figure \ref{degenerate_Fatou_chain} are also equivalent. 

{\bf Property 2.} Two relative patterns of 
$\Pi_r^{d_0}$ among $\mathcal{C}_0$ are different, which are determined by the relative patterns between $\Pi_r^{d_0}$ and the four components $U$, $U_0^0$, $U_0^1$ and $U_0^2$, in the sense that $\overline{\Pi_r^{d_0}\cup U\cup U^0_0\cup U_0^1\cup U_0^2}$ are not homotopic by a continuous family of orientation-preserving homeomorphisms rel $U$ and $P_3$. Therefore, the filled-in Julia sets realizing these two relative patterns of $\Pi_r^{d_0}$ among $\mathcal{C}_0$ are not homotopic to each other by a family of orientation-preserving homeomorphisms rel $U$ and $P_3$.

{\bf Property 3.} For each $\Pi_r^{d_0}$, 
the pattern of the labels of the components on $f_\alpha(\Pi_r^{d_0})$ follows the pattern of the labels of the components on $\tilde{\Pi}_g$ by identifying $\tilde{X}$ with $\tilde{V}$.

{\bf Property 4.} Two relative patterns of $\Pi_r^{d_0}\cup f_\alpha(\Pi_r^{d_0})$ 
among $\mathcal{C}_0\cup \mathcal{C}_1$ are respectively equivalent to two relative patterns of $\Pi_r^{d_0}\cup \tilde{\Pi}_g$ among $\mathcal{C}_0\cup \mathcal{C}_1$ by identifying $\tilde{X}$ with $\tilde{V}$, which are respectively presented in (i) and (ii) on Figure \ref{degenerate_Fatou_chain}.

{\bf Property 5.} The two relative patterns of $\Pi_r^{d_0}$ among $\mathcal{C}_0$ (i.e., the two relative patterns of $\Pi_p^{d_0}$ among $\mathcal{C}_0$) and one relative pattern of $\Pi_g^{d_0}$ among $\mathcal{C}_0$ imply that the filled-in Julia sets in these three cases are not homotopic to each other by a continuous family of orientation-preserving homeomorphisms rel $U$ and $P_3$. 

{\bf Property 6.} Assume that $f_\alpha^{2m+1}(c_0)=c_1$ for some $m>0$. Then the color of the chain converging to $P_3$ is preserved in a small neighborhood of $\alpha$. Furthermore, if there are two such parameters $\alpha$ and $\alpha'$ that the Fatou chain $\Pi(\alpha)$ of $f_\alpha$ converging to $P_3(\alpha)$ has color different from the Fatou chain $\Pi(\alpha')$ of $f_{\alpha'}$ converging to $P_3(\alpha')$, then the relative pattern between $\Pi(\alpha)$ and $U(\alpha)$, $U_0^0(\alpha)$, $U_0^1(\alpha)$ and $U_0^2(\alpha)$ is different from the one between $\Pi(\alpha')$ and 
$U(\alpha')$, $U_0^0(\alpha')$, $U_0^1(\alpha')$ and $U_0^2(\alpha')$. Thus, the filled-in Julia sets of $f_\alpha$ and $f_{\alpha'}$ are not homotopic to each other by a continuous family of orientation-preserving homeomorphisms rel $U$ and $P_3$. 

{\bf Property 7.} When $f_\alpha^{2m+1}(c_1)=c_0$ ($m\ge 0$), there is a similar process, symmetric to the one for the case $f_{\alpha}^{2m+1}(c_0)=c_1$ ($m\ge 0$), to understand the labels of the components on the degenerate Fatou chains of $\tilde{\Pi}_r$, $\tilde{\Pi}_p$ and $\tilde{\Pi}_g$ among $\mathcal{C}_1$, which we denote by $\tilde{\Pi}_r^{d_m}$, $\tilde{\Pi}_p^{d_m}$ and $\tilde{\Pi}_g^{d_m}$ respectively.
The above Properties 1-5 have corresponding versions for $\tilde{\Pi}_r^{d_0}$, $\tilde{\Pi}_p^{d_0}$ and $\tilde{\Pi}_g^{d_0}$, and the above Property 6 has a corresponding version for $\tilde{\Pi}_r^{d_m}$, $\tilde{\Pi}_p^{d_m}$ and $\tilde{\Pi}_g^{d_m}$ when $m>1$. 
\end{proof}

\begin{figure}[!htpb]
 \setlength{\unitlength}{1mm}
  \centering
  \includegraphics[width=1\textwidth]{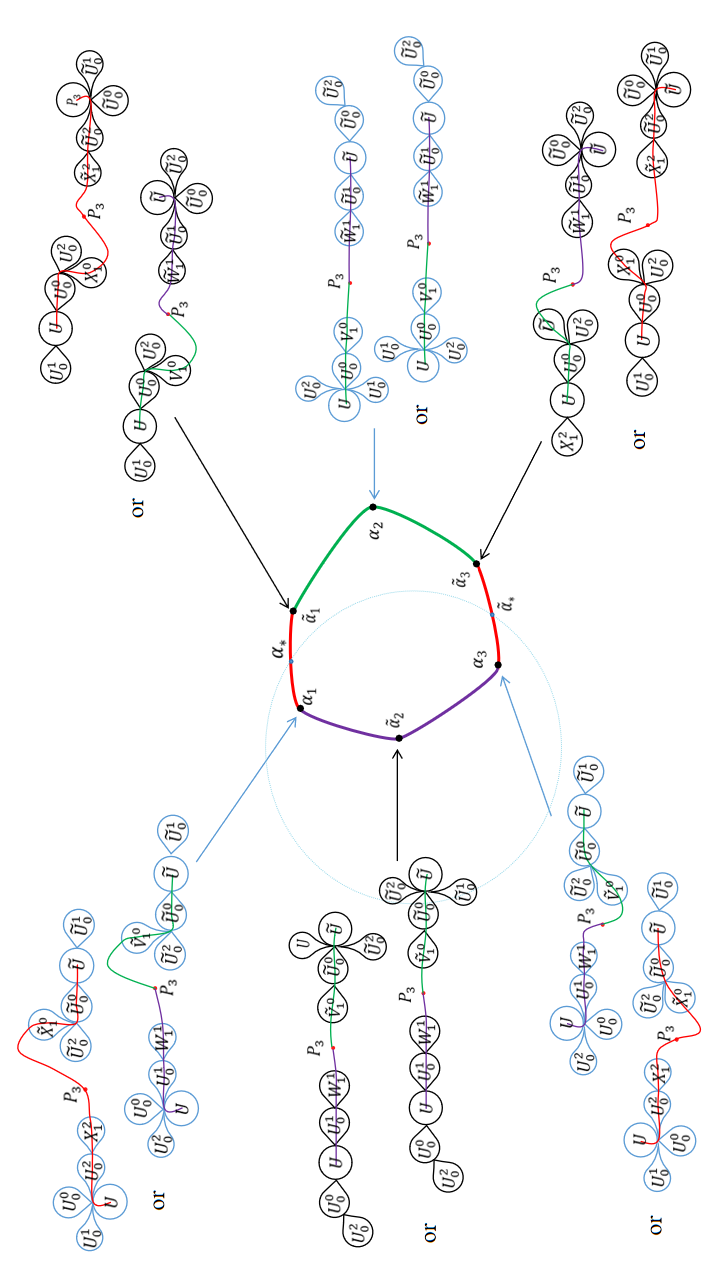}
  \caption{How do the patterns of $\Pi$ and $f_\alpha(\Pi)$ vary when $\alpha$ moves along a loop component $\Gamma_\theta'$ of $\Gamma_\theta$ (if exists)?}
  \label{three_special_parameters_in_jordan_curve_case}
\end{figure}
\begin{rmk}\label{rmk:global loop}
(a) By the symmetry of $\Gamma_\theta$ (Lemma \ref{lem:sym}), we know that the two fixed points of the map $\alpha\mapsto \frac{\lambda}{\alpha}$ belong to $\Gamma_\theta$, which we denote by $\alpha_*$ and $\tilde{\alpha}_*$. Note also that the Fatou chains of $f_{\alpha_*}$ (resp. $f_{\tilde{\alpha}_*}$) connected at $P_3$ are both red chains $\Pi_r$ and $\tilde{\Pi}_r$. 

(b) The above Property 5 indicates that there are possibly three different relative patterns of the Fatou chain landing at $P_3$ among $\mathcal{C}_0$ (and hence among $\mathcal{C}_0\cup \mathcal{C}_1$) such that $f_\alpha(c_0)=c_1$. Correspondingly, there are possibly three different relative patterns of the Fatou chain landing at $P_3$ among $\mathcal{C}_1$ (and hence among $\mathcal{C}_1\cup \mathcal{C}_0$) such that $f_\alpha(c_1)=c_0$. In Section \ref{jordancurve}, we will first prove that if there is a parameter $\alpha$ on $\Gamma_\theta$ such that the Fatou chain of $f_\alpha$ converging to $P_3$ presents one of these six relative patterns among $\mathcal{C}_0\cup\mathcal{C}_1$, then such a parameter is unique; we will prove secondly that $\Gamma_\theta$ contains a loop component separating $0$ from $\infty$, which we denote by $\Gamma_\theta'$. 

(c) Once we know that $\Gamma_\theta$ contains a loop component $\Gamma_\theta'$, it is easy to see that the loop component $\Gamma_\theta'$ must contain at least one parameter $\alpha $ such that $f_\alpha(c_0)=c_1$ or $f_\alpha(c_1)=c_0$. Then the above seven properties of the Fatou chains converging to $P_3$, including the properties on their relative patterns among $\mathcal{C}_0\cup\mathcal{C}_1$, imply that each of the six possible patterns of the Fatou chains landing at $P_3$ when $f_\alpha(c_0)=c_1$ or $f_\alpha(c_1)=c_0$ can be realized by a unique parameter on $\Gamma_\theta'$. Let us denote by $\alpha_1$, $\alpha_2$ and $\alpha_3$ the three parameters $\alpha$ such that $f_\alpha(c_0)=c_1$ and by $\tilde{\alpha}_1$, $\tilde{\alpha}_2$ and $\tilde{\alpha}_3$ the three parameters $\alpha$ such that $f_\alpha(c_1)=c_0$. By Lemma \ref{lem:sym}, we may arrange $\tilde{\alpha}_k=\frac{\lambda}{\alpha_k}$ for $k=1, 2, 3$. Assume that for these six parameters, the corresponding Fatou chains landing at $P_3$ are respectively given by the ones presented on Figure \ref{three_special_parameters_in_jordan_curve_case}. Then the above seven properties of the Fatou chains landing at $P_3$ also imply that the six parameters $\alpha_1$, $\alpha_2$, $\alpha_3$, $\tilde{\alpha}_1$, $\tilde{\alpha}_2$ and $\tilde{\alpha}_3$ must lie on $\Gamma_\theta'$ in the order given on Figure \ref{three_special_parameters_in_jordan_curve_case},  
and furthermore,  one of $\alpha_*$ and $\tilde{\alpha}_*$ lies on $\Gamma_\theta'$ between $\alpha_1$  and $\tilde{\alpha}_1$ and the other lies on $\Gamma_\theta'$ between $\tilde{\alpha}_3$ and $\alpha_3$ in the clockwise order, which we denote by $\alpha_*$ and $\tilde{\alpha}_*$ respectively. 
Moreover, when the parameter $\alpha$ moves along $\Gamma_\theta'$, the colors of the Fatou chain $\Pi$ of $f_\alpha$ converging to $P_3$ and its image $f_\alpha(\Pi)$ must follow the alternating pattern specified on Figure \ref{three_special_parameters_in_jordan_curve_case}, which includes the following properties.

(1) If $\alpha$ lies on the red segment, then the colors of $\Pi$ and $f_\alpha(\Pi)$ are both red; if $\alpha$ lies on the purple segment, then $\Pi$ is purple and $f_\alpha(\Pi)$ is green; if $\alpha$ lies on the green segment, then  $\Pi$ is green and $f_\alpha(\Pi)$ is purple.

(2) The red segment between $\alpha_1$ and $\tilde{\alpha}_1$ (resp. between $\alpha_3$ and $\tilde{\alpha}_3$) is symmetric with respect to $\alpha_*$ (resp. 
$\tilde{\alpha}_*$) under the map $\alpha\mapsto \frac{\lambda}{\alpha}$.

(3) The green segment contains $\alpha_2$ in its interior and its endpoints are at $\tilde{\alpha}_1$ and $\tilde{\alpha}_3$.

(4) The purple segment contains $\tilde{\alpha}_2$ in its interior and its endpoints are at $\alpha_3$ and $\alpha_1$.

(5) The eight points $\alpha_1$, $\alpha_*$, $\tilde{\alpha}_1$, $\alpha_2$, $\tilde{\alpha}_3$, $\tilde{\alpha}_*$, $\alpha_3$ and  $\tilde{\alpha}_2$ are arranged on the Jordan curve in the clockwise order. 

(6) The equivalent patterns of the Fatou chain $\Pi$ of $f_\alpha$ and $f_\alpha(\Pi)$
for $\alpha$ equal to each of $\alpha_1$, $\alpha_2$, $\alpha_3$, $\tilde{\alpha}_1$, $\tilde{\alpha}_2$ and $\tilde{\alpha}_3$ are presented on this figure as well. 
\end{rmk}

\subsection{Labels of eventually Siegel-disk components when $\alpha\in \Gamma_{\theta}^0\cup \Gamma_{\theta}^1$}\label{Labels of E Siegel comps - arc}
In this subsection, we describe how to introduce labels to the
eventually Siegel-disk Fatou components of $f_\alpha$ when $\alpha\in \Gamma_{\theta}^0\cup \Gamma_{\theta}^1$.
\begin{figure}[!htpb]
 \setlength{\unitlength}{1mm}
  \centering
  \includegraphics[width=0.8\textwidth]{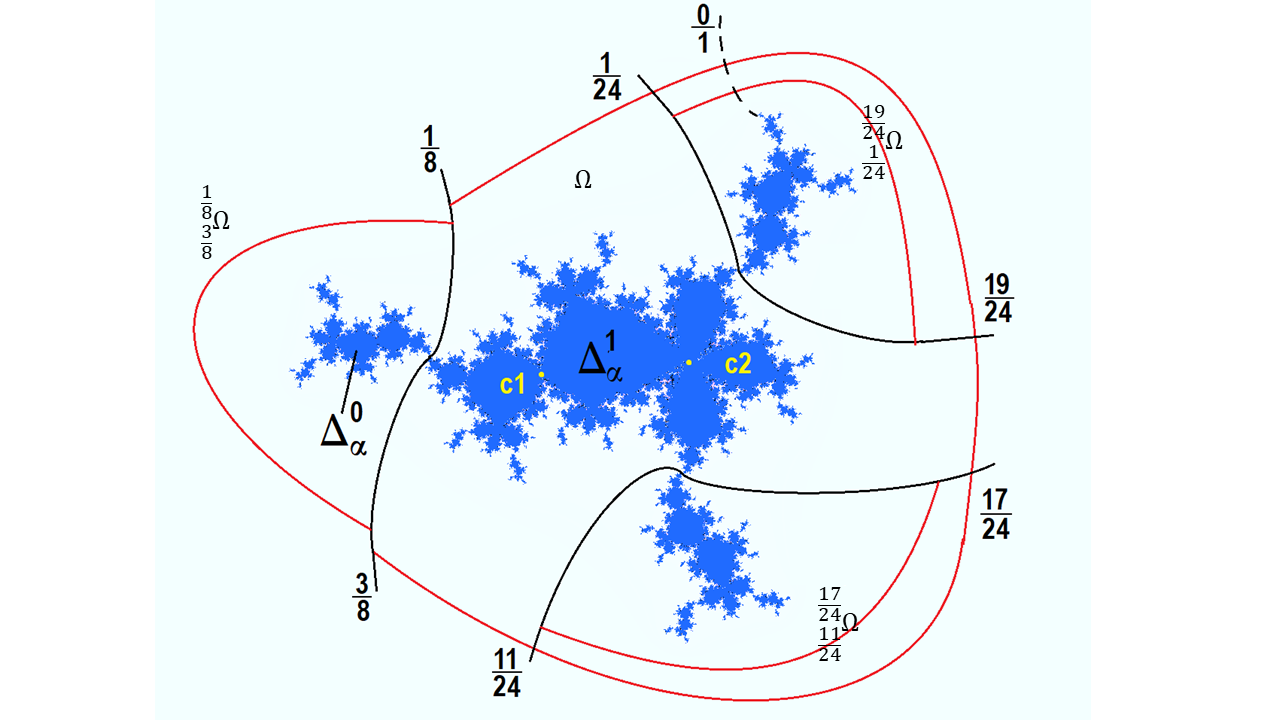}
  \caption{The filled-in Julia set of $f_{\alpha}$ for $\alpha\approx -3.2195-0.49565i$, with two critical points $c_0\approx0.76-0.1i$ and $c_1\approx1.377-0.05i$ on the boundary of $\triangle_{\alpha}^1$. Black curves represent external rays and red curves represent equi-potential curves.}
  \label{fatouset1}
\end{figure}
We focus on considering the case when $\alpha\in \Gamma_{\theta}^1$; that is, both critical points $c_0,c_1\in\partial\Delta_{\alpha}^1$. 
We know that the two external rays $R_{\frac{1}{8}}$ and $R_{\frac{3}{8}}$ on a $2$-cycle land at a repelling fixed point $P_3$ separating the Julia set $J(f_\alpha )$. Clearly, $R_{\frac{1}{8}}$ has two more preimages $R_{\frac{1}{24}}$ and $R_{\frac{17}{24}}$, and $R_{\frac{3}{8}}$ has two more preimages $R_{\frac{11}{24}}$ and $R_{\frac{19}{24}}$. We also know that $R_{\frac{1}{24}}$ and $R_{\frac{19}{24}}$ land at the same point (a preimage of $P_3$), and $R_{\frac{11}{24}}$ and $R_{\frac{17}{24}}$ land at the same point (another pre-image of $P_3$).
The complex plane $\C$ is divided by $R_{\frac{1}{8}}$ and $R_{\frac{3}{8}}$ into two regions, which we denote by $^{\frac{1}{8}}_{\frac{3}{8}}\Omega$ the one containing $0$ and by $\Omega$ the other containing $1$. Since both 
$c_0$ and $c_1$ lie on the boundary of $\Delta_\alpha^1$, $c_0$ and $c_1$ are contained in $\Omega$. We denote the three pre-images of $\Omega$ under $f_\alpha$ by $^{\frac{1}{8}}_{\frac{3}{8}}\Omega$, $^{\frac{11}{24}}_{\frac{17}{24}}\Omega$ and $^{\frac{19}{24}}_{\frac{1}{24}}\Omega$, where $^{r_1}_{r_2}\Omega$ stands for the domain bounded by $R_{r_1}$ and $R_{r_2}$ and on the right side of their union when one stands on $R_{r_1}$ and faces $R_{r_2}$. See Figure \ref{fatouset1}. Furthermore, the grand orbit of $R_{\frac{1}{8}}$ (or $R_{\frac{3}{8}}$) under $f_\alpha$ divides the eventually Siegel-disk Fatou components into infinitely many clusters 
(see Figure \ref{fatouset2}). One can see that there is only one crucial cluster in this case, that is the one containing $\Delta_\alpha^1$. We continue to denote this cluster by $\mathcal{C}_1$.
All other clusters are univalent pullbacks of $\mathcal{C}_1$ under $f_\alpha$. The combinatorial patterns of the clusters are determined by the grand orbit of $R_{\frac{1}{8}}$ under $f_\alpha $, which corresponds to the same combinatorics of the grand orbit of $R_{\frac{1}{8}}$ under the map $z\mapsto z^3$. It remains to give a complete description of the combinatorial pattern of the Fatou components on $\mathcal{C}_1$.

\begin{figure}[!htpb]
 \setlength{\unitlength}{1mm}
  \centering
  \includegraphics[width=0.9\textwidth]{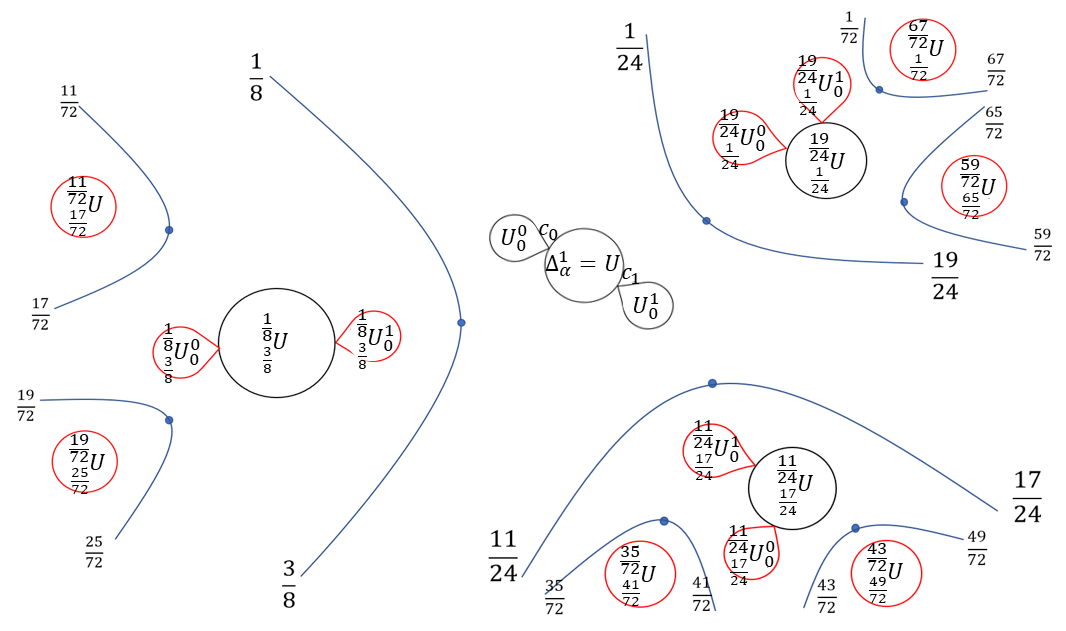}
  \caption{Illustration of the external rays separating clusters when $c_0$ and $c_1$ on the boundary of $\triangle_{\alpha}^1$, and the labels for the Fatou components of $\cup_{k=0}^{2}f^{-k}_{\alpha}(\triangle_{\alpha}^0\cup \triangle_{\alpha}^1)$.}
  \label{fatouset2}
\end{figure}
Now we introduce labels to the Fatou components on $\mathcal{C}_1$. There are two cases for the first pullbacks of $U=\Delta_\alpha^1$ under $f_\alpha^2$ corresponding to $c_0\neq c_1$ or $c_0=c_1$, which are shown as the yellow and blue regions in Case (i) and Case (ii) respectively on Figure \ref{Patterns of the first and second pullbacks for arc}. For further pullbacks, we need to consider three cases: (i) $f_\alpha^{\circ (2m)}(c_0)\neq c_1$ for any $m\geq 0$, which is the general case; (ii) $c_0=c_1$; (iii) $f_\alpha^{\circ (2m)}(c_0)=c_1$ for some $m\geq 1$. 

Let us first introduce a process to label the Fatou components on $\mathcal{C}_1$ in Case (i).  Let $U=\Delta_\alpha^1$ and denote by $U_0^0$ and $U_0^1$ the two preimages of $U$ under $f_\alpha^2$ on $\mathcal{C}_1$ attached at $c_0$ and $c_1$ respectively. Given integers $s\geq 1$ and $t\in\{0,1\}$, let $U_s^t$, $V_s^t$, $W_s^t$ stand for 
the components of $f_\alpha^{-2s}(U_0^t)$ such that
$\{U_s^0,U_s^1\}$, $\{V_s^0,V_s^1\}$, $\{W_s^0,W_s^1\}$ are pairs of the two components attached to $U$, $U_0^0$ and $U_0^1$ respectively.

Inductively, for any $m\geq 2$ and sequences $(s_1,\cdots, s_m)$ and $(t_1,\cdots,t_m)$, where $s_i\geq 1$ and $t_i\in\{0,1\}$ with $1\leq i\leq m$, let $U_{s_1,\cdots,s_m}^{t_1,\cdots,t_m}$, $V_{s_1,\cdots,s_m}^{t_1,\cdots,t_m}$, and $W_{s_1,\cdots,s_m}^{t_1,\cdots,t_m}$ be the unique components of $f_\alpha^{-2(s_1+\cdots +s_m+1)}(\Delta_\alpha^1)$ satisfying:
\begin{itemize}
\item $\{U_{s_1,\cdots,s_{m-1},s_m}^{t_1,\cdots,t_{m-1},0},U_{s_1,\cdots,s_{m-1},s_m}^{t_1,\cdots,t_{m-1},1}\}$ is the pair of the two components attached to $U_{s_1,\cdots,s_{m-1}}^{t_1,\cdots,t_{m-1}}$.
\item The above relationship holds similarly for $V_{s_1,\cdots,s_m}^{t_1,\cdots,t_m}$ and $W_{s_1,\cdots,s_m}^{t_1,\cdots,t_m}$.
\item If $Y$ stands for any of $U$, $V$ and $W$, then
\begin{equation*}
f_\alpha^{2}(Y_{s_1,s_2,\cdots,s_m}^{t_1,t_2,\cdots,t_m})=
\left\{
\begin{array}{ll}
U_{s_1-1,s_2\cdots,s_m}^{t_1,t_2,\cdots,t_m},  &~~~~~~~\text{if}~s_1\geq 2, \\
V_{s_2,\cdots,s_m}^{t_2,\cdots,t_m},  &~~~~~~~\text{if}~s_1=1 \text{ and } t_1=0, \\
W_{s_2,\cdots,s_m}^{t_2,\cdots,t_m},  &~~~~~~~\text{if}~s_1=1 \text{ and } t_1=1.
\end{array}
\right.
\end{equation*}
\end{itemize}
Following this procedure, the labels of the Fatou components on $\mathcal{C}_1$ up to the third pullbacks of $U$ under $f^2_\alpha$ in Case (i) are shown as Case (i) on Figure \ref{Patterns of the first and second pullbacks for arc}.
\begin{figure}[htbp]
  \setlength{\unitlength}{1mm}
  \subfigure[{$f^{2m}_\alpha (c_0)\not=c_1$ for any $m\ge 0$}]{\includegraphics[width=0.8\textwidth]{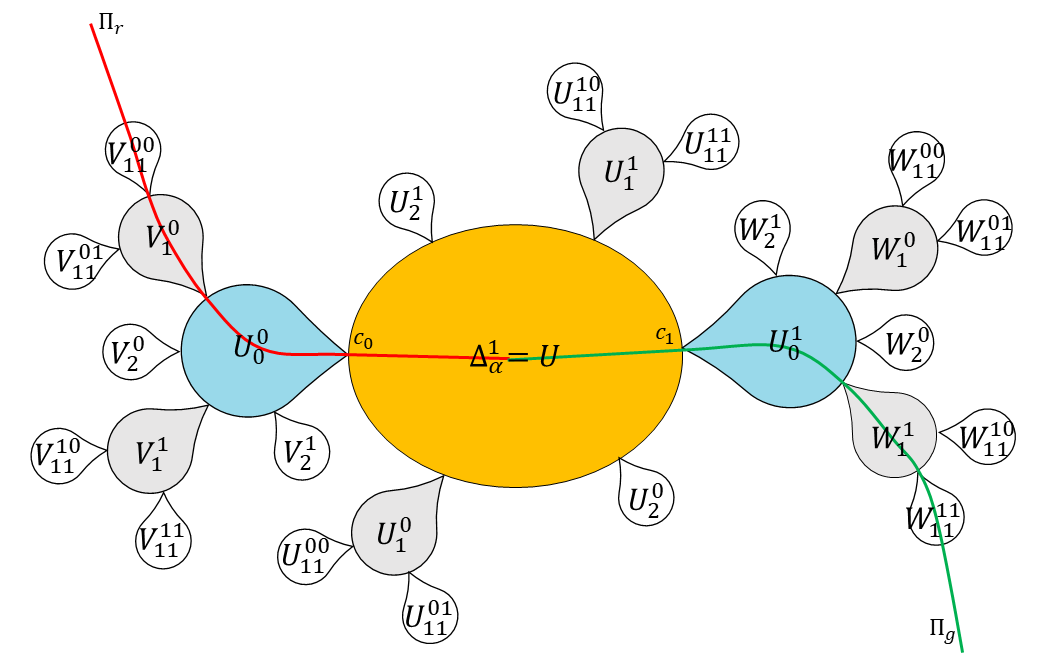}}
  \subfigure[{$c_0=c_1$}]{\includegraphics[width=0.9\textwidth]{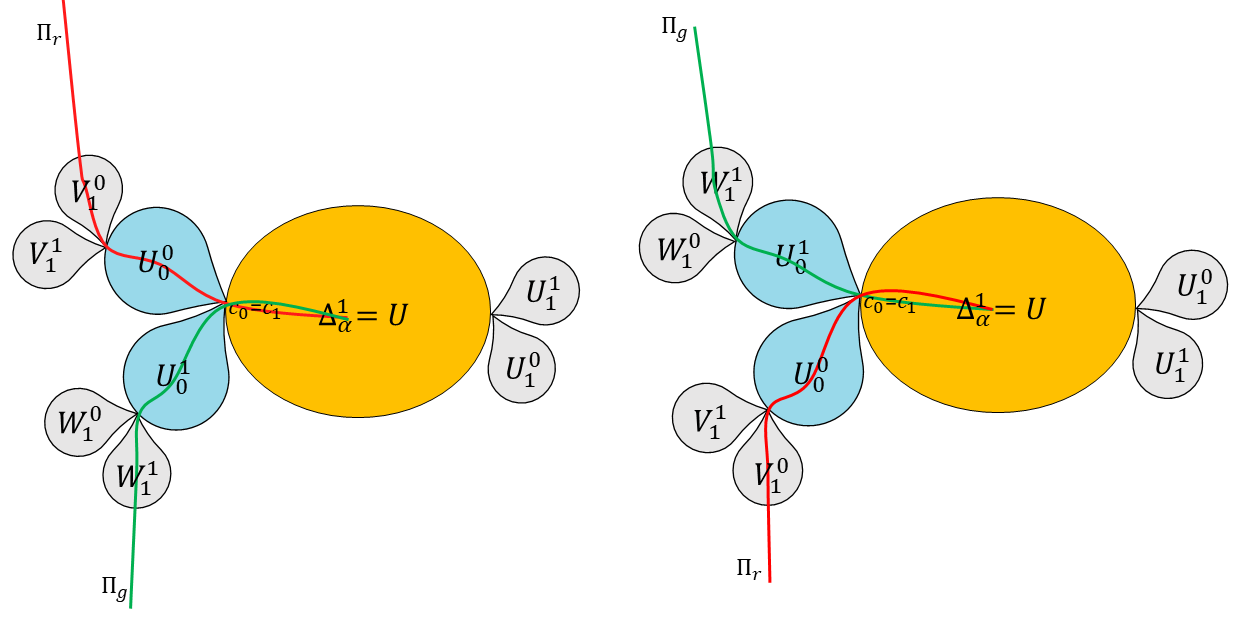}}
  \subfigure[{$f^{ 2}_\alpha(c_0)=c_1$}]{\includegraphics[width=0.9\textwidth]{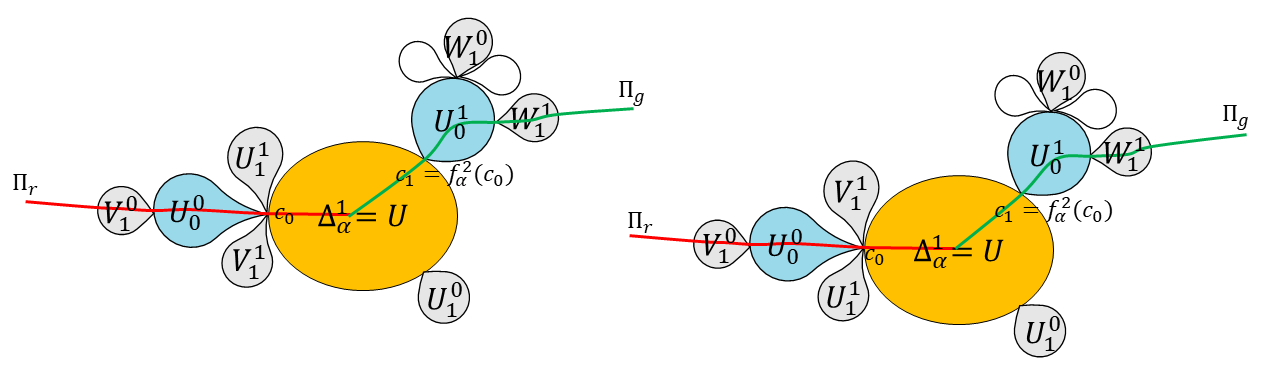}}
  \caption{Labels of the nine components of $f_\alpha^{-4}(\Delta_\alpha^0)$ on $\mathcal{C}_1$ and two convergent Fatou chains on $\mathcal{C}_1$ for Case (i), (ii) and (iii).}
  \label{Patterns of the first and second pullbacks for arc}
\end{figure}

When $c_0=c_1$, the two Fatou components $U^0_0$ and $U^1_0$ attach to the boundary of $U$ at the same point $c_0=c_1$. Further pullbacks of $U$, $U^0_0$ and $U^1_0$ under $f_\alpha^2$ on $\mathcal{C}_1$ show that every new pair of components attach to the same point on the boundary of an existing component. Except this difference, the components on $\mathcal{C}_1$ grow in the same way as in Case (i). Therefore, we can apply the same procedure as in Case (i) to introduce labels to all Fatou components on $\mathcal{C}_1$ in Case (ii). See Case (ii) on Figure \ref{Patterns of the first and second pullbacks for arc} for the labels of the Fatou components on $\mathcal{C}_1$ up to the second pullbacks of $U$ under $f^2_\alpha$.

If $f_\alpha^{2m}(c_0)=c_1$ for some $m\geq 1$, then $m$ is the unique moment for $c_0$ to land on $c_1$ under the iteration of $f_\alpha ^2$. Thus, $f_\alpha^{2j}(c_0)\neq c_1$ for any $0\le j\le m-1$ and hence the labels of the first $m$ pullbacks of $U$ under $f^2_\alpha$ on $\mathcal{C}_1$ are as the same as those components in Case (i). Among the new components on $\mathcal{C}_1$ generated by the next pullback by $f_\alpha^2$, two of them are attached to $U$ at $c_0$, which we may consider one of them as $U_m^1$ and the other as $V_m^1$. Then the labeling for the further pullbacks of $U$ under $f^2_\alpha$ continues to follow the same procedure for Case (i). See Case (iii) on Figure \ref{Patterns of the first and second pullbacks for arc} for the labels of the Fatou components up to the second pullbacks of $U$ under $f^2_\alpha$ in Case (iii) when $m=1$.

We pay attention to two Fatou chains on $\mathcal{C}_1$ starting from $U$ and $U^0_0$ or from $U$ and $U^1_0$.
\begin{lem}\label{convergent Fatou chain for arc} Let $\alpha\in \Gamma_\theta^1$.
    There are two different convergent Fatou chains on the crucial cluster $\mathcal{C}_1$, which are defined by (\ref{two chains of C_1 for arc case}) and denoted by $\Pi_r$ and $\Pi_g$ respectively, such that one of them, depending on $\alpha$, converges to the separating repelling fixed point $P_3$ and the other converges to a periodic point of period $2$ (see the red and green chains in Case (i) on Figure \ref{Patterns of the first and second pullbacks for arc}). Furthermore, when $c_0=c_1$, there are exactly two different relative combinatorial patterns between the chain converging to $P_3$ and the three domains $U$, $U^0_0$ and $U^1_0$, which are illustrated as the two red chains or the two green chains in (a) and (b) on Figure \ref{fatou_chain_for_arc}.
\end{lem}
\begin{proof}
Let $\alpha\in \Gamma_\theta^1$. By considering how $f_\alpha^2$ maps the Fatou components on $\mathcal{C}_1$, we can see that 
\begin{equation}\label{two chains of C_1 for arc case}
\begin{split}
        &\Pi_r=\{U,U_0^0,V_1^0,V_{11}^{00},V_{111}^{000},V_{1111}^{0000},\cdots\} \text{ and }\\
        &\Pi_g=\{U,U_0^1,W_1^1,W_{11}^{11},W_{111}^{111},W_{1111}^{1111},\cdots\}
\end{split}
\end{equation}
are two Fatou chains on $\mathcal{C}_1$, which are colored by red and green respectively on Figure \ref{Patterns of the first and second pullbacks for arc} in three cases. 

By Corollary \ref{local connectivity}, the Julia set of $f_\alpha$ is locally connected. Then $\Pi_r$ and $\Pi_g$ converge to some fixed points of $f_\alpha^2$. We also know that $$\mathcal{C}_1\subset \C\setminus (^{\frac{1}{8}}_{\frac{3}{8}}\Omega \cup ^{\frac{11}{24}}_{\frac{17}{24}}\Omega\cup ^{\frac{19}{24}}_{\frac{1}{24}}\Omega).$$ 
The landing points of $R_{\frac{1}{4}}$ and $R_{\frac{3}{4}}$ form a $2$-cycle of $f_\alpha$ and the landing points of $R_{\frac{5}{8}}$ and $R_{\frac{7}{8}}$ form another $2$-cycle of $f_\alpha $. Since $0$ and $1$ is the third $2$-cycle of $f_\alpha $, it follows that the limiting point of one of the two Fatou chains has to be $P_3$. 
Meanwhile, $\Pi_r$ and $\Pi_g$ do not have a common limit point since the filled-in Julia set of $f^{2}_\alpha$ has no loops of Fatou components. Thus, the landing point of $R_{\frac{3}{4}}$ is the limiting point of another Fatou chain, which is a periodic point of period $2$ under $f_\alpha$. 

\begin{figure}[!htpb]
 \setlength{\unitlength}{1mm}
  \centering
  \includegraphics[width=0.9\textwidth]{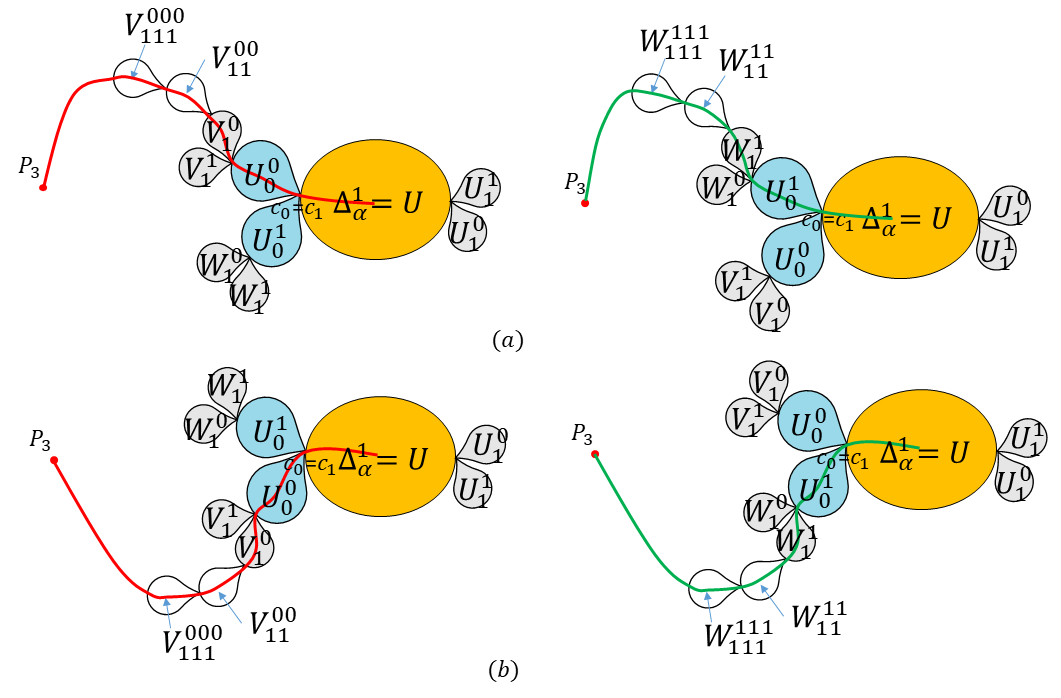}
  \caption{The relative combinatorical patterns between the red chain and $U$, $U^0_0$ and $U^1_0$ in (a) and (b) are different; in (a) or (b) the relative combinatorical pattern between the red chain and $U$, $U^0_0$ and $U^1_0$ is as the same as the one between the green chain and $U$, $U^0_0$ and $U^1_0$ by switching $U_0^0$ and $U^1_0$ and correspondingly updating the labels of other components on one of the chains.}
  \label{fatou_chain_for_arc}
\end{figure}

Let us assume that the red Fatou chain $\Pi_r$ is the one landing at $P_3$. When $c_0$ approaches $c_1$ along the boundary of $U=\Delta_\alpha^1$ clockwise or counterclockwise, the limiting patterns of the relative relationships between $\Pi_r$ and the three domains $U$, $U^0_0$ and $U^1_0$ are different, which are shown as (a) and (b) on Figure \ref{fatou_chain_for_arc}. These two limiting patterns become the two different relative combinatorial patterns between $\Pi_r$ and $U$, $U^0_0$ and $U^1_0$ for the case when $c_0=c_1$.
They are viewed as different in the sense that they are not homotopic to each other by a continuous family of orientation-preserving homeomorphisms rel $P_3$ and $U$. 

Similarly, if the green Fatou chain $\Pi_g$ is the one landing at $P_3$, then there are also two different relative combinatorial patterns between $\Pi_g$ and the three domains $U$, $U^0_0$ and $U^1_0$.
\end{proof}
\begin{rmk} Let $\alpha\in \Gamma_\theta^1$ and assume that $f_\alpha^{2m}(c_0)=c_1$ for some $m\ge 1$. Denote by $\Pi_r$ the Fatou chain landing at the separate repelling fixed point $P_3$ of $f_\alpha$. When $f_\alpha^{2m}(c_0)$ approaches $c_1$ along the boundary of $U=\Delta_\alpha^1$ clockwise or counterclockwise, we can observe the following properties:

(1) The limiting patterns of the labels of the components on $\Pi_r$ and $\Pi_g$ remain unchanged. 

(2) Although the limiting patterns of the labels of all components on $\mathcal{C}_1$ are different, the dynamics of $f_\alpha^2$ on $\mathcal{C}_1$ are the same; that is, the limiting patterns are equivalent in the sense that they are the same if we switch the labels of the components labeled by $U^1_m$ and $V^1_m$ and correspondingly update the labels of the components on $\mathcal{C}_1$ through the labeling algorithm. See Case (iii) on Figure \ref{Patterns of the first and second pullbacks for arc} for an illustration for the case when $m=1$.

In summary, we emphasize that the limiting combinatorial patterns of the components on $\mathcal{C}_1$ are the same as $f_\alpha^{2m}(c_0)$ approaches $c_1$ from either side of $c_1$ if $m\ge 1$, which are different in the case when $m=0$.
\end{rmk}

By Proposition \ref{Symmetry of Gamma},  $\alpha\in\Gamma_\theta^0$ if and only if
$\lambda/\alpha\in\Gamma_\theta^1$. In fact, Lemma \ref{lem:sym} states that $f_{\lambda/\alpha}$ is conjugated to $f_\alpha $ by the map $\tau(z)=-z+1$. Under such a relationship, the combinatorial pattern of the Fatou components of $f_\alpha$ for $\alpha\in\Gamma_{\theta}^0$ is uniquely determined by the one of the Fatou components of $f_{\lambda/\alpha}$, where $\lambda/\alpha\in \Gamma_{\theta}^1$, and of course the description of the Fatou components of $f_\alpha$ on its crucial cluster $\mathcal{C}_0$ follows from the one of the Fatou components of $f_{\lambda/\alpha}$ on its crucial cluster $\mathcal{C}_1$.

Till now, for each $\alpha $ on $\Gamma_\theta$, $\Gamma_\theta^1$ or $\Gamma_\theta^0$,
we have introduced labels to the eventually Siegel-disk Fatou components through three steps: (1) use the partition of the complex plane by the grand orbit of the external rays in $B_\infty$ on a $2$-cycle to distinguish the addresses of different clusters of the Fatou components; (2) use an inductive procedure to give addresses/labels to the Fatou components on a crucial cluster; (3) use the pullbacks of the crucial cluster under $f_\alpha $ to give the addresses/labels of the Fatou components on other clusters.

\section{Proof of the Main Theorem}\label{proofmain}
In this last section, we prove the Main Theorem of this paper. Part (a) of the Main Theorem has been proved in Proposition \ref{continuity}. By the symmetry between $\Gamma_\theta^0$ and $\Gamma_\theta^1$ given in Proposition \ref{Symmetry of Gamma}, we only need to prove either Part (c) or (d) of the Main Theorem. Thus, we prove Part (b) and (d) of the Main Theorem in this section.
In the first subsection, we give a slightly generalized statement of a result of Zhang in \cite{zhang2022}, which enable us to know that the Julia set of $f_\alpha$ has $0$ Lebesgue measure; in the second subsection, we show that there is a quasiconformal extension to a multiply connected domain when the boundaries are quasicircles and the prescribed maps among the components of the boundaries are quasisymmetric; in the third subsection, we associate $f_{\alpha}$ with a quantity called the conformal angle and prove that the conformal angle and the combinatorial pattern of the eventually Siegel components of $f_\alpha$ uniquely determine the parameter $\alpha$; in the fourth subsection, we prove that $\Gamma_\theta$ is a Jordan curve, and in the last subsection, we show that $\Gamma_\theta^1$ is a Jordan arc.

\subsection{Lebesgue measure of $J(f_\alpha)$} In this subsection, we address whether or not the Lebesgue measure of the Julia set $J(f_\alpha)$ of $f_\alpha$ is zero. In \cite{zhang2022}, Zhang proved the following result.
\begin{thm}[Lemma 6.1, \cite{zhang2022}]\label{Zero Lebesgure measure - Zhang} If a rational map $g$ has a fixed Sigel disk $\Delta $ with a rotation number of bounded type and the forward orbit of every critical point of $g$ either intersects the closure of $\Delta$, or is eventually periodic, or belongs to the basin of some attracting cycle, then the Julia set $J(g)$ of $g$ has Lebesgue measure equal to $0$.
\end{thm}
The proof of Theorem \ref{Zero Lebesgure measure - Zhang} presented in \cite{zhang2022} can be modified to obtain the same result under a slightly relaxed condition, which allows the map to have more than one fixed Sigel disk. Since the modification is straightforward, we will sate the generalized result without recapitulating the details of the proof. 
\begin{thm}\label{Zero Lebesgure measure under a weaker condition} Suppose that each fixed Siegel disk of a rational map $g$ has rotation number of bounded type and the forward orbit of every critical point of $g$ either intersects the closure of some fixed Siegel disk, or is eventually periodic, or belongs to the basin of some attracting cycle, then the Julia set $J(g)$ of $g$ has Lebesgue measure equal to $0$.
\end{thm}
An immediate corollary follows.
\begin{cor}\label{Leb measure of f-alpha} If $\alpha\in \Gamma_\theta\cup  \Gamma_\theta^0\cup \Gamma_\theta^1$, then the Lebesgue measure of $J(f_\alpha)$ is zero. 
\end{cor}

\subsection{Quasiconformal extensions to multiply connected domains}\label{QC ext to multiply conn domain} In this subsection, we develop an existence result on a quasiconformal extension to a multiply connected domain when the boundaries are quasicircles and the prescribed maps on the boundaries are quasisymmetric. 

By a {\em quasicircle} we mean a Jordan curve $\gamma $ on the complex plane $\mathbb{C}$ satisfying 
$$\diam \gamma (z_1, z_2)\le C|z_1-z_2| \text{ for any two points }z_1, \;z_2\in \gamma ,$$
where $C$ is a positive constant and $\gamma (z_1, z_2)$ stands for the arc on $\gamma $ between $z_1$ and $z_2$ and of a smaller diameter. 
A Jordan domain bounded by a quasicircle is called a {\em quasidisk}.
Theorems 5.17 and 5.25 in \cite{Pom13} imply that a Jordan curve $\gamma $ on $\mathbb{C}$ is a quasicircle if and only if there exists a quasiconformal map $\phi$ of the Riemann sphere $\EC$ fixing $\infty $ such that $\phi(\T)=\gamma $, where $\T$ is the unit circle on the complex plane centered at the origin. Using this necessary and sufficient condition and Theorem 8.1 in Chapter II of \cite{LV73}, one can see that quasicircles are quasiconformal invariants in the following sense. 
\begin{lem}\label{quasicircle is a quasiconformal invariant}
Let $U$ and $V$ be two open domains in $\mathbb{C}$ ($U$ is connected but not necessarily simply connected) and let $\psi:U\rightarrow V$ be quasiconformal. If $\gamma$ is a quasicircle contained in $U$, then $\psi(\gamma)$ is a quasicircle. 
\end{lem}
\begin{thm}{\cite[Chapter II, Theorem 8.1]{LV73}}\label{LV}
Let $G$ be a connected open set on the complex plane and $\Omega$ a compact subset of $G$, and let $\phi: G\to G'$ be a $K$-quasiconformal mapping. Then there exists a quasiconformal mapping of $\mathbb{C}$ (fixing $\infty $) which coincides with $\phi $ on $\Omega$ and whose maximal dilatation is bounded by a number depending only on $K$, $G$ and $\Omega$.
\end{thm}
\begin{proof}[Proof of Lemma \ref{quasicircle is a quasiconformal invariant}] By Theorem \ref{LV}, there exists a quasiconformal mapping $\tilde{\psi}$ of $\mathbb{C}$ (fixing $\infty$) such that $\tilde{\psi}|_{\gamma}=\psi|_{\gamma}$. Since $\gamma $ is a quasicircle, there exists a quasiconformal map $\phi $ of $\mathbb{C}$ (fixing $\infty $) such that $\gamma =\phi (\T)$. Clearly, $\tilde{\psi}\circ \phi$ is a quasiconformal mapping of $\mathbb{C}$ (fixing $\infty$) and $\psi (\gamma)=(\tilde{\psi}\circ \phi)(\T)$. Therefore, $\psi(\gamma)$ is also a quasicircle.
\end{proof}

Let $\gamma$ be a Jordan curve on $\mathbb{C}$, $G$ the Jordan domain on $\mathbb{C}$ bounded by $\gamma$, and $\phi$ a Riemann mapping from $\D$ onto $G$.
Proposition 5.10 and Theorem 5.11 in \cite{Pom13} imply that $\gamma $ is a quasicircle if and only if $\tilde{\phi}|_{\T}$ is a quasisymmetric map between $\T$ and $\gamma $, where $\tilde{\phi}$ is 
the extension of $\phi $ to the closure of $\D$. We call $\phi$ a {\em conformal parametrization} of $\gamma$. 

Given two quasicircles $\gamma_1$ and $\gamma_2$, let $\phi_1$ and $\phi_2$ be their conformal parametrizations. Then an orientation preserving homeomorphism $h:\gamma_1\to \gamma_2$ is said to be {\em quasisymmetric} if $ h\circ \hat{\mathcal{\phi}}_1: \mathbb{S}^1\to \gamma_2$ is quasisymmetric, or equivalently if $ \hat{\mathcal{\phi}}^{-1}_2\circ h\circ \hat{\mathcal{\phi}}_1: \mathbb{S}^1\to  \mathbb{S}^1$ is quasisymmetric. 

\begin{prop}\cite[Corollary 2.13]{BNX14}\label{quasisymmetric}
    For $j=1$, $2$, suppose that $A_j$ is an open annulus bounded by two quasicircles $\gamma_j^{i}$ and $\gamma_j^{o}$. Let $\phi^i:\gamma_1^{i}\to \gamma_2^{i}$ and $\phi^o:\gamma_1^{o}\to \gamma_2^{o}$ be quasisymmetric maps between the inner and outer boundaries respectively. Then there exists a quasiconformal map $\phi:\overline{A_1} \to \overline{A_2}$ extending the boundary maps $\phi^i$ and $\phi^o$.
\end{prop}

We extend this result to a connected region of connectivity number $n$ for any $n\ge 2$.
\begin{prop}\label{quasi-extension}
    Suppose that $X$ and $Y$ are two connected regions on $\C$ with connectivity number $n\ge 2$, and each of them is bounded by $n$ quasicircles. If $g:\partial X \to \partial Y$ is quasisymmetric, then $g$ has a quasiconformal extension 
    $\tilde{g}: X\to Y$.
\end{prop}
\begin{proof}
Lemma \ref{quasisymmetric} has shown the case for $n=2$. We prove this statement by an induction. Suppose it is true for $n=k$, where $k\ge 2$. We show the statement for $n=k+1$.
Denote by $U_X$ one of the bounded components of $\C\setminus X$ and let $U_Y$ be the bounded component of $\C\setminus Y$ such that $g: \partial U_X\to \partial U_Y$.
Let $ \tilde{X}= X\cup U_X$. Then $\tilde{X}$ is a connected region on $\C$ of connectivity number $k$, and $g|_{\partial X\setminus \partial U_X}$ is the restriction of $g$ on the boundaries of $\tilde{X}$, which is quasisymmetric. Therefore, it has a quasiconformal extension $\hat{g}$ defined on $\tilde{X}$.

Since $U_X\subset \tilde{X}$ and $\hat{g}^{-1}(U_Y)$ is also contained in $ \tilde{X}$, we may take a quasidisk $\Omega\subset\tilde{X} $ such that $U_X\cup \hat{g}^{-1}(U_Y)\subset \Omega$. Then $\Omega\setminus U_X$ is an annulus with quasicircle boundaries, and $\hat{g}|_{\partial \Omega}$ and $g|_{\partial U_X}$ are quasisymmetric boundary maps by Lemma \ref{quasicircle is a quasiconformal invariant}. Therefore, $\hat{g}|_{\partial\Omega} \cup g|_{\partial U_X}$ has a quasiconformal extension $\tilde{g}$ to $\Omega\setminus U_X$. Now we define
    \begin{equation*}
    \tilde{g}(z)=
    \begin{cases}
        \tilde{g} \quad\quad\quad \textrm{if}\ z\in \Omega\setminus U_X;\\
        \hat{g} \quad\quad\quad \textrm{if}\ z\in X\setminus \Omega.
    \end{cases}
\end{equation*}
Then $\tilde{g}: X\to Y$ is a quasiconformal extension of $g$ to $X$.
\end{proof}

Assume that $A_1$ and $A_2$ are two annuli with Jordan curve boundaries and $g$ is a homeomorphism between $\overline{A_1}$ and $\overline{A_2}$. Through precomposition and postcomposition by conformal maps, we may assume that $A_1$ and $A_2$ are round annuli; that is, we may assume that $A_1=\{z: r_1^i<|z|<r_1^o\}$ and $A_2=\{z: r_2^i<|z|<r_2^o\}$, where $0<r_1^i<r_1^o$ and $0<r_2^i<r_2^o$. Let $T_1$ be the torus as the quotient space of
 $\overline{A_1}$ by identifying $r_1^ie^{i\theta}$ with $r_1^oe^{i\theta}$ for any $0\le \theta<2\pi$ and $T_2$ be the torus as the quotient space of $\overline{A_2}$ by identifying $g(r_1^ie^{i\theta})$ with $g(r_1^oe^{i\theta})$ for any $0\le \theta<2\pi$.
 Furthermore, let the fundamental group $\pi_1(T_1)$ be generated by $a_1$ and $b_1$, where $a_1$
 is the closed curve by identifying the endpoints of the line segment between $r_1^i$ and $r_1^o$ and $b_1$ is the circle of radius $r_1^o$, and let the fundamental group $\pi_1(T_2)$ be generated by $a_2$ and $b_2$, where $a_2$
 is the closed curve by identifying the endpoints of the line segment between $g(r_1^i)$ and $g(r_1^o)$ and $b_2$ is the circle of radius $r_2^o$. We say that $g$ has {\em no full twist} on $A_1$ if $g$ induces an isomorphism between $\pi_1(T_1)$ and $\pi_1(T_2)$ that maps $a_1$ and $b_1$ to $a_2$ and $b_2$ respectively. On the other hand, $g$ performs some full twists on $A_1$ if and only if the induced isomorphism maps $a_1$ to $a_2+kb_2$ for some nonzero integer $k$. In fact, when $g$ performs some full twists on $A_1$, we may unwind the twists through postcomposing $g$ by a self homeomorphism of $\overline{A_2}$ that takes the identity map on the boundaries of $\overline{A_2}$ and induces an isomorphism on $\pi_1(T_2)$ that maps $a_2$ and $b_2$ to $a_2-kb_2$ and $b_2$ respectively.

\begin{lem}\label{homotopic-id}
    Assume that $X$ is a connected region on $\EC$ of connectivity number $3$ and its boundaries are Jordan curves. If $g$ is a quasiconformal mapping from $X$ into $\EC$,
    then $g$ can be modified to a quasiconformal map $\tilde{g}$ on $X$ without changing its values on the boundaries such that $\tilde{g}$ has no full twist on any annulus separating the boundaries of $X$.
\end{lem}    

\begin{proof} 
Suppose that there is an annulus $A$ on $X$ such that it separates some boundary component of $X$ from the other boundary components and it has some full twists on $A$.
We may assume that the boundaries of $A$ are quasicircles, which we denote by $\gamma_1$ and $\gamma_2$. Since $g$ maps quasicircles to quasicircles, there is a quasiconformal extension $g_A$ of $g|_{\gamma_1\cup \gamma_2}$ to $A$ such that $g_A$ has no full twist on any annulus contained in $A$. It follows that $g|_{X\setminus A}\cup g_A$ has no full twist on $A$. We may call this process an unwinding-full-twist process on $A$.

Since $g$ is quasiconformal on $X$, there exists only finitely many annuli $A_j$, $j=1, 2, \cdots, m$, such that these annuli have pairwise disjoint interiors, each of them separates the boundaries of $X$, and $g$ performs some full twists on each of them. Otherwise, if there are infinitely many such annuli $A_j$, $j=1, 2, \cdots, m, \cdots$, then the modulus of $A_j$ converges to $0$ and hence the maximal dilatation of $g$ on $A_j$ goes to $\infty$ as $j\rightarrow \infty$, which is a contradiction to the quasiconformality of $g$ on $X$.
After applying the above unwinding-full-twist process $A_j$ for $j=1, 2, \cdots, m$, we are able to modify $g$ to a quasiconformal map $\tilde{g}$ on $X$ such that it agrees with $g$ on the boundaries of $X$ and it has no full twist on each $A_j$. It follows that $\tilde{g}$ has no full twist on any annulus separating the boundary components of $X$. 
\end{proof}

\begin{rmk} Note that if a connected region $X$ is of connectivity number $3$ and with Jordan curve boundaries, then any simple closed curve on $X$ is homotopic to a single point or a boundary component curve. Then the full twists of a homeomorphism $f$ from $X$ into $\EC$ can be arranged on annuli with pairwise disjoint interiors. This is an essential reason
for the previous Lemma \ref{homotopic-id} to hold. When the connectivity number of $X$ is bigger than $3$, the above property doesn't hold; that is, $g$ may have some full twists along a simple closed curve $\gamma $ but has no full twist on another a simply closed curve $\gamma'$ intersecting $\gamma $ and not homotopic to $\gamma$, where both $\gamma$ and $\gamma'$ separate the boundary components of $X$. Then an unwinding process along $\gamma $ will produce some full twists along $\gamma'$. This means that Lemma \ref{homotopic-id} is not true in general when the connectivity number of $X$ is bigger than $3$.
\end{rmk}

\subsection{Conformal angle, and rigidity of combinatorial pattern} 
In the previous long Section \ref{Combinatorial pattern of eventually Siegel disk comp}, we have studied all possible combinatorial patterns of the eventually Fatou components of $f_\alpha $ when $\alpha \in \Gamma_\theta\cup \Gamma_\theta^0\cup \Gamma_\theta^1$.
In this subsection, we prove if such a combinatorial pattern is realized by $f_\alpha$ for some $\alpha \in \Gamma_\theta\cup \Gamma_\theta^0\cup \Gamma_\theta^1$, then such a parameter is unique.

We first define the conformal angle between $c_0$ and $f_\alpha(c_1)$ if $\alpha\in \Gamma_\theta$ or between $c_0$ and $c_1$ if $\alpha\in \Gamma_\theta^0\cup \Gamma_\theta^1$. 
For any $\alpha \in \Gamma_\theta\cup \Gamma_\theta^0\cup \Gamma_\theta^1$, there exist two conformal maps
\begin{equation}\label{equ:h-condition}
    h^0_{\alpha}: \Delta^0_{\alpha}\to \D \quad \text{and}\quad h^1_{\alpha}:\Delta^1_{\alpha}\to \D,
\end{equation}
such that
\begin{equation*}
\begin{split}
  h_{\alpha}^0 \circ f^{\circ 2}_{\alpha} \circ (h_{\alpha}^0)^{-1}(\zeta) =\rho_\theta(\zeta),& ~\forall\, \zeta\in \D, \text{ where } h_{\alpha}^0(0)=0,\text{ and } \\
  h_{\alpha}^1 \circ f^{\circ 2}_{\alpha} \circ (h_{\alpha}^1)^{-1}(\zeta) =\rho_\theta(\zeta),&~\forall\, \zeta\in \D, \text{ where }
  h_{\alpha}^1(1)=0.
\end{split}
\end{equation*}
Since $\partial\Delta^0_{\alpha}$
and $\partial\Delta^1_{\alpha}$ are quasicircles, $h_\alpha^0$ and $h_\alpha^1$ can be extended to two homeomorphisms
\begin{equation*}\label{equ:h-0-infty-extend}
H_{\alpha}^0: \overline{\Delta_{\alpha}^0}\to \overline{\D}   \text{\quad and\quad} H_{\alpha}^1: \overline{\Delta_{\alpha}^1} \to  \overline{\D}.
\end{equation*}

Case 1. $\alpha\in \Gamma_\theta$. We normalize $H_{\alpha}^0(c_0)=1$ and $H_{\alpha}^1(f_{\alpha}(c_0))=1$, where $c_0$ is the critical point on $\partial\Delta_\alpha^0$. Then both $H_{\alpha}^0$ and $H_{\alpha}^1$ are unique. It follows that 
\begin{equation*}\label{H}
    H_\alpha^1\circ f_\alpha(z)=H_\alpha^0(z)
\end{equation*} 
for all $z\in\partial\Delta^0_{\alpha}$. Let
\begin{equation*}\label{conformal angle curve}
A(\alpha):=\arg\,H_{\alpha}^0\circ f_\alpha(c_1) \text{\quad and \quad} \widetilde{A}(\alpha):=\arg\,H_{\alpha}^1(c_1),
\end{equation*}
which are measured in the counterclockwise direction. We respectively call them the conformal angle between $c_0$ and $f_\alpha(c_1)$ on $\overline{\Delta_\alpha^0}$ as viewed from $0$ and the conformal angle between $f_\alpha(c_0)$ and $c_1$ on $\overline{\Delta_\alpha^1}$ as viewed from $1$. Since
$H_\alpha^1\circ f_\alpha(z)=H_\alpha^0(z)$, we obtain
\begin{equation*}\label{equ:A-alpha}
\begin{split}
A(\alpha)
=&~\arg\,H_{\alpha}^0\circ f_\alpha(c_1) \\
=&~\arg\,H_{\alpha}^1\circ f_\alpha^{\circ 2}(c_1)=\widetilde{A}(\alpha)+2\pi\theta.
\end{split}
\end{equation*}

Case 2. $\alpha\in \Gamma_\theta^0$. Let $c_0$ be one of the two critical points on $\partial\Delta_\alpha^0$. We continue to normalize $H_{\alpha}^0(c_0)=1$ and $H_{\alpha}^1(f_{\alpha}(c_0))=1$. Then $H_\alpha^1\circ f_\alpha=H_\alpha^0$ on $\partial\Delta^0_{\alpha}$. Let
\begin{equation*}\label{conformal angle arc-0}
A(\alpha):=\arg\,H_{\alpha}^0 (c_1).
\end{equation*}

Case 3. $\alpha\in \Gamma_\theta^1$. Let $c_0$ be one of the two critical points on $\partial\Delta_\alpha^1$. We normalize $H_{\alpha}^1(c_0)=1$ and $H_{\alpha}^0(f_{\alpha}(c_0))=1$. Then $H_\alpha^0\circ f_\alpha=H_\alpha^1$ on $\partial\Delta^1_{\alpha}$. Let
\begin{equation}\label{conformal angle arc-1}
A(\alpha):=\arg\,H_{\alpha}^1 (c_1).
\end{equation}

For any $\alpha\in \Gamma_\theta\cup\Gamma_\theta^0\cup\Gamma_\theta^1$, let $K(f_\alpha)$ be the filled-in Julia set of $f_\alpha$. Then there is a conformal map $\Psi_\alpha: \EC\setminus K(f_{\alpha})\to \EC\setminus \overline{\D}$ such that $\Psi_\alpha(\infty)=\infty$ and
\begin{equation*}\label{bottcher}
    [\Psi_\alpha(z)]^3=\Psi_\alpha\circ f_{\alpha}(z)
\end{equation*}
for any $z\in\EC\setminus K(f_{\alpha})$, which is called a $B\ddot{o}ttcher$ coordinate map of $f_\alpha $ at $\infty $. Define $B_{\alpha}^{r}= (\Psi_\alpha)^{-1}(\overline{D_r}^c)$, where $D_r$ is the disk centered at $0$ and of radius $r>1$. 

To distinguish the objects corresponding to two different parameters $\alpha$ and $\alpha'$, we incorporate $\alpha$ and $\alpha'$ into the notation of these objects. For example, we use $c_{k,\alpha}$ and $c_{k,\alpha'}$ to denote the critical points of $f_{\alpha}$ and $f_{\alpha'}$ respectively, where $k=0,1$.

\begin{defi}\label{def:same combinatorics}
    Given $\alpha, \alpha'\in\Gamma_\theta$ (resp. $\Gamma_\theta^0$ or $\Gamma_\theta^1$), we say that the eventually Siegel-disk components of $f_{\alpha}$ and $f_{\alpha'}$ have {\em the same combinatorial pattern} if $A(\alpha)=A(\alpha')$ and there is a homeomorhphism of $\EC$ that fixes $0$, $1$ and $\infty$ and preserves the (equivalent) labels of the eventually Siegel-disk components.
\end{defi}
Note that when $\alpha, \alpha'\in\Gamma_\theta^0$ (resp. $\Gamma_\theta^1$), two critical points lie on the boundary of the same Siegel disk. Then $A(\alpha)$ or $2\pi-A(\alpha)$ may be taken as the measure of the conformal angle depending on the labels of the two critical points. On the other hand, if $c_{0, \alpha}$ and $c_{0, \alpha'}$ are expressed as the values of one complex analytic function of one critical point on a domain containing $\alpha$ and $\alpha'$, and $c_{1, \alpha}$ and $c_{1, \alpha'}$ are expressed as the values of the complex analytic function of another critical point on the same domain, then the requirements for the eventually Siegel-disk components of $f_{\alpha}$ and $f_{\alpha'}$ to have the same combinatorial pattern are the same ones given in the definition. 

\begin{prop}\label{rigidity of combinatorics - curve and arc}
Assume that $\alpha, \alpha'\in\Gamma_\theta$ (resp. $\Gamma_\theta^0$ or $\Gamma_\theta^1$). If the eventually Siegel-disk components of $f_{\alpha}$ have the same combinatorial pattern as the eventually Siegel-disk components of $f_{\alpha'}$, then $\alpha=\alpha'$.
\end{prop}
\begin{proof} We divide the proof of this proposition into three steps. 

\medskip
{\bf Step 1.} There exists a quasiconformal mapping $\phi_0:\EC\to \EC$ such that $\phi_0:\Delta_{\alpha}^0\cup\Delta_{\alpha}^1\cup B^{r}_{\alpha}\rightarrow \Delta_{\alpha'}^0\cup\Delta_{\alpha'}^1\cup B^{r}_{\alpha'}$ is a conformal conjugacy between $f_{\alpha}|_{\Delta_{\alpha}^0\cup\Delta_{\alpha}^1\cup B^{r}_{\alpha}}$ and $f_{\alpha'}|_{\Delta_{\alpha'}^0\cup\Delta_{\alpha'}^1\cup B^{r}_{\alpha'}}$, and $\phi_0$ is homotopic to the identity map on $\EC$ rel three points $0$, $1$ and $\infty$. 

Using the above definitions of $H_{\alpha}^0$, $H_{\alpha}^1$ and $\Psi_\alpha$, we obtain the corresponding maps $H_{\alpha'}^0$, $H_{\alpha'}^1$ and $\Psi_{\alpha'}$ for $\alpha'$. Define 
\begin{equation*}
    \phi(z)=
    \begin{cases}
        (H_{\alpha'}^0 )^{-1}\circ H_{\alpha}^0(z) \quad\quad\quad \textrm{if}\ z\in \overline{\Delta^0_{\alpha}},\\
        (H_{\alpha'}^1)^{-1}\circ H_{\alpha}^1(z) \quad\quad\quad \textrm{if}\ z\in \overline{\Delta_{\alpha}^1},\\
        (\Psi_{\alpha'})^{-1}\circ \Psi_\alpha(z) \quad\quad\quad\ \textrm{if}\ z\in \overline{B_{\alpha}^{r}}.
    \end{cases}
\end{equation*}
Clearly, $\phi: \overline{\Delta_{\alpha}^0}\cup \overline{\Delta_{\alpha}^1} \cup \overline{B_{\alpha}^{r}}\rightarrow \overline{\Delta_{\alpha'}^0}\cup \overline{\Delta_{\alpha'}^1} \cup \overline{B_{\alpha'}^{r}}$ is 
a homeomorphism and it is conformal between their interiors. Because $A(\alpha)=A(\alpha')$, $\phi$ satisfies
\begin{enumerate}
    \item $\phi(c_{0,\alpha})= c_{0,\alpha'}$, $\phi(c_{1,\alpha})= c_{1,\alpha'}$,\\ $\phi(f_{\alpha}(c_{0,\alpha}))= f_{\alpha'}(c_{0,\alpha'})$, and $\phi(f_{\alpha}(c_{1,\alpha}))= f_{\alpha'}(c_{1,\alpha'})$;
    \item $\phi\circ f_{\alpha}(z)= f_{\alpha'}\circ \phi(z)$ for any $z\in \overline{\Delta_{\alpha}^0}\cup \overline{\Delta_{\alpha}^1} \cup \overline{B_{\alpha}^{r}}$.
\end{enumerate}
Since $\overline{\Delta_{\alpha}^0}$, $\overline{\Delta_{\alpha}^1}$ and $\overline{B_{\alpha}^{r}}$ are quasidisks and mutually disjoint, by Proposition \ref{quasi-extension} and Lemma \ref{homotopic-id}, $\phi$ has a quasiconformal extension $\phi_0: \EC\to \EC$ such that $\phi_0$ has no full twist on any annulus contained in $\EC\setminus \overline{\Delta_{\alpha}^0} \cup \overline{\Delta_{\alpha}^{1}}\cup \overline{B^{r}_{\alpha}}$ and separating its boundary components. 

\medskip
{\bf Step 2.} There exists $\phi_1:\EC\to \EC$ such that 
\begin{itemize}
  \item $f_{\alpha'}\circ\phi_1=\phi_0\circ f_{\alpha}$;
  \item $\phi_1=\phi_0$ on $\overline{\Delta_{\alpha}^0} \cup \overline{\Delta_{\alpha}^{1}}\cup \overline{B^{r}_{\alpha}}$;
  \item $\phi_1$ has no full twist on any annulus contained in $\EC\setminus f_\alpha^{-1}(\overline{\Delta_{\alpha}^0} \cup \overline{\Delta_{\alpha}^{1}}\cup \overline{B^{r}_{\alpha}})$ and separating its boundary components.
  \item $\phi_1$ is homotopic to $\phi_0$ rel $\partial{\Delta_{\alpha}^0} \cup \partial{\Delta_{\alpha}^{1}}\cup \{\infty\}$;
  \item $\phi_1$ is holomorphic in every component of $f_{\alpha}^{-1}(\Delta_{\alpha}^0\cup\Delta_{\alpha}^1\cup B^{r}_{\alpha})$;
  \item $||\mu_{\phi_1}||_{\infty}=||\mu_{\phi_0}||_{\infty}$.
\end{itemize}
We say that $f_{\alpha}$ and $f_{\alpha'}$ are Thurston equivalent if there exist  $\phi_0$ and $\phi_1$ satisfying the above conditions (a), (b) and (c). We also call $\phi_1$ a lift of $\phi_0$.

Denote by $\mu_0$ the Beltrami coefficient of $\phi_0$, and let  
\begin{equation}\label{Beltrami coefficient}
    \mu_1(z)=f_{\alpha}^{*}\mu_0(z)=\mu_0(f_{\alpha}(z))\cdot \frac{\overline{f'_{\alpha}(z)}}{f'_{\alpha}(z)},
\end{equation}
which is the pull back of $\mu_0$ by $f_{\alpha}$. Then $||\mu_1||_{\infty}\le||\mu_0||_{\infty}$. By the measurable Riemannian mapping theorem, there is a unique quasiconformal map $\phi_1$ fixing $0$, $1$ and $\infty $ such that  $\mu_1$ is the Beltrami coefficient of $\phi _1$. 

Then $\tilde{f}=\phi_0\circ f_{\alpha}\circ \phi^{-1}_1$ is a rational map; that is, there is a rational map $\tilde{f}$ such that the following diagram commutes.
\begin{equation}\label{diagram commute}
    \begin{tikzcd}
  \EC  \ar[r,"\phi_1"] \ar[d,"f_{\alpha}"] & \EC  \ar[d,"\tilde{f}(=f_{\alpha'})"]\\
    \EC  \ar[r,"\phi_0"]     & \EC
    \end{tikzcd}
    \end{equation}
Furthermore, using the above condition (b), we obtain 
\begin{equation*}
\begin{split}
      \tilde{f}(\phi_1(\Delta_{\alpha}^0))=\phi_0(f_{\alpha}(\Delta_{\alpha}^0))
      =(\phi_0\circ f_{\alpha}\circ \phi_0^{-1}(\phi_0(\Delta_{\alpha}^0)))=f_{\alpha'}(\phi_0(\Delta_{\alpha}^0)).
\end{split}
\end{equation*}
Since $0\in \phi_{1}(\Delta_{\alpha}^0)\cap\phi_{0}(\Delta_{\alpha}^0)$, it follows that $\phi_{1}(\Delta_{\alpha}^0)\cap\phi_{0}(\Delta_{\alpha}^0)$ contains a neighborhood of $0$ on which $\tilde{f}=f_{\alpha'}$. Thus,  two rational maps $\tilde{f}$ and $f_{\alpha'}$ are the same one. 

Using the following two commutative diagrams
\begin{equation*}
    \begin{tikzcd}
  \Delta_{\alpha}^0  \ar[r,"\phi_1"] \ar[d,"f_{\alpha}"] & \phi_1(\Delta_{\alpha}^0)  \ar[d,"f_{\alpha'}"]\\
    \Delta_{\alpha}^1  \ar[r,"\phi_0"]     & \Delta_{\alpha'}^1 
    \end{tikzcd} \;\;\;\;\text{ and }\;\;\;\;
    \begin{tikzcd}
    \Delta_{\alpha}^0  \ar[r,"\phi_0"] \ar[d,"f_{\alpha}"] &   \Delta_{\alpha'}^0  \ar[d,"f_{\alpha'}"]\\
    \Delta_{\alpha}^1  \ar[r,"\phi_0"]     & \Delta_{\alpha'}^1,
    \end{tikzcd}
    \end{equation*}
we obtain $\phi_1|_{\Delta_{\alpha}^0}=f_{\alpha'}^{-1}|_{\Delta_{\alpha'}^1}\circ\phi_0\circ f_{\alpha}|_{\Delta_{\alpha}^0}=\phi_0|_{\Delta_{\alpha}^0}$. Similarly, we obtain $\phi_1|_{\Delta_{\alpha}^1}=\phi_0|_{\Delta_{\alpha}^1}$ and $\phi_1|_{B_{\alpha}^{r}}=\phi_0|_{B_{\alpha}^{r}}$.
Meanwhile, $\phi_1$ is holomorphic on $f_{\alpha}^{-1}(\Delta_{\alpha}^0\cup \Delta_{\alpha}^1\cup B_{\alpha}^{r} )$, and $\phi_1$ maps $f_{\alpha}^{-1}(\Delta_{\alpha}^0\cup \Delta_{\alpha}^1\cup B_{\alpha}^{r} )$ onto $f_{\alpha'}^{-1}(\Delta_{\alpha'}^0\cup \Delta_{\alpha'}^1\cup B_{\alpha'}^{r})$ and preserves the labels/positions of the preimages of the Siegel disks on the $2$-cycle. This means that for the existence of the map $\phi_1$, the preimages of the Siegel disks on the $2$-cycle under $f_{\alpha}$ must have the same combinatorial pattern as the preimages of the Siegel disks on the $2$-cycle under $f_{\alpha'}$.

Since the Beltrami coefficient $\mu_1$ is a pullback of $\mu_0$ in the sense of (\ref{Beltrami coefficient}), the complex dilatation of $\phi_1$ on the eventually Siegel-disk components of $f_{\alpha}$ are the pullback of the complex dilatation of $\phi_0$ on those components. Using the facts that (i) $\phi_0$ has no full twist on any annulus contained in $\EC\setminus \overline{\Delta_{\alpha}^0} \cup \overline{\Delta_{\alpha}^{1}}\cup \overline{B^{r}_{\alpha}}$ and separating its boundary components, 
(ii) the closures of the eventually Siegel-disk Fatou components of $f_\alpha$ don't form loops on $\C$, (iii) the Lebesgue measure of $J(f_{\alpha})$ is zero, and (iv) $\phi_1$ is a lift of $\phi_0$ in the sense that they satisfy the diagram (\ref{diagram commute}), we know that $\phi_1$ has no full twist on any annulus contained in $\EC\setminus f_\alpha^{-1}(\overline{\Delta_{\alpha}^0} \cup \overline{\Delta_{\alpha}^{1}}\cup \overline{B^{r}_{\alpha}})$ and separating its boundary components. Since 
 $\phi_1$ and $\phi_0$ agree on $\overline{\Delta_{\alpha}^0}\cup \overline{\Delta_{\alpha}^1} \cup \overline{B_{\alpha}^{r}}$ and both of them have no full twist on any annulus separating the boundary components of $\EC\setminus \overline{\Delta_{\alpha}^0}\cup \overline{\Delta_{\alpha}^1} \cup \overline{B_{\alpha}^{r}}$, we conclude that $\phi_1$ is homotopic to $\phi_0$ rel $\overline{\Delta_{\alpha}^0}\cup \overline{\Delta_{\alpha}^1} \cup \overline{B_{\alpha}^{r}}$.

\medskip
 {\bf Step 3.} Because the combinatorical patterns of the eventually Siegel-disk Fatou components of $f_{\alpha}$ and $f_{\alpha'}$ are same or equivalent, we can obtain a sequence $\{\phi _k\}_{k=0}^{\infty }$ of quasiconformal homeomorphisms of $\EC$ inductively as follows. 

Assume that for every $k\ge 1$, there is a quasiconformal homeomorphism $\phi_k :\EC \to \EC$ such that
\begin{itemize}
  \item $\phi_k$ is a lift of $\phi_{k-1}$ in the sense that $f_{\alpha'}\circ\phi_k=\phi_{k-1}\circ f_{\alpha}$;
  \item $\phi_k=\phi_{k-1}$ on $f_{\alpha}^{-(k-1)}(\overline{\Delta_{\alpha}^0} \cup \overline{\Delta_{\alpha}^{1}}\cup \overline{B^{r}_{\alpha}})$;
  \item $\phi_k$ has no full twist on any annulus contained in $\EC\setminus f_\alpha^{-k}(\overline{\Delta_{\alpha}^0} \cup \overline{\Delta_{\alpha}^{1}}\cup \overline{B^{r}_{\alpha}})$ and separating its boundary components.
  \item $\phi_k$ is homotopic to $\phi_{k-1}$ rel $\partial{\Delta_{\alpha}^0} \cup \partial{\Delta_{\alpha}^{1}\cup \{\infty\}}$;
  \item $\phi_k$ is holomorphic on $f_{\alpha}^{-k}(\Delta_{\alpha}^0\cup\Delta_{\alpha}^1\cup B^{r}_{\alpha})$, and maps every component of $f_{\alpha}^{-k}(\Delta_{\alpha}^0\cup\Delta_{\alpha}^1)$ onto the one of $f_{\alpha'}^{-k}(\Delta_{\alpha'}^0\cup\Delta_{\alpha'}^1)$ with the same/equivalent address;
  \item $\| \mu_{\phi_k}\|_{\infty}=\| \mu_{\phi_{k-1}}\|_{\infty}$.
\end{itemize}
Now we let $f_{\alpha}^{*}\mu_{\phi_k}$
be the pull back of $\mu_n$ under $f_{\alpha}$ and let $\phi_{n+1}$ be the quasiconformal homeomorphism of $\EC$ fixing three points $0$, $1$ and $\infty $ with its Beltrami coefficient equal to $f_{\alpha}^{*}\mu_{\phi_k}$ (i.e., $\mu_{\phi_{k+1}}=f_{\alpha}^{*}\mu_{\phi_k}$). By a similar process used in Step 2, we know that $\phi_{k+1}$ satisfies
\begin{itemize}
  \item $\phi_{k+1}$ is a lift of $\phi_{k}$ in the sense that $f_{\alpha'}\circ\phi_{k+1}=\phi_{k}\circ f_{\alpha}$;
  \item $\phi_{k+1}=\phi_{k}$ on $f_{\alpha}^{-k}(\overline{\Delta_{\alpha}^0} \cup \overline{\Delta_{\alpha}^{1}}\cup \overline{B^{r}_{\alpha}})$;
  \item $\phi_{k+1}$ has no full twist on any annulus contained in $\EC\setminus f_\alpha^{-(k+1)}(\overline{\Delta_{\alpha}^0} \cup \overline{\Delta_{\alpha}^{1}}\cup \overline{B^{r}_{\alpha}})$ and separating its boundary components.
  \item $\phi_{k+1}$ is homotopic to $\phi_{k}$ rel $\partial{\Delta_{\alpha}^0} \cup \partial{\Delta_{\alpha}^{1}\cup \{\infty\}}$;
  \item $\phi_{k+1}$ is holomorphic on $f_{\alpha}^{-(k+1)}(\Delta_{\alpha}^0\cup\Delta_{\alpha}^1\cup B^{r}_{\alpha})$, and maps every component of $f_{\alpha}^{-(k+1)}(\Delta_{\alpha}^0\cup\Delta_{\alpha}^1)$ onto the one of $f_{\alpha'}^{-(k+1)}(\Delta_{\alpha'}^0\cup\Delta_{\alpha'}^1)$ with the same/equivalent address;
  \item $\| \mu_{\phi_{k+1}}\|_{\infty}=\| \mu_{\phi_{k}}\|_{\infty}$.
\end{itemize}
Thus, we obtain a sequence of quasiconformal homeomorphisms $\{\phi_k:k\in\N\}$ of $\EC$ fixing three points $0$, $1$ and $\infty$ with their Beltrami coefficients satisfying
\begin{equation*}
\| \mu_{\phi_k}\|_{\infty}=\| \mu_{\phi_0}\|_{\infty}<1.
\end{equation*}
It follows that $\{\phi_k:k\in\N\}$ is a normal family.

Let $\Phi$ be the limiting map of a convergent subsequence of $\{\phi_k\}_{k\geq 0}$.
Then $\Phi$ is a conformal conjugacy on between $f_{\alpha}$ and $f_{\alpha'}$ on their Fatou sets. Since the Fatou sets are dense on $\EC$, it follows that $\Phi$ is also a quasiconformal conjugacy between $f_{\alpha}$ and $f_{\alpha'}$ on $\EC$.

By Corollary \ref{Leb measure of f-alpha}, the Julia set of $f_{\alpha}$ has zero Lebesgue measure. So $\Phi:\EC\to\EC$ is a conformal homeomorphism of $\EC$ fixing $0$, $1$ and $\infty$, which has to be the identity map. Thus, $f_{\alpha}=f_{\alpha'}$. Therefore, $\alpha=\alpha'$.
\end{proof}

The work of Section \ref{Combinatorial pattern of eventually Siegel disk comp} actually shows that given two parameters $\alpha, \alpha'\in\Gamma_\theta$ (resp. $\Gamma_\theta^0$ or $\Gamma_\theta^1$), if $A(\alpha)=A(\alpha')$ and the two Fatou chains under $f_\alpha$ converging to the separate repelling fixed point $P_{3, \alpha}$ have the same or equivalent combinatorial patterns as the two Fatou chains under $f_{\alpha'}$ converging to the separate repelling fixed point $P_{3, \alpha'}$, then the eventually Siegel-disk components of $f_\alpha$ and $f_{\alpha'}$ have the same combinatorial pattern. Therefore, we obtain the following corollary. 
\begin{cor}\label{rigidity of conformal angle and convergent Fatou chains}
(1) Given $\alpha, \alpha'\in\Gamma_\theta$, 
if $A(\alpha)=A(\alpha')\neq 0$ nor $\theta $ and the two Fatou chains $\Pi(\alpha)$ and $\Pi(\alpha')$ converging to the separating repelling fixed points respectively under $f_\alpha$ and $f_{\alpha'}$ have the same or equivalent combinatorial patterns, then $\alpha =\alpha'$; if $A(\alpha)=A(\alpha')=\theta$ and the relative combinatorial pattern between $\Pi^{d_0}(\alpha)$ and $U$, $U_0^0$, $U_0^1$ and $U_0^2$ is as the same as the relative combinatorial pattern between  $\Pi^{d_0}(\alpha')$ and $U$, $U_0^0$, $U_0^1$ and $U_0^2$, then $\alpha=\alpha'$; if $A(\alpha)=A(\alpha')=0$ and the relative combinatorial pattern between $\tilde{\Pi}^{d_0}(\alpha)$ and $\tilde{U}$, $\tilde{U}_0^0$, $\tilde{U}_0^1$ and $\tilde{U}_0^2$ is as the same as the relative combinatorial pattern between  $\tilde{\Pi}^{d_0}(\alpha')$ and $\tilde{U}$, $\tilde{U}_0^0$, $\tilde{U}_0^1$ and $\tilde{U}_0^2$, then $\alpha=\alpha'$. 

(2) Given $\alpha, \alpha'\in\Gamma_\theta^1$ (resp. $\Gamma_\theta^0$), if $A(\alpha)=A(\alpha')\neq 0$ and the two Fatou chains $\Pi(\alpha)$ and $\Pi(\alpha')$ converging to the separate repelling fixed points respectively under $f_\alpha$ and $f_{\alpha'}$ have the same combinatorial pattern or have the same combinatorial pattern after switching the markings of the two critical points of $f_{\alpha'}$, then $\alpha =\alpha'$; if $A(\alpha)=A(\alpha')=0$ and  the relative combinatorial pattern between $\Pi(\alpha)$ and $U$, $U_0^0$ and $U_0^1$ is as the same as the relative combinatorial pattern between $\Pi(\alpha')$ and $U$, $U_0^0$ and $U_0^1$, then $\alpha=\alpha'$. 
\end{cor}

\medskip
\subsection{Proof of $\Gamma_\theta$ to be a Jordan curve}\label{jordancurve}
In the subsection, we show that $\Gamma_\theta$ is a Jordan curve. Let us first prove three lemmas. 
\begin{lem}\label{c-G}
    The map $\alpha\mapsto e^{iA(\alpha)}$ is continuous on $\Gamma_\theta$.
\end{lem}
\begin{proof} To prove this lemma, it is more convenient to apply the normalization of $h_{\alpha}^0:\Delta_{\alpha}^0\to \D$ used in the proof of Proposition \ref{continuity}; that is, $h_{\alpha}^0$ is normalized by $h_{\alpha}^0(0)=0$ 
and $(h_{\alpha}^0)'(0)>0$, which is different from the normalization given in \eqref{equ:h-condition}. By Proposition \ref{continuity}, the inverse map $\eta_\alpha:\D\to \Delta_{\alpha}^0$ of $h_{\alpha}^0$ has a quasiconformal extension $\tilde{\eta}_\alpha:\EC\to \EC$, which depends continuously on $\alpha\in\Gamma_\theta$. By the expressions of the critical points in (\ref{critical pts}), $c_0(\alpha)$ and $c_1(\alpha)$ change continuously as $\alpha$ varies on $\Gamma_\theta$. Thus, the angle between $(\tilde{\eta}_\alpha)^{-1}(c_0(\alpha))$ and $(\tilde{\eta}_\alpha)^{-1}(f_{\alpha}(c_1(\alpha)))$ on $\overline{\D}$ as viewed from the origin  depends continuously on $\alpha\in\Gamma_\theta$. Clearly, this angle is equal to $A(\alpha)$ since the map $H_{\alpha}^0$ used to define $A(\alpha)$ in Section \ref{conformal angle curve} differs from $h_{\alpha}^0$ up to post-composition by a rigid rotation for each fixed $\alpha $. Hence $A(\alpha)$ depends continuously on $\theta$. Therefore, $\alpha\mapsto e^{iA(\alpha)}$ is continuous on $\Gamma_\theta$.
\end{proof}

\begin{lem}\label{injection}
The map $\alpha\mapsto e^{iA(\alpha)}$ is locally injective on $\Gamma_\theta$.
\end{lem}
Let us do some preparation before proving this lemma. Recall first that we have shown in Section \ref{critical points} that there are exact $8$ parameters $\alpha$ on the parameter plane such that one critical point is mapped to another by $f_\alpha$. Computer verification indicates that six of them lie on $\Gamma_\theta$, which we have denoted by $\alpha_k$ and $\tilde\alpha_k$, where $\tilde\alpha_k=\lambda/\alpha_k$ and $k=1,2,3$. Let us denote the other two by $\alpha_6$ and $\tilde\alpha_6$ with $\tilde\alpha_6=\lambda/\alpha_6$. Secondly, note that when $\alpha\in\Gamma_\theta$, $A(\alpha )=0$ (resp. $A(\alpha)=\theta$) if and only if $f_{\alpha}(c_1)=c_0$ (resp. $f_{\alpha}(c_0)=c_1$). Thirdly, we prove the following result on the Fatou chains converging to the separating repelling fixed point $P_3$. 

\begin{lem}\label{local invariance of color for the Fatou chain to P-3}
Let $\alpha\in \Gamma_\theta$. If $A(\alpha )\neq 0$ and $A(\alpha)\neq \theta$, then there exists a neighborhood $U$ of $\alpha$ on $\Gamma_\theta$ such that for all $\alpha'\in U$, the Fatou chains of $f_{\alpha'}$ converging to $P_{3, \alpha'}$ have the same colors (i.e., of the same combinatorial patterns) as the Fatou chains of $f_{\alpha}$ converging to $P_{3, \alpha}$ respectively.
\end{lem}
\begin{proof} At first, let $U$ be a small neighborhood of $\alpha$ such that $\alpha_k, \tilde\alpha_k\notin U$ for $k=1,2,3,6$. Since $P_{3, \alpha}$ is a repelling fixed point of $f_\alpha$, there is a neighborhood $W_{\alpha}$ of $P_{3, \alpha}$ such that $f_{\alpha}: f_{\alpha}^{-1}(W_{\alpha})\rightarrow W_{\alpha}$ is conjugated to a linear map $L_\alpha: z\rightarrow f_{\alpha}'(P_{3, \alpha})z$. When $\alpha'\in \Gamma_\theta$ is sufficiently close $\alpha$, the domain $W_{\alpha'}$ of the local conjugacy of $f_{\alpha'}$ at $P_{3, \alpha'}$ is sufficiently close to $W_{\alpha}$ in the Hausdorff metric, and there is a topological conjugacy $\phi: W_{\alpha}\rightarrow W_{\alpha'}$ between $f_{\alpha}: f_{\alpha}^{-1}(W_{\alpha})\rightarrow W_{\alpha}$ and $f_{\alpha'}: f_{\alpha'}^{-1}(W_{\alpha'})\rightarrow W_{\alpha'}$. Therefore, $W_{\alpha'}\cap W_{\alpha}$ contains a common neighborhood $W$ of $P_{3, \alpha}$ and $P_{3, \alpha'}$.

By Lemma \ref{convergent Fatou chain}, for any $\alpha\in \Gamma_{\theta}$, there are 
exactly one pair of Fatou chains, $\Pi_{\alpha}$ and $\tilde\Pi_{\alpha}$, on the crucial clusters $\mathcal{C}_{0, \alpha}$ and $\mathcal{C}_{1, \alpha}$, respectively, converging to the separate repelling fixed point $P_{3, \alpha}$. After a finitely many times of pullbacks under $f_\alpha$ along $\Pi_{\alpha}$, the further pullbacks of the head portion $f^{-2}_{\alpha}(\Delta_\alpha^0)\cap \Pi_{\alpha}$ is contained in $W$; that is, there exists $N>0$ such that for any $n\ge N$, $f^{-2n}_\alpha(f^{-2}_{\alpha}(\Delta_\alpha^0)\cap \Pi_{\alpha})\cap \Pi_{\alpha}\subset W$. Now when $\alpha'\in U\cap \Gamma_\theta$ is sufficiently close to $\alpha$, $f^{-2n}_{\alpha'}(f^{-2}_{\alpha'}(\Delta_{\alpha'}^0)\cap \Pi_{\alpha'})\cap \Pi_{\alpha'}\subset W$ for any $n\ge N$. Therefore,
$$\cup_{n\ge N}[f^{-2n}_\alpha(f^{-2}_{\alpha}(\Delta_\alpha^0)\cap \Pi_{\alpha})\cap \Pi_{\alpha}]\subset W$$ and 
$$\cup_{n\ge N}[f^{-2n}_{\alpha'}(f^{-2}_{\alpha'}(\Delta_{\alpha'}^0)\cap \Pi_{\alpha'})\cap \Pi_{\alpha'}]\subset W.$$
Therefore, there is a topological conjugacy between $\cup_{n\ge N}[f^{-2n}_\alpha(f^{-2}_{\alpha}(\Delta_\alpha^0)\cap \Pi_{\alpha})\cap \Pi_{\alpha}] $ and $\cup_{n\ge N}[f^{-2n}_{\alpha'}(f^{-2}_{\alpha'}(\Delta_{\alpha'}^0)\cap \Pi_{\alpha'})\cap \Pi_{\alpha'}]$. This conjugacy on the tails of $\Pi_{\alpha}$ and $\Pi_{\alpha'}$ can be promoted to a conjugacy on the Fatou chains $\Pi_{\alpha}$ and $\Pi_{\alpha'}$. Therefore, $\Pi_{\alpha}$ and $\Pi_{\alpha'}$ must be of the same color (the same combinatorial pattern). Similarly, we know that if $\alpha'$ is sufficiently close to $\alpha $, $\tilde{\Pi}_{\alpha}$ and $\tilde{\Pi}_{\alpha'}$ must be of the same color.
\end{proof}

\begin{proof}[Proof of Lemma \ref{injection}] We divide the proof into three cases.

Case 1: $\alpha \in \Gamma_\theta$ and $A(\alpha)\neq 0, \theta$. Then there exists a neighborhood $V$ of $\alpha $ on $\Gamma_\theta$ on which the map $\alpha\mapsto e^{iA(\alpha)}$ is injective. 

By Lemma \ref{local invariance of color for the Fatou chain to P-3}, there exists a neighborhood $V$ of $\alpha $ on $\Gamma_\theta$ such that for each $\alpha'\in V$,
the Fatou chains of $f_{\alpha'}$ converging to $P_{3, \alpha'}$ are of the same combinatorial patterns as the Fatou chains of $f_{\alpha}$ converging to $P_{3, \alpha}$ respectively. Then for any $\alpha', \alpha''\in V$, the eventually Siegel-disk components of $f_{\alpha''}$ has the same combinatorial pattern as the eventually Siegel-disk components of $f_{\alpha'}$. By Proposition \ref{rigidity of combinatorics - curve and arc}, $A(\alpha')=A(\alpha'')$ implies $\alpha'=\alpha''$. Thus, the map $\alpha\mapsto e^{iA(\alpha)}$ is injective on $V$.

Case 2: $\alpha \in \Gamma_\theta$ and $A(\alpha )=\theta$. In this case,  $c_0$ is mapped to $c_1$. We know from Section \ref{critical points} that there are at most four such parameters on $\Gamma_\theta$, which are isolated points on $\C$. Therefore, for each $\alpha \in \Gamma_\theta$ with $A(\alpha )=\theta$, there is a neighborhood $V$ of $\alpha $ on $\Gamma_\theta$ on which there is only one parameter $\alpha$ such that $A(\alpha )=\theta$ and on which for any $\alpha \in V$,  $A(\alpha )$ is close to $\theta$ (by Lemma \ref{continuity}). Now given $\alpha', \alpha''\in V$ with $A(\alpha')$ and $A(\alpha'')$ not equal to $\theta$. If the Fatou chains of $f_{\alpha'}$ converging to $P_{3, \alpha'}$ are of the same types as the Fatou chains of $f_{\alpha''}$ converging to $P_{3, \alpha''}$ respectively, then Proposition \ref{rigidity of combinatorics - curve and arc} and  $A(\alpha')=A(\alpha'')$ implies $\alpha'=\alpha''$. It remains to show that 
if the Fatou chains of $f_{\alpha'}$ converging to $P_{3, \alpha'}$ present different combinatorial patterns from the Fatou chains of $f_{\alpha''}$ converging to $P_{3, \alpha''}$ respectively, then $A(\alpha')\neq A(\alpha'')$. 

Property 1 in the proof of Lemma \ref{convergent Fatou chain} states that there are three different patterns of the Fatou chains converging to the repelling fixed point $P_3$ when $c_0$ is mapped to $c_1$ and each of them has two equivalent representations, which are  considered as the limiting situations of two regular Fatou chains converging to $P_3$ by moving $c_0'$ and $c_0''$ towards $c_0$ clockwise on the boundaries of $U$ and $U_0^0$ on one regular Fatou chain and counterclockwise on the boundaries of $U$ and $U_0^0$ on the other regular Fatou chain (see Figure \ref{degenerate_Fatou_chain}). Pulling into the consideration of the conformal angle $A(\alpha)$, we can see that $A(\alpha)$ approaches $\theta $ from one side in one limiting process and approaches $\theta$ from the other side in the other limiting process. This means that $A(\alpha')\neq A(\alpha'')$. 

In summary, we have shown that the map $\alpha\mapsto e^{iA(\alpha)}$ is injective on $V$.

Case 3: $\alpha \in \Gamma_\theta$ and $A(\alpha)=0$. Through a process similar to the one used to handle Case 2, we obtain a neighborhood $V$ of $\alpha$ on $\Gamma_\theta$ on which the map $\alpha\mapsto e^{iA(\alpha)}$ is injective.
\end{proof}

Now we prove that $\Gamma_\theta$ is a Jordan curve.
\begin{proof}[Proof of Part (b) of Main Theorem] The proof is divided into the following three steps.

\medskip
{\bf Step 1.} We show that $\Gamma_\theta$ contains a connected component $\Gamma_\theta'$ separating $0$ from $\infty$.
 
Denote by  
\begin{equation*}
\tilde{c}_0=\frac{(\frac{\lambda}{\alpha}+2\alpha+3)- \sqrt{(\frac{\lambda}{\alpha}+\alpha+3)^2-\lambda}}{3(\frac{\lambda}{\alpha}+\alpha+2)}
\end{equation*}
and
\begin{equation*}
\tilde{c}_1=\frac{(\frac{\lambda}{\alpha}+2\alpha+3)+ \sqrt{(\frac{\lambda}{\alpha}+\alpha+3)^2-\lambda}}{3(\frac{\lambda}{\alpha}+\alpha+2)}.
\end{equation*}

{\bf Sublemma 1:}  If $|\alpha |$ is sufficiently small, then $\tilde{c}_0\in \partial{\Delta_\alpha^0}$ and $\tilde{c}_1\notin  \partial{\Delta_\alpha^0}\cup \partial{\Delta_\alpha^1}$.  

Clearly, when $\alpha \to 0$, 
\begin{equation*}
    \tilde{c}_0\to\frac{\frac{\lambda}{\alpha}-\frac{\lambda}{\alpha}}{3\frac{\lambda}{\alpha}}=0 \quad \text{and} \quad\tilde{c}_1\to\frac{\frac{\lambda}{\alpha}+\frac{\lambda}{\alpha}}{3\frac{\lambda}{\alpha}}=\frac{2}{3}.
\end{equation*}\\
Furthermore, $\tilde{c}_0=O(\epsilon^2)$ when $|\alpha|=\epsilon$ is sufficiently small. 

We first show that $\tilde{c}_0\in \partial{\Delta_\alpha^0}$ when $|\alpha |$ is sufficiently small. Suppose that $\tilde{c}_0\notin \partial{\Delta_\alpha^0}$. Then $\tilde{c}_0$ is outside $\overline{\Delta_\alpha^0}$. 
By Corollary \ref{disjoint} and the expression (\ref{equ:Leh}), $\Delta_\alpha^0$ is contained in the disk $\{z:|z|<c(k_0)|\tilde{c}_0|=O(\epsilon^2)\}.$
Since $$|f(z)|=|
(\frac{\lambda}{\epsilon}+\epsilon+2)z^3-(\frac{\lambda}{\epsilon}+2\epsilon+3)z^2+\epsilon z+1|<1+O(\epsilon^3)$$ for $|z|<O(\epsilon^2)$,
it follows that $\Delta_\alpha^1=f(\Delta_\alpha^0)$ is contained in a small disk $\{z:|z-1|<O(\epsilon^3)\}$. On the other hand, $\tilde{c}_1$ is very close to $\frac{2}{3}$ and $\tilde{c}_0$ is very close to $0$. Therefore, $\tilde{c}_1\notin  \partial{\Delta_\alpha^0}\cup \partial{\Delta_\alpha^1}$ and $\tilde{c}_0\notin \partial{\Delta_\alpha^1}$.
Therefore, both  $\tilde{c}_1$ and $\tilde{c}_0$ don't belong to 
$\partial{\Delta_\alpha^0}\cup \partial{\Delta_\alpha^1}$. 
This is a contradiction to Theorem \ref{zhang11}. 
Thus, $\tilde{c}_0\in \partial{\Delta_\alpha^0}$ when $|\alpha |$ is sufficiently small. 

It remains to show even if $\tilde{c}_0\in \partial{\Delta_\alpha^0}$, $\tilde{c}_1\notin  \partial{\Delta_\alpha^0}\cup \partial{\Delta_\alpha^1}$. 

From $\tilde{c}_0=O(\epsilon^2)$, $\tilde{c}_0\in \partial{\Delta_\alpha^0}$, Corollary \ref{disjoint} and the expression (\ref{equ:Leh}), we obtain 
 $\Delta_\alpha^0$ is contained in the disk $$\{z:|z|<c'(k_0)|\tilde{c}_0|=O(\epsilon^2)\}.$$  It follow that 
 $\Delta_\alpha^1=f(\Delta_\alpha^0)$ is contained in the disk $\{z:|z-1|<O(\epsilon^3)\}$. Thus, $\tilde{c}_1\notin  \partial{\Delta_\alpha^0}\cup \partial{\Delta_\alpha^1}$.

\medskip
{\bf Sublemma 2:} If $|\alpha|$ is sufficiently large, then $\tilde{c}_1\in \partial{\Delta_\alpha^1}$ and $\tilde{c}_0\notin  \partial{\Delta_\alpha^0}\cup \partial{\Delta_\alpha^1}$.

Lemma \ref{lem:sym} states that $f_{\tilde{\alpha}}$ is conjugate to $f_\alpha$ by the map $\tau(z)=1-z$, where $\tilde{\alpha}=\lambda/\alpha$. Since $\tau $ switch the centers of the Siegel disks on the $2$-cycle and maps a critical point of $f_\alpha$ on $\partial{\Delta_\alpha^0}$ to a critical point of $f_{\tilde{\alpha}}$ on $\partial{\Delta_{\tilde{\alpha}}^1}$, {\bf Sublemma 2} follows from {\bf Sublemma 1} through this conjugation between $f_{\tilde{\alpha}}$ and $f_\alpha$.

\medskip
From the above {\bf Sublemmas} {\bf 1} and {\bf 2}, $\Gamma_\theta$ is bounded away from $0$ and bounded away from $\infty $. Proposition \ref{continuity} implies that $\Gamma_\theta$ is closed. Thus, $\Gamma_\theta$ is compact.  

As we have pointed out in Section \ref{critical points}, there are different choices for a maximal domain of complex analyticity for $\tilde{c}_0$ and $\tilde{c}_1$. In this proof, we use the maximal domain obtained from $\EC$ by removing $\alpha_0$ and $\tilde{\alpha}_0$, two disjoint arcs $\beta_4$ and $\beta_5$ respectively connecting $\alpha_4$ and $\alpha_5$ to $\infty$, and two disjoint arcs $\tilde{\beta}_4$ and $\tilde{\beta}_5$ respectively connecting $\tilde{\alpha}_4$ and $\tilde{\alpha}_5$ to $0$, where $\alpha_0$, $\tilde{\alpha}_0$, $\alpha_4$, $\alpha_5$, $\tilde{\alpha}_4$ and $\tilde{\alpha}_5$ are given in Section \ref{critical points}. Furthermore, we require that the four arcs $\beta_4$, $\beta_5$, $\tilde{\beta}_4$ and $\tilde{\beta}_5$ are disjoint from $\Gamma_\theta$. 

Let $\alpha_a\in \Sigma_{\theta}\setminus (\beta_4\cup\beta_5\cup \tilde{\beta}_4\cup\tilde{\beta}_5)$ and satisfy the condition of {\bf Sublemma 1} and let $\alpha_b\in \Sigma_{\theta}\setminus (\beta_4\cup\beta_5\cup \tilde{\beta}_4\cup\tilde{\beta}_5)$ and satisfy the condition of {\bf Sublemma 2}. Assume that $\gamma:[0,1]\to \Sigma_{\theta}\setminus (\beta_4\cup\beta_5\cup \tilde{\beta}_4\cup\tilde{\beta}_5)$ is a continuous path in $\Sigma_{\theta}\setminus (\beta_4\cup\beta_5\cup \tilde{\beta}_4\cup\tilde{\beta}_5)$ such that $\gamma(0)=\alpha_a$ and $\gamma(1)=\alpha_b$. Define 
\begin{equation*}
    t_0=\sup\{0<t<1|\ \tilde{c}_0\in \partial{\Delta_{\gamma(t)}^0}\ \text{and}\ \tilde{c}_1\notin \partial{\Delta_{\gamma(t)}^1}\}.
\end{equation*}
Let $\alpha'=\gamma(t_0)$. By Proposition \ref{continuity}, we know that 
\begin{center}
    $\partial\Delta_{\gamma(t)}^0\to\partial\Delta_{\alpha'}^0$ and 
    $\partial\Delta_{\gamma(t)}^1\to\partial\Delta_{\alpha'}^1$
\end{center}
with respect to the Hausdorff metric as $t\to t_0$. By the difinition of $t_0$, there is a sequence $t_n\to t_0^-$ such that $\tilde{c}_0(\gamma(t_n))\in \partial{\Delta_{\gamma(t_n)}^0}$ for every $n\ge1$. Thus, $\tilde{c}_0(\alpha')\in \partial{\Delta_{\alpha'}^0}$. If $\tilde{c}_1(\alpha')\notin \partial{\Delta_{\alpha'}^1}$, by Proposition \ref{continuity}, there exists a small neighborhood $W$ of $\alpha'$ such that $\tilde{c}_1(\alpha)\notin  \partial{\Delta_\alpha^0}\cup \partial{\Delta_\alpha^1}$ for any $\alpha\in W$. This implies $\tilde{c}_0(\alpha)\in \partial{\Delta_\alpha^0}$ and $\tilde{c}_1(\alpha)\notin \partial{\Delta_\alpha^1}$ for all $\alpha\in W$. This contradicts to the definition of $\alpha'$. Therefore, $\tilde{c}_1(\alpha')\in \partial{\Delta_{\alpha'}^1}$ and hence $\alpha'\in \Gamma_{\theta}$.

Using the existence of $\alpha '$ on any continuous curve $\gamma$ in $\Sigma_{\theta}\setminus (\beta_4\cup\beta_5\cup \tilde{\beta}_4\cup\tilde{\beta}_5)$, we can see that $\Gamma_\theta$ separates $0$ from $\infty$ and hence $\Gamma_\theta$ contains a connect component $\Gamma_\theta'$ separating $0$ from $\infty $. 

\medskip
{\bf Step 2.} We show that $\Gamma_\theta'$ is a Jordan curve containing three parameters $\alpha $ such that $f_\alpha(c_{0, \alpha})=c_{1, \alpha}$ and three parameters $\alpha $
such that $f_\alpha(c_{1, \alpha})=c_{0, \alpha}$, which we denote by $\alpha_1$, $\alpha_2$ and $\alpha_3$, and $\tilde{\alpha}_1$, $\tilde{\alpha}_2$ and $\tilde{\alpha}_3$, where $c_{0, \alpha}$ denotes the critical point of $f_\alpha $ on $\partial{\Delta_{\alpha}^0}$ and $c_{1, \alpha}$ denotes the critical point of $f_\alpha $ on $\partial{\Delta_{\alpha}^1}$. Using Lemma \ref{lem:sym}, we let $\tilde{\alpha}_k=\lambda/\alpha_k$, where $k=1, 2, 3$.

By Lemma \ref{c-G} and Lemma \ref{injection}, we know that $\Gamma_\theta'$ is Jordan arc locally. Suppose that $\Gamma_\theta'$ has a self intersection at some point $\alpha$. Let $U$ be a neighborhood of $\alpha $ on $\Gamma_\theta$ on which the map $\alpha \mapsto e^{iA(\alpha)}$ is injective. Using Lemma \ref{c-G}, we can find two different parameters $\alpha'$ and $\alpha''$ on $U$ such that $A(\alpha')=A(\alpha '')$. This is a contradiction to the injectivity of the map on $U$. Thus, $\Gamma_\theta'$ has no self intersection. Therefore, $\Gamma_\theta'$ is a Jordan curve. Then the map $\alpha \mapsto e^{iA(\alpha)}$ maps $\Gamma_\theta'$ onto the unit circle $\T$. It follows that one of $\alpha_1$, $\alpha_2$ and $\alpha_3$ lies on $
\Gamma_\theta'$. Without loss of generality, we may assume that it is the point $\alpha_1$ marked on Figure \ref{three_special_parameters_in_jordan_curve_case}. Part (c) of Remark \ref{rmk:global loop} tells us that when $\alpha\in \Gamma_\theta'$ moves from one side of $\alpha_1$ to the other side of $\alpha_1$ in the counterclockwise direction, the combinatorial patterns of the Fatou chains converging to $P_{3, \alpha}$ under $f_\alpha$ change, with the left Fatou chains colored by red and purple respectively on Figure \ref{three_special_parameters_in_jordan_curve_case}. The color of the left convergent Fatou chain remains to be red for the parameter $\alpha$  on $\Gamma'_\theta$ moving away from $\alpha_1$ in the clockwise direction until it reaches a value such that $f_\alpha(c_{1, \alpha})=c_{0, \alpha}$, which we denote by $\tilde{\alpha}_1$ on Figure \ref{three_special_parameters_in_jordan_curve_case}, and remains to be purple for the parameter $\alpha$ on $\Gamma'_\theta$ moving away from $
\alpha_1$ in the counterclockwise direction until it reaches another value such that $f_\alpha(c_{0, \alpha})=c_{1, \alpha}$, which we denote by $\alpha_3$ on Figure \ref{three_special_parameters_in_jordan_curve_case}. Then the left Fatou chain converging to $P_{3, \alpha}$ under $f_\alpha$ has the color back to red and remain unchanged until $\alpha $ reaches a value (different from $\tilde{\alpha}_1$) such that $f_\alpha(c_{1, \alpha})=c_{0, \alpha}$, which we denote by $\tilde{\alpha}_3$ on Figure \ref{three_special_parameters_in_jordan_curve_case}. Finally, the left Fatou chain converging to $P_{3, \alpha}$ under $f_\alpha$ keeps the same color (green) while $\alpha$ moves along $\Gamma_\theta'$ counterclockwise until $\tilde{\alpha}_1$. In more detail, $A(\alpha)$ changes from $0$ to $\theta$ when $\alpha$ moves along $\Gamma_\theta'$ counterclockwise from $\tilde{\alpha}_1$ to $\alpha_1$; $A(\alpha)$ is viewed as changing from $\theta$ to $\theta+2\pi$ as $\alpha$ goes on moving along  $\Gamma_\theta'$ counterclockwise until $\alpha_3$; $A(\alpha)$ is viewed as changing from $\theta+2\pi$ to $4\pi$ as $\alpha$ moves to $\tilde{\alpha}_3$; $A(\alpha)$ is viewed as changing from $4\pi$ to $6\pi$ as $\alpha$ moves back to $\tilde{\alpha}_1$. Therefore, there is $\alpha\in \Gamma_\theta'$ on the portion of $\Gamma_\theta'$ from $\tilde{\alpha}_3$ to $\tilde{\alpha}_1$ in the counterclockwise direction such that $f_\alpha(c_{0, \alpha})=c_{1, \alpha}$ (i.e., $A(\alpha)=4\pi+\theta$), which we denote by $\alpha_2$ on Figure \ref{three_special_parameters_in_jordan_curve_case}; there is $\alpha\in \Gamma_\theta'$ on the portion of $\Gamma_\theta'$ from $\alpha_1$ to $\alpha_3$ in the counterclockwise direction such that $f_\alpha(c_{1, \alpha})=c_{0, \alpha}$ (i.e., $A(\alpha)=2\pi$),  which we denote by $\tilde{\alpha}_2$ on Figure \ref{three_special_parameters_in_jordan_curve_case}.
In summary, we obtain:

{\bf Conclusion 1.} There are exactly six parameters $\alpha$ on $\Gamma_\theta'$ such that one critical point is mapped to another critical point, which are distributed on $\Gamma_\theta'$ as shown on Figure \ref{three_special_parameters_in_jordan_curve_case}. 
The map $\alpha \mapsto e^{iA(\alpha)/3}$ introduces a global parametrization of $\Gamma_\theta'$. 

Then the work of Section \ref{Labels of eventually Siegel comps - curve} and the above {\bf Conclusion 1} imply:

{\bf Conclusion 2.} $\Gamma_\theta'$ is complete in the sense that if a combinatorial pattern of the eventually Siegel-disk components is realized by a parameter on $\Gamma_\theta$, then it can be realized by a parameter on $\Gamma_\theta'$.

\medskip
{\bf Step 3.} We show that $\Gamma_\theta=\Gamma_\theta'$.

This is a consequence of the completeness of $\Gamma_\theta'$ given in the above {\bf Conclusion 2} and the rigidity result provided by Proposition \ref{rigidity of combinatorics - curve and arc}. Thus, $\Gamma_\theta$ is a Jordan curve. 
\end{proof}

\subsection{Proofs of $\Gamma_\theta^0$ and $\Gamma_\theta^1$ to be Jordan arcs}\label{jordanarc}
Because of the symmetry between $\Gamma_\theta^0$ and $\Gamma_\theta^1$ in the sense that $f_\alpha$ and $f_{\lambda/\alpha}$ are conjugated by the map $z\mapsto 1-z$, it suffices to prove that $\Gamma_\theta^1$ is a Jordan arc. 

Recall in \eqref{conformal angle arc-1}, for each $\alpha \in \Gamma_\theta^1$, 
the conformal angle between $c_0$ and $c_1$ is defined by 
\begin{equation*}
   A(\alpha):=\arg\,H_{\alpha}^1 (c_1),
\end{equation*} 
which is measured in the counterclockwise direction, where $H_{\alpha}^1$ and $H_{\alpha}^0$ are normalized so that $H_{\alpha}^1(c_0)=1$ and $H_{\alpha}^0(f_{\alpha}(c_0))=1$.

Using the same method to prove Lemma \ref{c-G}, we obtain:
\begin{lem}\label{c-Arc}
    The map $\alpha\mapsto A(\alpha)$ is continuous on  $\Gamma_\theta^1$.
\end{lem}

In Section \ref{critical points}, we have shown that there are four parameter values $\alpha$ such that $f_\alpha$ has a finite double critical point, which we denote by 
$\alpha_4$, $\alpha_5$, $\tilde{\alpha}_4$, and $\tilde{\alpha}_5$ with $\tilde{\alpha}_4=\lambda/\alpha_4$ and $\tilde{\alpha}_5=\lambda/\alpha_5$. Using the Jordan curve property of $\Gamma_\theta$ and {\bf Sublemma 1} in the proof of this property (there is a critical point on $\partial \Delta_\alpha^0$ when $|\alpha|$ is sufficiently small), we know that for any $\alpha\in \Sigma_\theta$ inside $\Gamma_\theta$, there is at least one critical point on $\partial \Delta_\alpha^0$. Therefore, $\Gamma_\theta^1$ is outside $\Gamma_\theta$. Using Lemma \ref{lem:sym}, we know that two of $\alpha_4$, $\alpha_5$, $\tilde{\alpha}_4$, and $\tilde{\alpha}_5$ are outside $\Gamma_\theta$, which we denote by $\alpha_4$ and $\alpha_5$, and two of them are inside $\Gamma_\theta$, which are $\tilde{\alpha}_4$, and $\tilde{\alpha}_5$. Thus, $\alpha_4, \alpha_5\in \Gamma_\theta^1$ and $\tilde{\alpha}_4, \tilde{\alpha}_5\in \Gamma_\theta^0$.
\begin{lem}\label{injection-arc}
The map $\alpha\mapsto A(\alpha)$ is locally injective on $\Gamma_\theta^1\setminus \{\alpha_4, \alpha_5\}$.
\end{lem}
\begin{proof}
For all $\alpha \in \Gamma_\theta^1$, the combinatorial patterns of the eventaully Siegel-disk Fatou components of $f_\alpha$ have been described in Section \ref{Labels of E Siegel comps - arc}.  
Using the same method to prove Lemma \ref{local invariance of color for the Fatou chain to P-3}, we can obtain that if $A(\alpha)\neq 0$, then the color of the Fatou chain converging to $P_{3,\alpha}$ is locally preserved. Therefore, if $A({\alpha})=A({\alpha'})\neq 0$ and $\alpha$ and $\alpha'$ are close to each other enough, then the eventually Siegel-disk components of $f_\alpha$ has the same combinatorial pattern as the eventually Siegel-disk components of $f_{\alpha'}$. By Proposition \ref{rigidity of combinatorics - curve and arc}, we know that $\alpha=\alpha'$. Note that $\alpha_4$ and $\alpha_5$ are the only two elements $\alpha$ of $\Gamma_\theta^1$ such that $A(\alpha)=0$. 
Thus, for any $\alpha \in \Gamma_\theta^1\setminus \{\alpha_4, \alpha_5\}$, there is a neighborhood $U$ of $\alpha$ on $\Gamma_\theta^1\setminus \{\alpha_4, \alpha_5\}$ on which the map $\alpha\mapsto A(\alpha)$ is injective. 
\end{proof}

Now we prove that $\Gamma_\theta^1$ is a Jordan arc. 
\begin{proof}[Proof of Part (d) of Main Theorem] We divide the proof into the following three steps. 

\medskip
{\bf Step 1.} We show that $\alpha_4$ and $\alpha_5$ stay on the same connected component of $\Gamma_\theta^1$.

By the expressions \eqref{critical pts} of the (finite) critical points in terms of $\alpha$, we know that 
$\alpha_4$, $\alpha_5$, $\tilde{\alpha}_4$ and $\tilde{\alpha}_5$ are essential singularities of \eqref{critical pts}. Recall that we denote the roots of $\frac{\lambda}{\alpha}+\alpha+2=0$ by $\alpha_0$ and $\tilde{\alpha}_0$ with $\tilde{\alpha}_0=\lambda/\alpha_0$. Then $\Sigma_\theta=\EC \setminus \{0, \infty, \alpha_0,\tilde{\alpha}_0\}$.  Consider each finite critical point as a complex analytic function of $\alpha$. In this proof, we use the maximal domain of analyticity that is obtained from $\Sigma_\theta$ by removing one Jordan arc $\beta$ outside $\Gamma_\theta$ connecting $\alpha_4$ and $\alpha_5$, and one Jordan arc $\tilde{\beta}$ inside $\Gamma_\theta$ connecting $\tilde{\alpha}_4$ and $\tilde{\alpha}_5$, and we denote this maximal domain by $\Omega$. We also know that when $|\alpha|$ is sufficiently large, there is a critical point on $\partial\Delta_\theta^1$, which we denote by $c_1$. Using the analytic germ of $c_1(\alpha)$ at such a point near $\infty$ and analytic continuation, we obtain the value of $c_1(\alpha)$ for each $\alpha\in \Omega$. Now we focus on using this complex analytic function $c_1(\alpha)$ on $\Omega$.

In order to deduce that $\alpha_4$ and $\alpha_5$ stay on the same connected component of $\Gamma_
\theta^1$, it is sufficient to prove that for any Jordan curve $\gamma$ on $\Sigma_\theta$ and outside $\Gamma_\theta$, if it contains a point $\alpha_\infty$ sufficiently close $\infty $, separates $\alpha_4$ from $\alpha_5$, and intersects $\beta$ only once, then there exists a point $\alpha_*$ on $\gamma$ such that $\alpha_*\in\Gamma_\theta^1$.
\begin{figure}[htbp]
  \setlength{\unitlength}{1mm}
\includegraphics[width=1\textwidth]{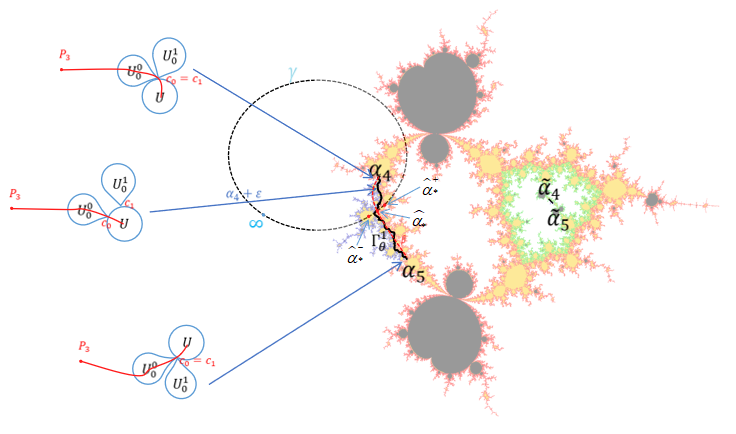}
  \caption{An illustration for the proof of $\hat{\alpha}_*\in \Gamma_\theta^1$ and three different positions of $U=\Delta_{\alpha}^1$, $U_0^0$ and $U_0^1$ for $\alpha_4, \alpha_4+\epsilon$ and $\alpha_5\in \Gamma_\theta^1$ respectively.}
  \label{arc_0^1}
\end{figure} 
Denote by $\hat{\alpha}_*$ the intersection point between $\gamma$ and $\beta$. Assume that there is no point $\alpha$ on $\gamma\setminus \{\hat{\alpha}_*\}$ such that $\alpha\in \Gamma_\theta^1$. We show $\hat{\alpha}_*\in \Gamma_\theta^1$. Denote by $\hat{\alpha}_*^-$ and $\hat{\alpha}_*^+$ two points on $\gamma$ sufficiently close to $\hat{\alpha}_*$ and on different sides of $\beta$ (see Figure \ref{arc_0^1}). Although $c_1(\alpha)$ is a single-valued analytic function globally defined on $\Omega$, the values of $c_1(\hat{\alpha}_*^-)$ and $c_1(\hat{\alpha}_*^+)$ go to different limiting values as $\hat{\alpha}_*^-$ and $\hat{\alpha}_*^+$ approach $\hat{\alpha}_*$ from the different sides of $\beta$, which we denote by $\tilde{c}_1^-$ and $\tilde{c}_1^+$. By Proposition \ref{continuity}, $\tilde{c}_1^-$ and $\tilde{c}_1^+$ lie on the boundary of $\Delta_{\hat{\alpha}_*}^1$. Thus, $\hat{\alpha}_*\in \Gamma_\theta^1$. Therefore, $\gamma$ contains a point $\alpha_*$ on $\Gamma_\theta^1$. It follows from this property that $\alpha_4$ and $\alpha_5$ lie on the same connected component of $\Gamma_\theta^1$, which we denote by $\Gamma_\theta^1(\alpha_4, \alpha_5)$.

\medskip
{\bf Step 2.} We show that $\Gamma_\theta^1(\alpha_4, \alpha_5)$ contains a Jordan arc with $\alpha_4$ and $\alpha_5$ as two endpoints, which we denote by $S(\alpha_4, \alpha_5)$. 

Let $\alpha \in \Gamma_\theta^1(\alpha_4, \alpha_5)\setminus \{\alpha_4, \alpha_5\}$. We denote by $c_{0, \alpha}$ the critical point on the Fatou chain $\Pi_r(\alpha) $ of $f_\alpha$ converging to the separating repelling fixed point $P_{3, \alpha}$. Following the notation introduced in Section \ref{Labels of E Siegel comps - arc}, $\Pi_r(\alpha)$ is colored by red. Then we denote by $c_{1, \alpha}$ the other critical point on $\partial\Delta_\alpha^1$. Applying Lemma \ref{c-Arc} and Lemma \ref{injection-arc}, we know that $\Gamma_\theta^1(\alpha_4, \alpha_5)\setminus \{\alpha_4, \alpha_5\}$ contains a Jordan arc $S$ with $\alpha$ in its interior. Let $\alpha_l$ and $\alpha_r$ be the two endpoints of $S$. Since each of $\alpha_l$ and $\alpha_r$ is not equal to $\alpha_4$ or $\alpha_5$, it follows that $\Gamma_\theta^1(\alpha_4, \alpha_5)\setminus \{\alpha_4, \alpha_5\}$ contains two Jordan arcs $S_l$ and $S_r$ with $\alpha_l$ and $\alpha_r$ in their interiors respectively. Then $S_l\cup S\cup S_r$ is a Jordan arc in $\Gamma_\theta^1(\alpha_4, \alpha_5)\setminus \{\alpha_4, \alpha_5\}$ containing $\alpha$ in its interior and larger than $S$. Repeating this process, we obtain a maximal Jordan arc in $\Gamma_\theta^1(\alpha_4, \alpha_5)\setminus \{\alpha_4, \alpha_5\}$ containing $\alpha$ in its interior, which we denote by $S_{\alpha}$. Let us denote the two endpoints of $S_{\alpha}$ by $\tilde{\alpha}_l$ and $\tilde{\alpha}_r$, and let $\{\alpha_n^l\}$ be a sequence of points on $S_\alpha$ converging to $\tilde{\alpha}_l$ and $\{\alpha_n^r\}$ be a sequence of points on $S_\alpha$ converging to $\tilde{\alpha}_r$. Then $A(\alpha_n^l)$ converges to $0$ and $A(\alpha_n^l)$ converges to $2\pi$ or vise versa as $n\rightarrow \infty$. Thus, $A(\tilde{\alpha}_l)=0$ and $A(\tilde{\alpha}_r)=2\pi$ or vise versa. In the meantime, the relative combinatorial pattern between the convegent Fatou chain $\Pi_r(\tilde{\alpha}_l)$ and three Fatou components $U$, $U_0^0$ and $U_0^1$ is different from the relative combinatorial pattern between $\Pi_r(\tilde{\alpha}_r)$ and $U$, $U_0^0$ and $U_0^1$ (these two relative patterns are illustrated as the two left drawings on Figure \ref{fatou_chain_for_arc}). Therefore, the combinatorial patterns of the eventually Siegel-disk components of $f_{\tilde{\alpha}_l}$ and $f_{\tilde{\alpha}_r}$ are different. Thus, $\tilde{\alpha}_l\neq \tilde{\alpha}_r$. Hence, one of $\tilde{\alpha}_l$ and $\tilde{\alpha}_r$ is equal to $\alpha_4$ and the other is equal to $\alpha_5$. So far, we have shown that $\Gamma_\theta^1(\alpha_4, \alpha_5)$ contains a Jordan arc  connecting $\alpha_4$ and $\alpha_5$, which we denote by $S(\alpha_4, \alpha_5)$.

\medskip
{\bf Step 3.} We prove that $\Gamma_\theta^1=S(\alpha_4, \alpha_5)$.

Let $\alpha'\in \Gamma_\theta^1\setminus \{\alpha_4, \alpha_5\}$. We denote by $c_{0, \alpha'}$ the critical point on the Fatou chain $\Pi_r(\alpha') $ of $f_{\alpha'}$ converging to the separating repelling fixed point $P_{3, \alpha'}$. Following the notation introduced in Section \ref{Labels of E Siegel comps - arc}, $\Pi_r(\alpha')$ is colored by red.
Since $0<A(\alpha')<2\pi$ and $A(\alpha)$ varies continuously between $0$ and $2\pi$ when $\alpha \in S(\alpha_4, \alpha_5)\setminus \{\alpha_4, \alpha_5\}$. Thus, there exists $\alpha''\in S(\alpha_4, \alpha_5)\setminus \{\alpha_4, \alpha_5\}$ such that $A(\alpha'')=A(\alpha')$. We also denote by $c_{0, \alpha''}$ the critical point on the Fatou chain $\Pi_r(\alpha'')$ of $f_{\alpha''}$ converging to the separating repelling fixed point $P_{3, \alpha''}$. Then $\Pi_r(\alpha'')$ is also colored by red. Since $0<A(\alpha'')=A(\alpha')<2\pi$, it follows from (2) of Corollary \ref{rigidity of conformal angle and convergent Fatou chains} that the eventually Siegel-disk components of $f_{\alpha'}$ and $f_{\alpha''}$ have the same combinatorial pattern. By Proposition \ref{rigidity of combinatorics - curve and arc}, we conclude that $\alpha'=\alpha ''$. Thus, $\alpha'\in S(\alpha_4, \alpha_5)\setminus \{\alpha_4, \alpha_5\}$. Therefore, $\Gamma_\theta^1=S(\alpha_4, \alpha_5)$.
\end{proof}

\bibliographystyle{amsplain}

\end{document}